\documentclass[11pt, aos, preprint]{imsart}
\usepackage[english]{babel}
\usepackage{parskip}
\usepackage{amsmath,amsfonts,amsthm, amssymb}
\usepackage{graphicx} 
\usepackage{enumitem} 
\usepackage{xcolor, latexsym} 
\usepackage{subfigure} 
\usepackage[pdftex, bookmarks=true, breaklinks, hypertexnames=false, pdfborder={0 0 0 0}]{hyperref} 
\usepackage[normalem]{ulem}

\startlocaldefs
\newcommand{\E}{\mathbb{E}}
\newcommand{\mZ}{\mathbb{Z}}
\newcommand{\mT}{\mathbb{T}}

\newcommand{\1}{\textbf{1}}
\def\1{1\!{\rm l}}

\newcommand{\la}{\lambda}
\newcommand{\La}{\Lambda}

\newcommand{\eif}{\psi_{b}}

\newcommand{\EM}{\ensuremath}

\newcommand{\al}{\alpha}
\newcommand{\be}{\beta}

\newcommand{\ga}{\gamma}

\newcommand{\te}{\theta}
\newcommand{\ta}{\tau}
\newcommand{\veps}{\varepsilon}
\newcommand{\vphi}{\varphi} 

\newcommand{\leqa}{\lesssim}
\newcommand{\geqa}{\gtrsim}

\newcommand{\cA}{\EM{\mathcal{A}}}
\newcommand{\cB}{\EM{\mathcal{B}}}
\newcommand{\cC}{\EM{\mathcal{C}}}

\newcommand{\cF}{\EM{\mathcal{F}}}
\newcommand{\cG}{\EM{\mathcal{G}}}
\newcommand{\cH}{\EM{\mathcal{H}}}
\newcommand{\cI}{\EM{\mathcal{I}}}
\newcommand{\cJ}{\EM{\mathcal{J}}}

\newcommand{\cL}{\EM{\mathcal{L}}}
\newcommand{\cM}{\EM{\mathcal{M}}}
\newcommand{\cN}{\EM{\mathcal{N}}}

\newcommand{\cV}{\EM{\mathcal{V}}}

\newcommand{\psg}{{\langle}}
\newcommand{\psd}{{\rangle}}

\definecolor{blendedblue}{rgb}{0.2,0.2,0.7}
\definecolor{darkpurple}{RGB}{49,0,94}
\definecolor{darkgreen}{RGB}{11, 84, 37}

\newcommand{\RR}{\mathbb{R}}

\newcommand{\given}{\,|\,}

\newcommand{\rn}{\sqrt{n}}
\newcommand{\psin}{\psi_{b,L_n}} 
\newcommand{\mh}{\mathbb{H}}

\newcommand{\op}{o_{P_0}(1)}

\allowdisplaybreaks

\theoremstyle{plain} 
 \newtheorem{thm}{Theorem}
 \newtheorem{prop}{Proposition}
 \newtheorem{lem}{Lemma}
 \newtheorem{corollary}{Corollary}

\theoremstyle{definition}

\theoremstyle{remark}
 \newtheorem{remark}{Remark} 
  
\theoremstyle{assumption}

\graphicspath{{./Figures/}} 
\endlocaldefs

\begin{document}

\begin{frontmatter}

\title{Multiscale Bayesian Survival Analysis}

\runtitle{Bayesian Survival Analysis}


\author{\fnms{Isma\"el} \snm{Castillo}\thanksref{t1, m1}\ead[label=e1]{ismael.castillo@upmc.fr}}
\and
\author{\fnms{St\'ephanie} \snm{van der Pas}\thanksref{t2, m2}\corref{}\ead[label=e2]{s.l.vanderpas@amsterdamumc.nl}}
 
\thankstext{t1}{The author gratefully acknowledges support from the Institut Universitaire de France (IUF) and from the ANR grant ANR-17-CE40-0001 (BASICS).}
\thankstext{t2}{This work is (partly) financed by the Dutch Research Council (NWO), under Veni grant 192.087.}
 
\affiliation{Sorbonne University \& IUF\thanksmark{m1} and Amsterdam UMC\thanksmark{m2}}

\address{Sorbonne Universit\'e \&
Institut Universitaire de France\\
  Laboratoire de Probabilit\'es, 
Statistique et Mod\'elisation\\ 4, Place Jussieu, 75252, Paris cedex 05, France\\
  \printead{e1}}

\address{Amsterdam UMC, Vrije Universiteit Amsterdam\\
Department of Epidemiology and Data Science\\
De Boelelaan 1117, 
Amsterdam, 
The Netherlands\\
 \printead{e2}}

\runauthor{Castillo and Van der Pas}

\begin{abstract} We consider Bayesian nonparametric inference in the right-censoring survival model, where modeling is made at the level of the hazard rate. 
We derive posterior limiting distributions for linear functionals of the hazard, and then for `many' functionals simultaneously in appropriate multiscale spaces.  As an application, we derive Bernstein-von Mises theorems for the cumulative hazard and survival functions, which lead to asymptotically efficient confidence bands for these quantities. Further, we show optimal posterior contraction rates  for the hazard in terms of the supremum norm. In medical studies, a popular approach is to model hazards a priori as random histograms with possibly dependent heights. This and more general classes of arbitrarily smooth prior distributions are considered as applications of our theory. A sampler is provided for possibly dependent histogram posteriors. Its  finite sample properties are investigated on both simulated and real data experiments. 
\end{abstract}



\end{frontmatter}

\tableofcontents

\section{Introduction} \label{sec:intro}
Survival models are at the heart of biomedical applications of statistics, and form an integral part of many other fields, including, among others, the insurance sector, sociology and engineering. In medical studies, relevant quantities for inference are survival probability curves, as well as hazard rate functions. Nonparametric methods have proved very helpful in the analysis of these models, in particular Bayesian methods are often used in such settings for their flexibility and ability to measure uncertainty \cite{Ibrahim2013}. %
Histogram priors are especially attractive, because they model the hazard in an intuitively appealing and interpretable way, by splitting the follow-up period into distinct intervals with a potentially different hazard rate during each interval. Indeed, a great variety of histogram priors -- with possibly dependent heights -- has been proposed to model the hazard rate, see e.g. \cite{Sinha1993, Arjas1994, Aslanidou1998, Sahu1997, Gamerman1991, Fahrmeir2001, Ibrahim2013, Berry2004}. A great benefit of the Bayesian approach is that confidence bands for survival curves are created in a natural way, which offers a typically more meaningful way of quantifying uncertainty compared to the confidence intervals that are only valid at a particular point in time often reported in practice.  A prominent point of interest is whether the use of Bayesian credible bands as confidence bands is justified. While such bands are already often used in practice, it is important to have mathematical guarantees that these sets have the desired frequentist coverage. Our results provide theory validating exactly this use of credible bands as frequentist confidence bands for many priors, including histogram priors as a particular case.

The desired mathematical guarantees can be achieved by  proving a   Bernstein--von Mises theorem (BvM), in the form of a (Bayesian) Donsker--type result for the posterior distribution on the survival curve. We will  derive such a result as a consequence of a nonparametric Bernstein--von Mises result on the hazard itself, thereby providing the sought-after theoretical guarantee for uncertainty quantification. The proofs are partly based on multiscale techniques as introduced in \cite{cn13}, \cite{cn14}, \cite{ic14}, as well as on semiparametric tools \cite{Castillo2012}, \cite{Castillo2015}. Also, we provide a sampler for histogram priors (allowing for dependent heights) and illustrate its use both in simulated and real data situations. We observe in simulations that the corresponding credible sets outperform several popular choices for confidence bands in terms of area. We refer to the books \cite{gillbook93, Klein2016}, as well as to the simulation study and data analysis in Sections \ref{sec:simulation} and \ref{sec:data}  for more on existing non--Bayesian methods.

We now briefly review the literature on the frequentist analysis of Bayesian methods for right--censored data. In the context of neutral to the right priors, Hjort \cite{Hjort1990} considered Beta process priors for the cumulative hazard, showing in particular conjugacy. Kim and Lee \cite{kimlee01} derived sufficient conditions for posterior consistency, while  \cite{Kim2004} obtained a BvM for the survival function for classes of neutral to the right priors in the right-censoring model (see also \cite{Kim2006} for results for the Cox model). We note that all such results, that model  the cumulative hazard or the survival function directly, do use some form of conjugacy of the prior and model,  which is not the case for the theory built up in the present paper. A work that also models the hazard rate -- and in this sense is closer to ours -- is De Blasi et al. \cite{DeBlasi2009a}, where the authors model a priori the hazard rate using a kernel mixture with respect to a completely random measure. They derive posterior consistency for the hazard, as well as limiting results for linear and quadratic functionals of the hazard. In \cite{DeBlasi2009b}, the semiparametric BvM is derived in competing risks models. The Cox model is treated as an example of the general semiparametric BvM theorem in  \cite{Castillo2012}, which uses non-conjugate techniques and a Gaussian prior, although requires the hazard to be sufficiently smooth (at least $3/2$--H\"older).

This paper has three main goals. A first aim is to provide theory for posterior distributions modeling both the hazard rate and the survival function, investigating practically used priors, including random histograms and more generally arbitrarily smooth priors. To do so, we focus on the commonly used right-censoring model. We derive both BvM theorems for the survival curve and hazard rate, as well as minimax optimal supremum--norm rate results for the hazard. Supremum--norm results are particularly desirable in practice, as they justify evaluating the quality of an estimated curve through visual closeness to a true curve, rather than through a criterion less easily visualized such as closeness in the  $L_2$--sense. In addition, supremum--norm results may make contributions to other questions like change-point problems, identifications of `cusps' and estimation of level sets (e.g. \cite{lighosal20}), although we do not explore this further in this paper.  Second, the paper is  intended to serve as a platform to derive such results in more complex survival models; in particular, we develop non-conjugate techniques, that do not rely on an explicit expression  of the posterior distribution, and will apply more generally, provided some form of Local Asymptotic Normality (LAN) holds. Third, we try to minimize  regularity assumptions on the hazard rate as much as possible. We will assume that the hazard is $\be$--H\"older,  for an arbitrary $\be>0$ for most of the results (just assuming  $\be>1/2$ for a few examples of priors). The techniques we introduce also enable   improved minimal regularity requirements for posterior supremum norm convergence compared to \cite{ic14}--\cite{cn14}. We come back to the implications of the present work for density estimation in the Discussion in Section \ref{sec:discussion}.

Let us now briefly describe our main results and give an outline of the paper. In Section \ref{sec:intromodel}, we introduce the right-censoring model and recall standard notation and assumptions in this setting. In Section \ref{sec:prior}, we introduce the main families of prior distributions we consider, namely histogram priors on the hazard rate with possibly dependent height coefficients,  and smoother series priors on the log-hazard. The main mathematical results are stated in Section \ref{sec:main}. A BvM theorem for linear functionals of the hazard is first obtained. Next, we state a Donsker--BvM result for the cumulative hazard and survival functions, whose proof is based on a nonparametric BvM theorem (discussed and proved in Section \ref{sec:npbvm}). Then, optimal supremum norm posterior convergence results are derived for the hazard. Finally, we apply these results to the considered classes of priors. In Section \ref{sec:simulation}, we present a sampler for histogram priors and illustrate the estimators' performance through simulated data, while a real data example is considered in Section \ref{sec:data}. A  discussion is presented in Section \ref{sec:discussion}. Section \ref{sec:proofs} contains the proofs of the BvM for linear functionals. Sections \ref{sec:wave}--\ref{sec:suppsim} {contain the statement and proof of the nonparametric BvM theorem}, and gather the proof of the supremum norm results as well as a number of useful lemmas and additional simulation and data application results. For readers mainly interested in practical applications of our methods, we note that Sections \ref{sec:bvmlin} through \ref{sec:sn} can be skipped at first read.\\

\section{The survival model with independent right censoring}\label{sec:intromodel}
We are interested in i.i.d. survival times $T_1, \ldots, T_n$, but their observation is possibly interfered with by i.i.d. censoring times $C_1, \ldots, C_n$, which are independent of the survival times, so that we observe $X=X^n =((Y_1, \delta_1), \ldots, (Y_n, \delta_n))$ i.i.d. pairs, where $Y_i = T_i \wedge C_i$ and $\delta_i = \1\{T_i \leq C_i\}$.

One main object of interest is the survival function $S(t) = P(T > t)$, where we use the generic notation $T$ for a random variable of same distribution as  $T_1$ (and similarly for $C, Y$ below). In the present setting, as in many others, it is useful to assume that $S$ is induced by a certain  hazard function. Assuming $T$ admits a continuous density, which we denote by $f$, the hazard rate is  $\lambda(t) = \lim_{h \downarrow 0} h^{-1}P(t \leq T < t + h \mid T \geq t)$. Integrating the hazard yields the cumulative hazard $\Lambda(\cdot) = \int_0^\cdot \la(u)du$, which is related to the survival by $S(\cdot) = \exp\{-\Lambda(\cdot)\}$.

We assume there exists a `true' continuous hazard $\lambda_0$ underlying the data-generating process and require a few assumptions, mainly to ensure that influence functions are well-defined and to avoid division by zero. For some $\ta>0$ corresponding to the time at end of the study and $c_i>0$, $i=1,\ldots,4$:
\begin{enumerate}[label=\textbf{(M)}]
\item \label{c:mod}  at the end of follow-up, some individuals are still eventfree and uncensored: 
$P(T > \tau) > 0$ and $P(C \geq \tau) = P(C = \tau) > 0$. 

\item[] 
We have $c_1\le \inf_{t\in[0,\ta]} \la_0(t)\le \sup_{t\in[0,\ta]} \la_0(t) \le c_2$.
\notag

\item[] 
The censoring $C$ has a distribution function $G$ and admits a density
\[ p_C(u) = g(u)\1\{0 \leq u < \tau\} + (1 - G(\tau-))\1\{u = \tau\}\]
with respect to $\text{Leb}([0,\ta])+\delta_\ta$, with $\text{Leb}(I)$ the Lebesgue measure on $I$ and $g$ such that 
$c_3\le \inf_{t\in[0,\ta)} g(t)\le \sup_{t\in[0,\ta)} g(t)\le c_4$. \notag
\end{enumerate}
Henceforth, for notational simplicity and without loss of generality we set $\tau=1$, otherwise one can consider rescaled versions of the procedures. It follows from \ref{c:mod} that $\La_0(\ta)=\La_0(1)\le c_2$.

As the censoring distribution factors out from the likelihood (see Section \ref{sec:proofs}) we do not need to model $G$ and we denote the distribution of the data under the `true' unknown parameters simply by $P_{\la_0}$ (or even $P_0$), and keep the notation $G$ for the censoring law, with $g$ its density. We also denote $\bar{G}(u)=1-G(u-)$ for $u\in[0,1]$ and note that in our setting $1-G(t)=\bar{G}(t)$ for $t<1$. 
The function, for $u\in[0,1]$,
\begin{equation} \label{emo}
M_0(u) = E_{\la_0}\1\{u \leq Y\}=\bar{G}(u)e^{-\La_0(u)}
\end{equation}
 plays an important role in the sequel. Under assumption \ref{c:mod} it is bounded away from $0$.

{\em Notation.}  The distribution function of the  variable of interest $T$ is denoted $F(t) = \int_0^t f(x)dx$ and the survival function is $S=1-F$. For a bounded function $b$ on $[0,1]$ and $\La$ a cumulative hazard, we denote 
$(\La b)(\cdot)  =\int_0^\cdot b(u)d\Lambda(u)$ and, in slight abuse of notation $\La b = (\La b)(1)=\int_0^1 b(u)d\Lambda(u)$. For real $a,b$, we set $a\wedge b=\min(a,b)$ and $a\vee b=\max(a,b)$. 

The notation $(\psi_{lk})$ refers to one of the following two wavelet bases,
\begin{enumerate}
\item {\em the Haar basis on $[0,1]$} sets $\vphi=\1_{(0,1]}$, $\psi=\1_{(0,1/2]}-\1_{(1/2,1]}$ and $\psi_{lk}(\cdot)=2^{l/2}\psi(2^l\cdot{-k})$, {$0 \leq k < 2^l, l \geq 0$,}
 and $\psi_{-1\, -1/2}=\vphi$;
\item {\em a smooth boundary-corrected} wavelet basis on $[0,1]$, such as the CDV wavelets of \cite{cdv93}.  
We refer to Section \ref{sec:wave} for more on their properties. 
\end{enumerate}           
We denote dyadic intervals by $I_k^{l} = (k2^{-l}, (k+1)2^{-l}]$ for $l\ge 0$ and $0\le k\le 2^l-1$. For $(\psi_{lk})$ the Haar basis, $I_k^l$ is the support of $\psi_{lk}$.                   

We denote by $L^2=L^2[0,1]$ the space of square--integrable functions on $[0,1]$, with $\psg f,g\psd=\int_0^1 fg$ the associated inner product and $\|f\|^2=\psg f,f \psd$ the squared $L^2$--norm. Also, $L^2(\La)=\{f:\ \int f^2d\La=\int f^2\la <\infty\}$.  
Given a wavelet basis $(\psi_{lk})$ as above, for $f\in L^2$ we denote $f_{lk}=\psg f,\psi_{lk}\psd$ its wavelet coefficients. 
For any $L\ge 0$, we set 
\begin{equation}\label{cve} 
 \cV_L=\text{Vect}\{\vphi,\, \psi_{lk},\ l\le L, 0\le k<2^l\}
\end{equation}
the space generated by wavelet functions up to level $L$.  The space of continuous (resp. bounded) functions on $[0,1]$ is denoted by $\cC[0,1]$ (resp. $L^\infty[0,1]$) and is equipped with the (essential) supremum norm $\|\cdot\|_\infty$. 
For $\beta, D> 0$, and $l$ the largest integer smaller than $\be$, let $\cH(\be,D)=\{f:\ |f^{(l)}(x)-f^{(l)}(y)|\le D|x-y|^{\be-l},\, x,y\in[0,1]\}$.

In the paper asymptotics are as $n\to\infty$, and $o(1), \op$ are respectively a deterministic sequence going to $0$ as $n\to\infty$ and a random one going to $0$ in probability under $P_0=P_{\la_0}$. 

For $(\mathcal{S},d)$ a metric space, the bounded Lipschitz metric $\cB_{\mathcal{S}}$ on probability measures on $\mathcal{S}$ is, for any $\mu,\nu$ probability measures of $\mathcal{S}$,
\begin{equation}
\cB_{\mathcal{S}}(\mu,\nu)=\sup_{F;\|F\|_{BL}\leq1}\left|\int_{\mathcal{S}}F(x)(d\mu(x)-d\nu(x))\right|,
\end{equation}
where $F:\mathcal{S}\to \mathbb{R}$ and
\begin{equation}
\|F\|_{BL}=\sup_{x\in \mathcal{S}}|F(x)|+\sup_{x\neq y}\frac{|F(x)-F(y)|}{d(x,y)}.
\end{equation}

Throughout the paper, the following rate is frequently used, for $\be>0$,
\begin{align}
\veps_{n,\be}^* & = \left( \frac{\log{n}}{n} \right)^{\frac{\be}{2\be+1}}.\label{ratesup}
\end{align}

\section{Prior distributions and glimpse of the results} \label{sec:prior}

In order to model both the cumulative hazard $\La$ (and related survival $S$) and its rate $\la$, we define a prior distribution on $\la$ via a prior $\Pi$ on the log--hazard $r=\log{\la}$. 

\subsection{Families of priors on the log-hazard}\label{sec:familiesofpriors}

Our results are illustrated with two vast families of prior distributions: dyadic histogram priors, allowing for possibly dependent heights, and referred to as {\bf(H)}-type priors in the sequel; and more general possibly smooth priors,   referred to as {\bf(S)}-type priors.
 
In practice, it can be appealing to model the prior distribution directly on certain time-intervals, with a possible dependence in the choice of amplitudes across times. We allow this in the following two prior classes, with respectively independent and dependent heights. 

{\em {$\bf(H_1)$} Regular dyadic histograms with independent coefficients}. Set
\[ r = \sum_{k=0}^{2^{L+1}-1} r_k  \1_{I_{k}^{L+1}}, \ \ \text{or equivalently } \ \la = \sum_{k=0}^{2^{L+1}-1} \la_k  \1_{I_{k}^{L+1}}, \]
for $L =L_n$ a sequence of integers (called `cut-off'), $r_k=\log{\la_k}$,  
and $\la_k$ are positive independent random variables, all to be specified in the sequel. 

{\em $\bf{\boldsymbol{(}H_2}${\bf )} Regular dyadic histograms with dependent coefficients.} These have a structure similar to {$\bf(H_1)$}, but each $\la_k$ depends on $\la_{k-1}$, for $k \geq 1$. Intuitively this dependence induces some `smoothness', while the prior itself remains a histogram.
Histogram priors with a martingale structure have a long history of success \cite{Arjas1994, Aslanidou1998, Fahrmeir2001, Ibrahim2013,  Sahu1997, Sinha1993}. Specific examples will be studied in Section \ref{sec:appli}.

A more general class of priors on $r=\log{\la}$ we consider sets
\begin{align} \label{priorgen}
r & = \sum_{l=-1}^{L_n} \sum_{k=0}^{2^l-1}\sigma_l Z_{lk} \psi_{lk},
\end{align}
where $(\psi_{lk})$ is a wavelet basis as above, the variables $Z_{lk}$ are independent, $L_n$ and $\sigma_l$ are positive real numbers to be chosen. This includes 
\begin{enumerate}
\item[{$\bf(H_3)$}] {\em Dyadic Haar wavelet histograms} with $(\psi_{lk})$  the Haar basis.
\item[{\bf (S)}] {\em Smooth wavelet priors} with $(\psi_{lk})$ a smooth wavelet basis.
\end{enumerate}
Note that the priors ${\bf (H_1), (H_2)}$ can also be written as in  \eqref{priorgen} with $(\psi_{lk})$ the Haar basis. Conversely,  ${\bf (H_3)}$ is a special case of ${\bf (H_2)}$, see Section \ref{sec:wave}.

\subsection{Frequentist analysis of posterior distributions} Given a prior distribution $\Pi$ on log-hazards $r=\log{\la}$ as above and data $X=X^n$ from the right censoring model, one can form the posterior distribution $\Pi[\cdot\given X]$ on $r$, that is the conditional distribution $\cL(r\given X)$, in the usual way. Taking a frequentist approach to analyse the posterior, we assume that there exists a `true'  $r_0=\log\la_0$ so that the data is generated from $X\sim P_{\la_0}$, and we study $\Pi[\cdot\given X]$ in probability under $P_{\la_0}$.

\subsection{A glimpse of the results}

Consider cut-offs $L_n$ defined as, for $\ga>0$,
\begin{align}
2^{L_n}=2^{L_n^U} & \circeq n^{1/2}, \label{ctunder}\\
\text{or}\quad 2^{L_n}=2^{L_n(\ga)} & \circeq \left(\frac{n}{\log{n}}\right)^{\frac{1}{2\ga+1}},\label{ctga}
\end{align}
where $\circeq$ means that one picks a closest integer solution in $L_n$ of the equation. 

Let $G_{\La_0}$ denote the Gaussian process, for $W$ standard Brownian motion, 
\begin{equation} \label{limdonsker}
G_{\La_0}(t) = W(U_0(t)),\qquad t\in[0,1],
\end{equation}
where we have set $U_0(t)=\int_0^t (\la_0/M_0)(u)du$ and $M_0$ is as in \eqref{emo}.  The appearance of the scaling function $M_0$ is a particular feature of the survival model, which appears in multiple results in this paper. It is closely connected to the LAN inner product, which is $\langle f, g \rangle_L = \int_0^1 f g M_0 \la_0$, as described in more detail in Section \ref{sec:lan}.

\begin{thm} \label{thm:glimpse}
Suppose the true log--hazard  $r_0=\log{\la_0}$ belongs to the H\"older ball $\cH(\be,L)$ for $\be,L>0$. 
Let $\Pi$ be a histogram prior  of type ${\bf (H_3)}$ with standard Laplace  coefficients,  and  cut--off $L_n$ as in \eqref{ctunder} or \eqref{ctga} with $\gamma=1/2$. 
Then, regardless of $\be>0$, for $G_{\La_0}$ as in 
\eqref{limdonsker}, and  $\hat\Lambda_n$  Nelson--Aalen's estimator,
\begin{align*}
\cB_{\mathcal{D}[0,1]}\left(\,\cL(\rn(\La-\hat\La_n)\given X)\ ,\, \cL(G_{\La_0})\right) & =\op,
\end{align*}
with $\mathcal{D}[0,1]$ the space of c\`adl\`ag functions on $[0,1]$. 
Now assuming $0<\be\le 1$, for $L_n$ as in \eqref{ctga} with $\gamma=\be$, 
for arbitrary $M_n\to\infty$, and $\veps_{n,\be}^*$ as in \eqref{ratesup},
\begin{align*}
\Pi[ \| \la - \la_0\|_\infty > M_n \veps_{n,\be}^* \given X] &=\op.
\end{align*}
\end{thm}

Theorem \ref{thm:glimpse} first provides a functional result for the posterior distribution of $\La$ for histogram priors which has several consequences for inference, notably providing credible bands with optimal coverage for the true cumulative hazard and survival curve, see Section \ref{sec:donsker}. The hazard itself is modelled and its posterior converges at optimal  rate in the supremum norm (when $\ga=\be$ this is the minimax rate; otherwise the rate is still optimal for the regularity $\ga$ and given by \eqref{rategen} below). This and more general results, including smooth priors on hazards, handling any regularity $\be>0$, and statements for the survival function, are considered below.

\section{Main results} \label{sec:main}

In all what follows, the prior is of the form \eqref{priorgen} with cut-off 
$L_n=L_n(\ga)$ given by, for $\ga>0$,
\begin{equation} \label{eln}
2^{L_n} \circeq \left(\frac{n}{\log{n}}\right)^{\frac{1}{1+2\ga}}.
\end{equation}

As the prior distribution sits on levels $l\le L_n$,  it is helpful to introduce $P_{L_n}$ the orthogonal projection onto $\cV_{L_n}=\text{Vect}\{\psi_{lk}, l\le L_n, 0\le k<2^l\}$,  and $P_{L_n^c}$ the orthogonal projection onto the orthocomplement of $\cV_{L_n}$. 

\subsection{Bernstein--von Mises theorems for linear functionals of $\la$}
\label{sec:bvmlin}
To begin with, we consider estimation of, for $b\in L^2(\La)$, 
\[ \psg b,\la \psd = \int_0^1 b(u)\la(u)du = \int_0^1 b(u)d\La(u).\]
For $P_{L_n}$ the projection onto the first $L_n$ wavelet levels and $M_0$ as in \eqref{emo}, let 
\begin{equation} \label{psiln}
\psi_b=b/M_0,\quad \psi_{b,L_n} = P_{L_n}(b/M_0),
\end{equation}
all well-defined quantities by \ref{c:mod}, which guarantees $M_0(1)>0$. As in \eqref{limdonsker}, the function $M_0$ appears in \eqref{psiln} as a particular feature of the right-censoring model and is closely connected to the LAN norm, as becomes apparent in the proof of Theorem \ref{thm:bvmlin} in Section \ref{sec:proofbvmlin}.

We assume that a convergence rate for $\la$ is available: for a sequence $\varepsilon_n=o(1)$,
\begin{enumerate}[label=\textbf{(P\arabic*)}]
\item \label{c:l1}  
for $A_n = \{\lambda : \|\lambda - \lambda_0\|_1 \leq \varepsilon_n\}$, we have $\ \Pi[A_n \mid X] = 1 + o_{P_0}(1)$,
\end{enumerate}
where we require that $\log{n}=o(\rn\veps_n)$ and $\veps_n=o((\log{n})^{-2})$, which is always the case for nonparametric rates arising in the present setting. 
As shown in Section \ref{sec:suffrates}, such a rate is implied by a Hellinger rate for estimating the data--generating distribution, which itself can be obtained through the posterior convergence rates theory of \cite{Ghosal2000}. Alternatively, for some priors, one could obtain an $L^1$-rate directly using the 
method of \cite{Donnet2017}.

For stating a limiting result on the linear functional $\psg b,\la \psd$, one assumes
\begin{enumerate}[label=\textbf{(B)}]
\item \label{c:funb}   
that $b\in L^\infty[0,1]$ and,
for $\veps_n$ as in \ref{c:l1} and  $\eif$ as in \eqref{psiln}, 
\begin{equation*}
\rn \veps_n \| \eif - \psin\|_\infty =o(1);
\end{equation*}
\end{enumerate}
\begin{enumerate}[label=\textbf{(Q)}]
\item \label{c:cvarb} 
with $r = \log{\lambda}, r_0 = \log\lambda_0, r_t^n = r - t\psi_{b,L_n}/\rn$, 
for $A_n$ as in \ref{c:l1},

\begin{equation*}
\frac{\int_{A_n} e^{\ell_n(r_t^n) - \ell_n(r_0)} d\Pi(r)}
{\int e^{\ell_n(r) - \ell_n(r_0)} d\Pi(r)} = 1+\op.
\end{equation*}
\end{enumerate}
For $a\in\RR$ and $b\in L^2(\La)$, let $L_a$ be the map 
\[ L_a : r\to \rn(\psg e^r , b \psd - a) = \rn(\psg \la,b\psd - a), \] 
so $\Pi[\cdot\given X]\circ L_a^{-1}$ denotes the distribution induced on $\rn(\psg\la,b\psd-a)$ if $r\sim \Pi[\cdot\given X]$. 
\begin{thm}[BvM for linear functionals] \label{thm:bvmlin}
Under conditions \ref{c:mod} on the model, suppose the prior distribution $\Pi$ is such that assumption \ref{c:l1} is satisfied. Let $\hat\vphi_b$ denote any efficient estimator of $\vphi_b=\psg b,\la \psd$, for $b$ a fixed element of  $L^2(\La)$  that together with $\Pi$ satisfies \ref{c:cvarb}.

For such a functional representer $b$, under \ref{c:funb}, as $n\to\infty$,
\[ \cB_\RR\left( \Pi[\cdot\given X]\circ L_{\hat \vphi_b}^{-1}\, ,\, \cN(0, v_b) \right)
=\op, \]
where we denote  $v_b=\Lambda_0(b^2/M_0)=\int_0^1 b^2(u)(\la_0/M_0(u))du$ and  $\cB_\RR$ is the bounded--Lipschitz metric on real distributions.
\end{thm}

Condition  \ref{c:cvarb} can be checked for the priors we consider by change--of--variables techniques, see Section  \ref{sec:changevar}. Condition \ref{c:funb} is a `no--bias' type condition; its form is for simplicity and it can be improved (we note that a simple condition on projections such as (4.12) in \cite{Castillo2015} in the simpler density estimation model is not readily available here though). We refrain from doing so here, and refer the reader to the more general BvM theorem, Theorem \ref{thm:bvm}, which allows to cover a broader class of functionals, including the cumulative hazard, as discussed in the Section \ref{sec:donsker}.

{\em Application: smooth linear functional.}  Suppose $b$ and $r_0$ satisfy a H\"older condition:  $b\in\cH(\mu, D_1)$ and $r_0\in \cH(\be,D_2)$ for some $\mu, \be, D_1 , D_2>0$.  As an example, consider the Laplace prior with independent coefficients and $\sigma_l=1$.  
Then \ref{c:funb} and  \ref{c:cvarb}
 are satisfied if $\mu\wedge 1 + \be\wedge\ga > 1/2 + \ga$. 
For instance if $\ga=1/2$, it is enough that $\be> 1/2$ and $\mu>1/2$; see Section \ref{sec:cfun} for a proof.

{\em Non-linear functionals.} One may obtain results for appropriately smooth non-linear functionals by linearisation. Given a functional $\psi(\la)$, for an appropriate  $a(\cdot)$, one may write
\begin{equation*}
\psi(\la)= \psi(\la_0) + \int_0^1  (\la-\la_0)(u)a(u)du + \rho(\la,\la_0), 
\end{equation*}
where $\rho(\la,\la_0)$ is a remainder term. For example, with a quadratic functional $\psi(\la) = \int_0^1 \la^2(u)du$, one can take $a(u)=2\la_0(u)$ and $\rho(\la,\la_0)=\int_0^1 (\la-\la_0)^2$. Then, provided one can control the remainder term uniformly over a set of high posterior probability, one can state an analogue of Theorem \ref{thm:bvmlin} for the non-linear functional at stake. For instance, tools developed later in the paper enable one to derive a posterior convergence rate for $\|\la-\la_0\|_2$ (this follows e.g. from the stronger $L^\infty$--result in Theorem \ref{thm:sn} below). If $\delta_n$ is the corresponding rate, one can check that controlling the remainder term to adapt the proof from the linear--functional case amounts to asking that $\rn\delta_n^2=o(1)$ holds, which itself can be obtained if $\be>1/2$. We refer to Section 4 of \cite{Castillo2015} for details on implementing this technique and more examples.

\subsection{Bayesian Donsker theorem for cumulative hazard and survival}
\label{sec:donsker}
The standard Donsker theorem in density estimation provides asymptotic normality of the cumulative distribution function in the functional sense. Here we show corresponding analogous Donsker theorems for the cumulative hazard $\La(\cdot)$ and survival $S(\cdot)$. Those in turn imply posterior tightness in the supremum norm for such quantities at rate $1/\rn$ as well as uncertainty quantification in terms of optimal coverage of confidence bands. Other consequences include convergence for suitably regular functionals, in particular the median survival. All these results nearly immediately follow from the nonparametric BvM result in multiscale spaces from  Section \ref{sec:npbvm}, specifically Theorem \ref{thm:bvm}.

We first introduce an appropriate centering for the Bayesian Donsker theorem. The centering we have in mind here comes from the more abstract result from Section \ref{sec:npbvm}. We first define, for a bounded function $g$ on $[0,1]$, 
\begin{equation} \label{wun}
  W_n(g) =W_n(X;g) = \frac{1}{\sqrt{n}} \sum_{i=1}^n[\delta_i g(Y_i) - \La_0g(Y_i)]. 
\end{equation}
We then define  $T_n= T_n(L_n)$ (with $L_n$ as in \eqref{eln}) by the sequence of its wavelet coefficients, for any $k$ and $W_n$ as in \eqref{wun},
\begin{equation} \label{tn}
\psg T_n , \psi_{lk} \psd =
\begin{cases}
\ \psg \la_0 , \psi_{lk} \psd+W_n\left(\psi_{lk}/M_0\right)/\rn \quad & \text{if } l\le L_n\\
\ 0&  \text{if } l>L_n.
\end{cases}
\end{equation}
The centering for the Bayesian Donsker theorem for the cumulative hazard will then be the primitive of $T_n$, that is,  $\mathbb{T}_n(t)=\int_0^t T_n(u)du$.

In combination with \ref{c:l1}, we shall assume the following for the next result.
\begin{enumerate}[label=\textbf{(P\arabic*)}, resume]
\item \label{c:sn}
For some $\zeta_n=o(1)$, we have $\ \Pi[\|\la-\la_0\|_\infty \le \zeta_n \given X] =1 + o_{P_0}(1).$
\end{enumerate}
Let us just mention that $\zeta_n$ just needs to be $o(1)$, and that such supremum norm consistency follows, for instance, from Theorem \ref{thm:sn} below.

Let us require, for suitable directions $b$, and $A_n$ as in \ref{c:l1}:
\begin{enumerate}[label=\textbf{(T)}]
\item \label{c:cvarsn} 
with $r = \log{\lambda}, r_0 = \log\lambda_0, r_t^n = r - \frac{t}{\sqrt{n}}\psi_{b,L_n}$ and $C_1>0$, suppose,  for any $|t|\le \log{n}$, 
\begin{equation*}
\log\frac{\int_{A_n} e^{\ell_n(r_t^n) - \ell_n(r_0)} d\Pi(r)}
{\int e^{\ell_n(r) - \ell_n(r_0)} d\Pi(r)} \le C_1(1+t^2).
\end{equation*}

\end{enumerate} 
Condition \ref{c:cvarsn} can be seen as a non--asymptotic version of \ref{c:cvarb}, and with only a control from above required, and is verified in a similar way. Note that the quantity on the left hand side of \ref{c:cvarsn} is random, but by changing variables it can typically be bounded by a non-random quantity uniformly, as will be seen in the examples. For simplicity, we state here Theorem \ref{thm:donsker} for {\em histogram} priors only: a general statement covering arbitrary priors of type \eqref{priorgen} is Theorem \ref{thm:donskergen}.

\begin{thm}[Donsker's theorem for cumulative hazards] \label{thm:donsker} 
Let $\Pi$ be a histogram prior on hazards, that is, as in \eqref{priorgen} with $(\psi_{lk})$ the Haar basis. Suppose 
  \ref{c:l1}--\ref{c:sn} are satisfied with cut--off $L_n$ and rate $\veps_n$ verifying 

$\rn\veps_n2^{-L_n} = O\left(L_n^{-3}\right)$. 
 
Suppose  \ref{c:cvarb} is satisfied for any  $b\in \cV_{\cL}$ and any fixed  $\cL\ge 0$, and that \ref{c:cvarsn} holds uniformly for $b=\psi_{LK}$ with $L\le L_n,$ $0\le K<2^L$. 

Let $\cL(\La\in \cdot\given X)$  denote the distribution induced on the cumulative hazard $\La$ when  $\la \sim \Pi[\cdot\given X]$. 

Let $G_{\La_0}(t) = W(U_0(t))$ with $W$ Brownian motion and $U_0(t)=\int_0^t (\la_0/M_0)(u)du$. Then, with $\mathbb{T}_n(t)=\int_0^t T_n$,  as $n\to \infty$,
\begin{equation*} 
   \cB_{\cC[0,1]}\left(\,\cL(\rn(\La-\mT_n)\given X)\ ,\, \cL(G_{\La_0})\right) \to^{P_{0}} 0.
\end{equation*}

\end{thm}

Theorem \ref{thm:donsker} on the cumulative hazard is derived as a consequence of   a more general nonparametric BvM theorem (Theorem \ref{thm:bvm}) for the hazard, see Section \ref{sec:npbvm} for a proof. 
 In order to verify that the previous statement leads to efficient estimation, one now derives a result with centering at an efficient estimator, namely the Nelson--Aalen estimator $\hat\La_n$ (\cite{Nelson69}, \cite{Aalen1975}, see \cite{gillbook93} Section IV.1 for an overview). Convergence in distribution for the latter is considered on the space $\mathcal{D}[0,1]$ of c\`adl\`ag functions on $[0,1]$, equipped with the supremum norm (as is usual in this setting, see \cite{gillbook93}, Section II.8 for details).   Replacing the Nelson-Aalen estimator with a smooth approximation would enable a result in $\mathcal{C}[0,1]$.  
\begin{corollary} \label{cor:donsker}
Suppose the prior $\Pi$ and model satisfy the conditions of Theorem \ref{thm:donsker} and suppose $\ga<\be+1/2$. 

Let $\hat\Lambda_n$ be  Nelson-Aalen's estimator. Then for $G_{\La_0}(t) = W(U_0(t))$ as before, and $\Sigma_0(t)=-S_0(t) G_{\La_0}(t)$,
\begin{align*}
\cB_{\mathcal{D}[0,1]}\left(\,\cL(\rn(\La-\hat\La_n)\given X)\ ,\, \cL(G_{\La_0})\right) & \to^{P_{0}} 0,\\
\cB_{\mathcal{D}[0,1]}\left(\,\cL(\rn(S-\hat{S}_n)\given X)\ ,\, \cL(\Sigma_0)\right) & \to^{P_{0}} 0,
\end{align*}
where $\hat{S}_n$ is the survival function corresponding to $\hat\La_n$. 
\end{corollary}

\begin{corollary}[Survival confidence bands] \label{cor:bands}
Under the conditions of Corollary \ref{cor:donsker}, as $n\to\infty$,
\begin{align*} \label{kslim}
 \cB_{\mathbb{R}}\left(\,\cL(\rn\|\La-\hat\La_n\|_\infty \given X)\ ,\ \cL(\|G_{\La_0}\|_\infty)\,\right) & \to^{P_{0}} 0 \\
  \cB_{\mathbb{R}}\left(\,\cL(\rn\|S-\hat{S}_n\|_\infty \given X)\ ,\ \cL(\|\Sigma_{\La_0}\|_\infty)\,\right) & \to^{P_{0}} 0.
\end{align*}
In particular, quantile credible bands for  $\La$ (resp. $S$) at level $1-\al$ are asymptotically confidence bands for $\La$ (resp. $S$) at level $1-\al$.  
\end{corollary}
A quantity particularly useful in medical applications is the median survival time $m=m_{\La}=S^{-1}(1/2)$. Let us consider the Gaussian variable $Z_\infty$, with $m_0=m_{\La_0}$ and recalling $U_0(t)=\int_0^t (\la_0/M_0)(u)du$,
\begin{equation} \label{medlim}
Z_\infty \sim \cN\left(0,\frac{U_0(m_0)}{4f_0^2(m_0)}\right).
\end{equation}
\begin{corollary}[BvM for median survival time] \label{cor:median}
Under the conditions of Corollary \ref{cor:donsker}, let $M_\La=S^{-1}(1/2)$  denote the median survival time and let $\hat{M}_n$ be an efficient estimator thereof. Then, for $Z_\infty$ as in \eqref{medlim}, 
\begin{equation} \label{mediansurv}
 \cB_{\mathbb{R}}\left(\,\cL(\rn(M_\La - \hat{M}_n)\given X)\ ,\ \cL(Z_\infty)\,\right)\to^{P_{0}} 0. 
\end{equation}
\end{corollary}
This result justifies the asymptotic normality for the posterior median survival observed in Section \ref{sec:simulation}, Figure \ref{fig:sim_med_surv}.

\subsection{Supremum--norm convergence rate for $\la$}
\label{sec:sn}

Let us set, for $\be>0, L_n>0$,
\begin{equation} \label{rategen}
 \veps_n^{\be,L_n}=\sqrt{\frac{L_n 2^{L_n}}{n}} + 2^{-\be L_n}.
\end{equation}
In the next statement, as earlier the prior $\Pi$ is as in \eqref{priorgen}, with $L_n$ as in \eqref{eln}.  In statement (a), $\Pi_{L_n}[\cdot\given X]$ denotes the posterior distribution on $\la$ projected  onto the first $L_n$ levels of wavelet coefficients (i.e. setting $\psg \la,\psi_{lk}\psd=0$ for $l>L_n$). In other words, it is the distribution of $P_{L_n}Z$ if $Z\sim\cL(\la\given X)$. The projected posterior is considered, as  controlling the `high frequencies' $\la_{lk}$ for $\la > L_n$ seems technically challenging for very low regularities. While the prior on $r = \log{\lambda}$ is already truncated to $L_n$, the induced posterior on $\lambda$ rather than $r$ may give mass to wavelet coefficients with frequencies above $L_n$ (unless the prior is a dyadic histogram).

\begin{thm}[$\|\,\cdot\,\|_\infty$--contraction for posterior hazard]\label{thm:sn}
Let  $X=(X_1, \dots, X_n)$ be a sample of law $P_0$ with hazard rate $\la_0$ under conditions \ref{c:mod}. For  $\be,D>0$, suppose $\log\la_0\in \cH(\be,D)$.  

Suppose $\Pi$ satisfies \ref{c:l1} with $\veps_n\leqa \veps_n^{\be,L_n}$ defined in \eqref{rategen}, and  that \ref{c:cvarsn} holds uniformly for $b=\psi_{LK}$, for any $L\le L_n$ and $0\le K<2^L$.

(a) Suppose   
$(\psi_{lk})$ is the Haar or a CDV wavelet basis with high enough regularity.   Then, for $\Pi_{L_n}\left[\cdot\given X\right]$ the projected posterior onto $\cV_{L_n}$, for arbitrary $M_n\to\infty$ 
\[ \Pi_{L_n}\left[\|\la-\la_0\|_\infty >M_n \veps_n^{\be,L_n} \given X\right] =\op. \] 
Note that for Haar wavelets, $\Pi_{L_n}[\cdot\given X]=\Pi[\cdot\given X]$ is the ordinary posterior. 
 
(b) Suppose
$(\psi_{lk})$ is a CDV wavelet basis with high enough regularity.  If 
$\be\wedge\ga>1/2$,
for arbitrary $M_n\to\infty$, 
\[ \Pi\left[\|\la-\la_0\|_\infty >M_n \veps_n^{\be,L_n} \given X\right] =\op. \] 
 
As a consequence, for both (a)--(b), if $L_n=L_n(\be)$, the corresponding posterior distributions contract at optimal minimax rate $\veps_{n,\be}^*$ in supremum norm.  
\end{thm}

The rate in Theorem \ref{thm:sn} is sharp. In  Theorem \ref{lbhazard}, we prove the matching (in case $\ga=\be$) lower bound $\veps_{n,\be}^*$ as in \eqref{ratesup} up to constants for the minimax risk for the supremum loss for hazard estimation. The condition $\veps_n\leqa \veps_{n}^{\be,L_n}$ is typically not hard to check  by using the generic tools from \cite{Ghosal2000}; it is satisfied for all considered priors, see Section \ref{sec:l1} and the proof of Theorem \ref{thm:appli}. We note that supremum--norm posterior convergence results are still relatively few in the literature, and have mostly been developed in  white noise regression and density estimation, see for instance \cite{ic14}. Some recent contributions further deal with inverse problems and diffusions \cite{nicklschroedinger}--\cite{nickl_ray_diffusions20}, and also follow the multiscale approach \cite{cn14} adopted here. We further discuss the links with the results in \cite{ic14},\cite{cn13} and \cite{cn14} in the Discussion Section \ref{sec:discussion}, but for now let us just note that the previous results do not follow directly from density estimation results, the survival model structure needs further substantial work, in part due to the specific LAN norm structure (see Section \ref{sec:lan}) which underpins it.

\subsection{Applications: histogram classes and wavelet priors}
 \label{sec:appli}
 We now turn to the consequences for practice of our main results. We have verified the conditions for our main results for the four classes of priors announced in Section \ref{sec:familiesofpriors}. We recall that the classes are  {\bf (H$_1$)} random histograms with independent heights, {\bf (H$_2$)} random histograms with dependent heights, {\bf (H$_3$)} Haar wavelet priors, and {\bf (S)} wavelet priors with smoother wavelet bases. Here we present the details for the dependent and independent Gamma histogram priors as representatives of classes {\bf (H$_1$)} and {\bf (H$_2$)}.  Within  {\bf (H$_1$)} and {\bf (H$_2$)}, results are available for log-normal and log-Laplace priors as well. The specification of these log-normal and log-Laplace priors, as well as the wavelet priors in {\bf (H$_3$)} and   {\bf (S)}  can be found in  Section  \ref{sec:specificationpriors}.
 
 With $\operatorname{Gamma}(\alpha, \beta)$ we will refer to the Gamma distribution with shape parameter $\alpha$ and rate parameter $\beta$. The \emph{independent Gamma prior} then takes the form
 \begin{equation} \label{histheightmain}
 \la = \sum_{k=0}^{2^{L_n+1}-1}  \la_k \1_{I_{k}^{L_n+1}}, \quad \lambda_k \sim \operatorname{Gamma}(\alpha_0, \beta_0) \text{ i.i.d.}, \ k = 0, \ldots, 2^{L_n+1}-1
\end{equation} 
with $\alpha_0 > 0, \beta_0 > 0$ to be chosen freely.

The \emph{dependent Gamma prior} follows the formulation from \cite{Arjas1994}. It takes the same form as \eqref{histheightmain}, but now a dependence structure between the $\lambda_k$'s is introduced so that for $k = 1, \ldots, 2^{L_n+1}-1$ and some $\alpha > 0$:
\begin{equation*}\begin{aligned} 
E[\la_k \mid \la_{k-1}, \ldots, \la_0] &= \la_{k-1}\\
\text{Var}(\la_k \mid \la_{k-1}, \ldots, \la_0) &= (\la_{k-1} / \alpha)^2,
\end{aligned} \end{equation*}
The intuition behind such dependence is that it allows for borrowing of information across time periods and  leads to more `smoothly' varying prior histograms. The desired structure is achieved by setting, for some $\alpha_0, \beta_0, \alpha > 0$ to be chosen freely:
\begin{align*}
\lambda_0 &\sim \operatorname{Gamma}(\alpha_0, \beta_0)\\
\lambda_k \mid \la_0, \ldots, \la_{k-1} &\sim \operatorname{Gamma}(\alpha, \alpha/\la_{k-1}), \quad k = 1, \ldots, 2^{L_n+1}-1.
\end{align*}

We now state a result for the dependent and independent Gamma priors, which justify the use in practice of their corresponding credible bands for the cumulative hazard and the survival as if they were confidence bands. A more general theorem, covering all four classes of priors considered in this paper, is available as Theorem \ref{thm:applifull}. In particular, in Theorem \ref{thm:applifull} several classes of priors are identified for which the smoothness assumption can be decreased to $\beta > 0$ rather than $\beta > 1/2$.

Similar to Theorem \ref{thm:donsker} above, for proving the Donsker--type result on $\La$ in the next statement, we shall first prove a nonparametric BvM result on $\la$ in a multiscale space $\mathcal{M}_0(w)$, which is the last result mentioned in the next statement. Although this could be skipped at first read, we refer the reader to Section \ref{sec:npbvm}  for precise definitions.
\begin{thm} \label{thm:appli}
Let  $X=(X_1, \dots, X_n)$ be a sample of law $P_0$ with hazard rate $\la_0$ under conditions \ref{c:mod}. For  $\be,D>0$, suppose $\log\la_0\in \cH(\be,D)$. Suppose the prior $\Pi$ is an independent or dependent Gamma prior as described above, with $L_n$ chosen as in \eqref{eln} with $\ga=\be$. Then 
 for any $1/2 < \beta \leq 1$, 
\[ \cB_{\mathcal{D}[0,1]}\left(\,\cL(\rn(\La-\hat\La_n)\given X)\ ,\, \cL(G_{\La_0})\right)  \to^{P_{0}} 0,\]
for $\hat\La_n$ Nelson Aalen's estimator, as well as, for $\veps_{n,\be}^*$ as in \eqref{ratesup}, 
\[ \Pi\left[\|\la-\la_0\|_\infty >M_n \veps_{n,\be}^* \given X\right] =\op. \] 
Also, the posterior distribution satisfies the nonparametric BvM theorem, Theorem \ref{thm:bvm}, in $\cM_0(w)$ with the choices $w_l=l$ or $w_l=2^{l/2}/(l+2)^2$ and centering $T_n$ as in \eqref{tn}.
 \end{thm}

While for full details on all classes of priors, we refer the reader to Theorem \ref{thm:applifull}, we would like to draw attention here to one additional case of particular practical interest, which is the dependent log-Laplace prior. For the dependent log-Laplace prior, the condition $\beta > 1/2$ can be improved to $\beta > 0$. It is far from the only example for which $\beta > 0$ suffices, as the same is true for most wavelet priors from classes {\bf (H$_3$)} and   {\bf (S)}. Nevertheless, as computation for the histogram-type priors is  convenient, as expanded on in Section \ref{sec:simulation}, the dependent log-Laplace prior may be of particular practical interest.

\begin{remark} \label{rem:lnu}
The choice of $L_n$ as in \eqref{eln} is intended for obtaining  sharp supremum norm rates, but many other choices are possible, in particular if the focus is on the nonparametric BvM or the Donsker BvM. For instance, the Donsker BvM holds for any $\be>0$ with histogram priors if    $L_n=L_n^U$  as in \eqref{ctunder}, as stated  in Theorem \ref{thm:glimpse}; see the proof of Theorem \ref{thm:appli}. 
\end{remark}

\section{\label{sec:simulation}Simulation study}

We verify our coverage results empirically, employing the two histogram priors with independent or dependent Gamma heights from the previous section. In addition, we draw two comparisons: (i) between the aforementioned dependent and independent Gamma priors; and (ii) between the credible bands for the survival function arising from said Gamma priors, and two non-Bayesian confidence bands.

\subsection{Computing the credible and confidence bands}
For the frequentist confidence bands for the survival function, the Hall-Wellner band \cite{Hall1980} and the log-transformed equal precision band \cite{Nair1984} are competitive options \cite{Nair1984, Borgan1990}.  

We compute credible bands  for the cumulative hazard and for the survival function, in eight scenarios described below, expecting to see the good coverage guaranteed by  Corollary \ref{cor:bands}. To build the credible bands, we find, for each object, a minimal radius $r$ such that $(1-\alpha)100$\% of all posterior draws is within distance $r$ of the posterior mean, with $\alpha$ set to 0.05. As a final step, we bound the credible bands from below by 0, and in case of the survival function, from above by 1.

The piecewise constant priors allow us to exploit a convenient Poisson representation of the survival likelihood for sampling \cite{Holford1980, Laird1981}.  Details on this representation and about  the algorithm we used to sample from the posterior distributions are given in Section \ref{sec:sampler}.  The samplers for the dependent and independent Gamma priors, as well as the functions used to compute the credible bands are available in the BayesSurvival R package \cite{packageBayesSurvival}. The documentation of the BayesSurvival package contains further details about the sampler. While we did not pursue this further, we remark that for the CDV wavelet prior the convenient Poisson representation is not available, but standard MCMC methods should allow sampling from the corresponding posterior distribution.

\subsection{Scenarios and evaluation measures}
We consider two hazards, with $\tau = 1$:
\begin{enumerate}
\item The \emph{smooth} hazard, $\lambda_s(t) = 6((t+0.05)^3 - 2(t+0.05)^2 + t + 0.05) + 0.7$. 
\item The \emph{piecewise linear} hazard, $\lambda_{pl}$, which is equal to 3 on $[0, 0.4]$, to 1.5 on $[0.6, 1]$ and the linear interpolation on $[0.4, 0.6]$.
\end{enumerate}
These hazards and their corresponding cumulative hazards and survival functions are depicted in Figure \ref{fig:sim_truth}. Both hazards meet the conditions of our main results.

\begin{figure}
\caption{Illustration of the two hazards used in the simulation study, Section \ref{sec:simulation}.}
\label{fig:sim_truth}
\includegraphics[width=\textwidth]{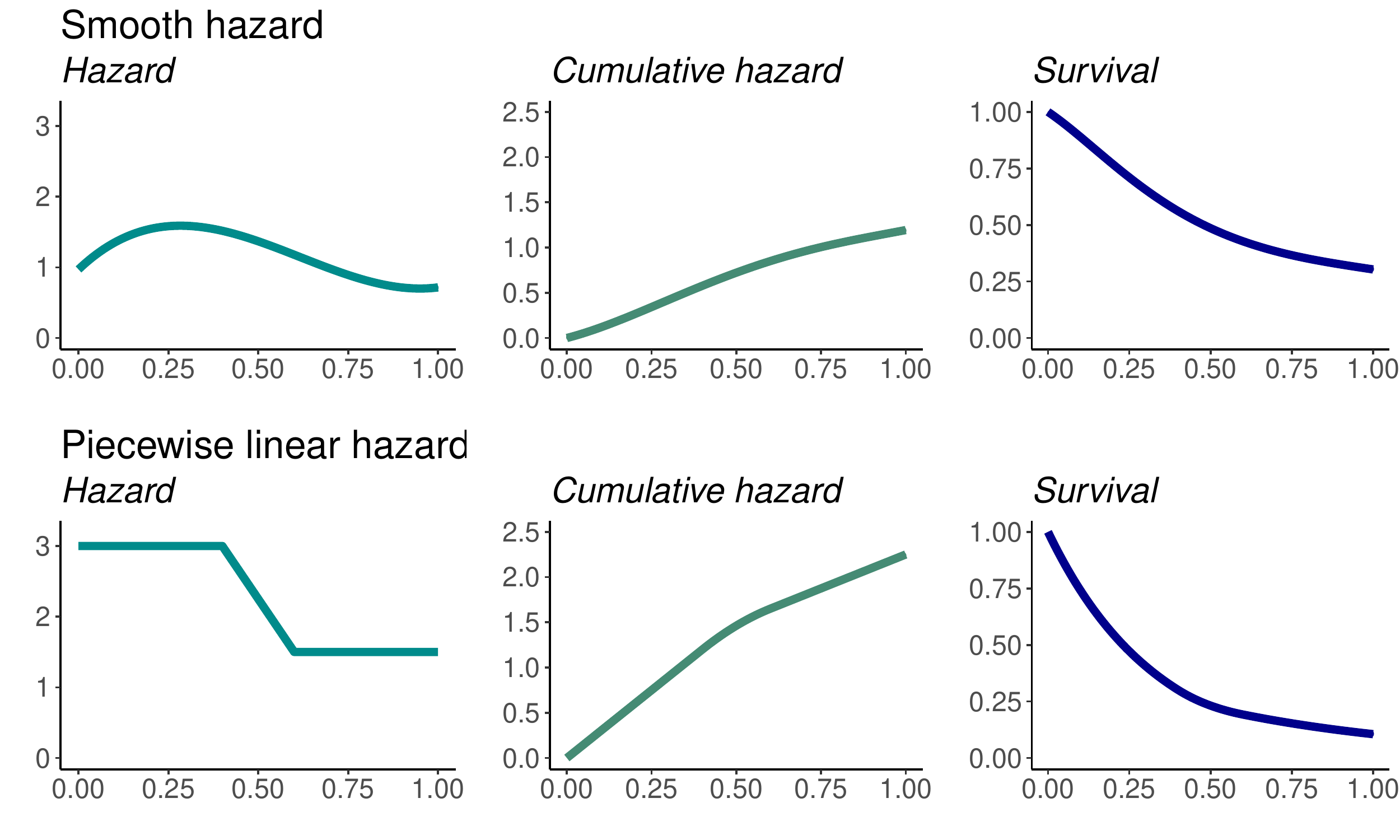}
\end{figure}

Per hazard, we consider two scenarios at two sample sizes, for a total of eight scenarios:
\begin{enumerate}
\item Independent uniform censoring throughout the interval $[0, 1)$ and administrative censoring at $t = 1$ (meaning that everyone still under follow-up at $t = 1$ is censored), with $n = 200$ or $n = 2000$ (55\% and 34\% censoring with $\la_s$ and $\la_{pl}$ respectively).
\item Administrative censoring only, with $n = 200$ or $n = 2000$ (30\% and 11\% with $\la_s$ and $\la_{pl}$ respectively). This scenario does not meet condition \ref{c:mod}, because the censoring density is equal to 0 on $[0, 1)$.
\end{enumerate}

We evaluate the coverage of the bands for the cumulative hazard and survival. In addition, we compute the areas of the credible bands for the survival function, and of the Hall-Wellner and log-transformed equal precision bands. To ensure a fair comparison, we post-process all survival bands to lie between 0 and 1. Finally, we retain and plot the posterior draws of the median survival to illustrate the Bernstein-von Mises phenomenon as expected from Corollary \ref{cor:median}. 

We create $N = 1000$ synthetic data sets for each setting, and set the confidence level to 95\%. The parameters for the dependent Gamma prior are $\alpha = \beta_0 = 1, \alpha_0 = 1.5$, and for the independent Gamma prior $\alpha = 1.5, \beta = 1$.  

As in our theoretical results, we work with a number of intervals $K\circeq (n / \log{n})^{\frac{1}{1+2\gamma}}$ as in \eqref{eln}. We consider both $\gamma = 1/2$, which can be viewed as the recommended default value (as the Donsker BvM theorem for $\La$ and $S$ then holds regardless of the smoothness value $\be>0$ of the hazard, see below Theorem \ref{thm:appli}), and $\gamma = 1$, exploiting that in our scenarios, the true hazards are smooth.

\subsection{Results}
The coverage and size results for the survival bands are given in Table \ref{table:surv}.  Plots of the posterior draws of the median survival, for both $\lambda_{pl}$ and $\lambda_{s}$ in the scenario with  administrative censoring only,  $n = 2000$, $\gamma = 1/2$ and the dependent Gamma prior are given in Figure \ref{fig:sim_med_surv}. The coverage results for the credible bands for the cumulative hazard are reported in  Section \ref{sec:simhaz}.

\begin{table}[t]
\caption{Coverage and area of the credible bands for the survival function using the dependent and independent Gamma priors, and of the Hall-Wellner (H-W) and log-transformed equal precision (log-EP) bands. The parameter $\gamma$ is that of \eqref{eln}, so $\gamma = 1/2$ corresponds to $K = \left\lceil \left(n/\log{n}\right)^{1/2} \right\rceil$ intervals and $\gamma = 1$ to  $K = \left\lceil \left(n/\log{n}\right)^{1/3} \right\rceil$ intervals.}
\label{table:surv}
\begin{tabular}{l|rr|rr|rr}
& \multicolumn{2}{c|}{$\gamma = 1/2$} & \multicolumn{2}{c|}{$\gamma = 1$} & \\
 & dep. & indep. & dep. & indep.	& H-W	& log-EP\\
 \hline
\multicolumn{7}{l}{ \textbf{Smooth hazard}} \\
 \hline
  $n = 200$, adm. + unif. & 0.95  & 0.97 & 0.94 & 0.96 & 0.94  & 0.93  \\
 \emph{area} & \emph{0.22} & \emph{0.21} & \emph{0.22} & \emph{0.21} & \emph{0.26} & \emph{0.25} \\
 $n = 2000$, adm. + unif. & 0.95 & 0.95 &  0.96 & 0.96 &  0.96 &  0.96 \\
 \emph{area} & \emph{0.08} & \emph{0.08} & \emph{0.08} & \emph{0.08} & \emph{0.08} & \emph{0.08} \\
   $n = 200$, adm. & 0.95 & 0.95 & 0.97 & 0.96 & 0.97 & 0.91 \\
\emph{area} & \emph{0.16} & \emph{0.15} & \emph{0.15} & \emph{0.15} & \emph{0.18} & \emph{0.19} \\
 $n = 2000$, adm. & 0.94  & 0.94 & 0.96  & 0.96 & 0.97   & 0.95  \\
 \emph{area} & \emph{0.05} & \emph{0.05} & \emph{0.05} & \emph{0.05} & \emph{0.06} & \emph{0.06} \\
  \hline
\multicolumn{7}{l}{ \textbf{Piecewise linear hazard}} \\
 \hline
  $n = 200$, adm. + unif. & 0.94 & 0.95  & 0.95  & 0.95  & 0.95  & 0.94 \\
 \emph{area} & \emph{0.18} & \emph{0.17} & \emph{0.16} & \emph{0.16} & \emph{0.27} & \emph{0.22} \\
 $n = 2000$, adm. + unif. & 0.94  & 0.95  & 0.94 & 0.95 & 0.95  & 0.95  \\
 \emph{area} & \emph{0.07} & \emph{0.06} & \emph{0.06} & \emph{0.06} & \emph{0.10} & \emph{0.07} \\
   $n = 200$, adm. & 0.94  & 0.93 & 0.95  & 0.95 & 0.96  & 0.96 \\
\emph{area} & \emph{0.15} & \emph{0.15} & \emph{0.14} & \emph{0.14} & \emph{0.19} & \emph{0.18} \\
 $n = 2000$, adm. & 0.95  &  0.94 & 0.95  & 0.95  & 0.96  & 0.97  \\
 \emph{area} & \emph{0.05} & \emph{0.05} & \emph{0.05} & \emph{0.05} & \emph{0.06} & \emph{0.06} \\
 \end{tabular}
\end{table}

\begin{figure}
\caption{Histograms of the $N = 1000$ posterior draws of the median survival using the dependent Gamma prior, in the scenario with  administrative censoring only, $n = 2000$ and $K = \lceil (n/\log{n})^{1/2}\rceil$, with normal distributions centered at the mean of the draws with variance equal to the empirical variance of the draws (in gray) and the Gaussian expected from Corollary \ref{cor:median}, centered at the posterior median (in black). The histograms illustrate the Bernstein-von Mises result for the median survival of Corollary \ref{cor:median}. }
\label{fig:sim_med_surv}
\includegraphics[width=\textwidth]{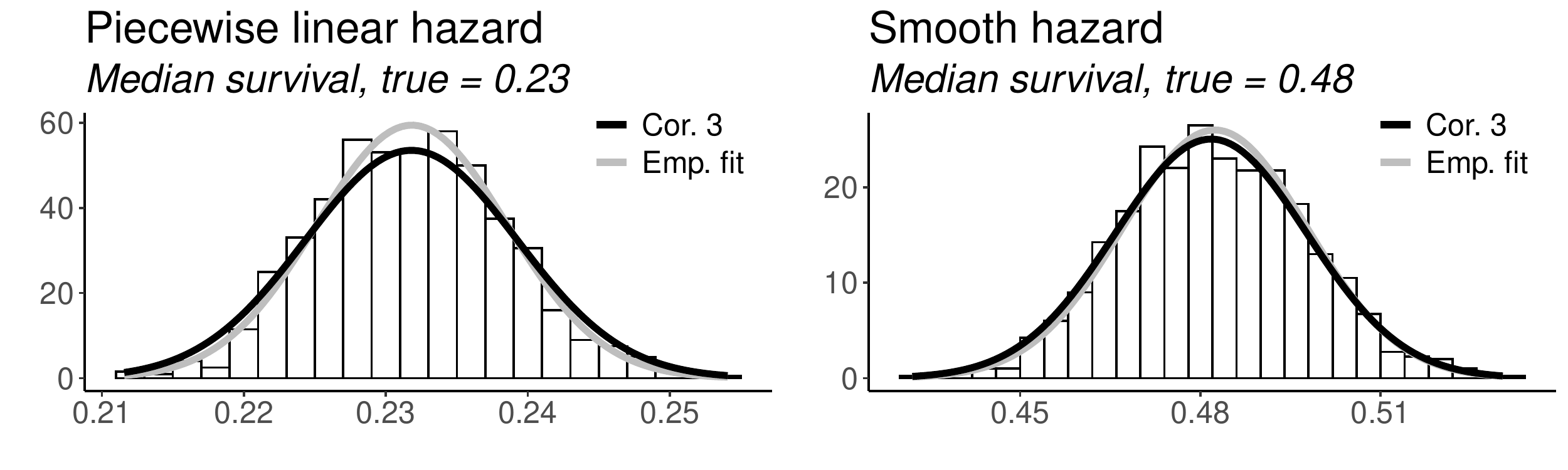}
\end{figure}

Figure \ref{fig:sim_med_surv} shows a normal shape of the posterior distribution of the median survival, as expected by Corollary \ref{cor:median}. For comparison, the Gaussian density centered at the posterior median of the median survival and with variance as in \eqref{medlim} (approximated by numerical integration) is shown, as well as the Gaussian centered at the mean of the draws with variance equal to the empirical variance of the draws.

\subsection{Discussion on simulations}
Comparing the results for the dependent and independent Gamma priors, the differences between the two are minor. Both priors achieve (close to) the nominal level of the band, as expected from Corollary \ref{cor:bands}. This even holds in the scenarios with only administrative censoring, despite the partial violation of condition \ref{c:mod} in this case. Decreasing the number of intervals from  $K = \lceil (n/\log{n})^{1/2}\rceil$ to  $K = \lceil (n/\log{n})^{1/3}\rceil$ leads to smaller bands in some cases, and somewhat higher coverage. The differences are small, and the number of intervals  $K = \lceil (n/\log{n})^{1/2}\rceil$ seems like a good choice when nothing is known about smoothness of the true hazard.

Comparing the coverage and areas of the credible bands for the survival  to those of the Hall-Wellner and log-transformed equal precision bands (Table \ref{table:surv}), we find the highest coverage in most scenarios by the Hall-Wellner band, but at the cost of an area that is up to roughly 70\% larger than that of the Bayesian version. The largest absolute differences in area are observed for the $n = 200$ sample size. The log-transformed equal precision band is closer in size to the credible bands, although still up to 30\% larger than the Bayesian credible bands, but comes with a decrease in coverage. Subtle differences matter, as is shown for example in the scenario with the piecewise linear hazard, $n = 200$ and administrative as well as uniform censoring. The Hall-Wellner band has 95\% coverage at an area of 0.27, the smaller log-transformed equal precision band has coverage 94\% at an area of 0.22, while a further decrease to an area of 0.16 to 0.18 (depending on the choice of $K$) for the Bayesian bands still results in 94-95\% coverage.

For context, the bands formed by the pointwise confidence intervals typically calculated around the Kaplan-Meier and Nelson-Aalen estimators offer no guarantee of coverage of the survival or cumulative hazard. Indeed, in the scenario with the smooth hazard, $n = 200$ observations and both uniform and administrative censoring, if the pointwise confidence intervals are collated into a band, we observed coverage of 43\% for the survival, and of 43\% for the cumulative hazard.

In conclusion, the Bayesian survival bands are an attractive option, providing high coverage despite their small size compared to the popular Hall-Wellner and log-transformed equal precision bands, and their use seems especially promising for small and moderate sample sizes. We note that we have considered here only a certain type of credible sets for simplicity: bands with radii determined by a posterior quantile. This simple choice already works remarkably well. One could also  consider more elaborate constructions of credible bands with varying radius: this is left for future work.

\section{Data application} \label{sec:data}
We visually illustrate how the Bayesian procedures compare to other existing popular methods. We do so on the North Central Cancer Treatment Group lung cancer data set \cite{Loprinzi1994}, which contains 228 observations of which 63 are censored. 

For the prior, we take the dependent Gamma prior, with the same parameter settings as in Section \ref{sec:simulation}. As we have no knowledge of the true smoothness of the hazard, we take as number of intervals $K = \lceil (n/\log{n})^{1/2}\rceil$ the default choice corresponding to $\gamma = 1/2$ in \eqref{eln}. 

We compare the  credible band and posterior mean for the survival function to three frequentist methods to quantify uncertainty: the Hall-Wellner band,  the log-transformed equal precision band, and the Kaplan-Meier estimator with its pointwise intervals,  in Figure \ref{fig:cancer_surv}. We also report on results for the hazard (posterior mean only) and cumulative hazard in  Section \ref{sec:datahaz}.

\begin{figure}
\caption{Real data experiment $(n=228)$. 
Posterior mean of the survival (solid) with credible band (shaded area), compared to (dashed, from left to right, top to bottom):  the Hall-Wellner band; the log-transformed equal precision (log-EP) band; Kaplan-Meier with pointwise confidence intervals. The plots illustrate the differences in sizes between the bands. }
\label{fig:cancer_surv}
\includegraphics[width=0.95\textwidth]{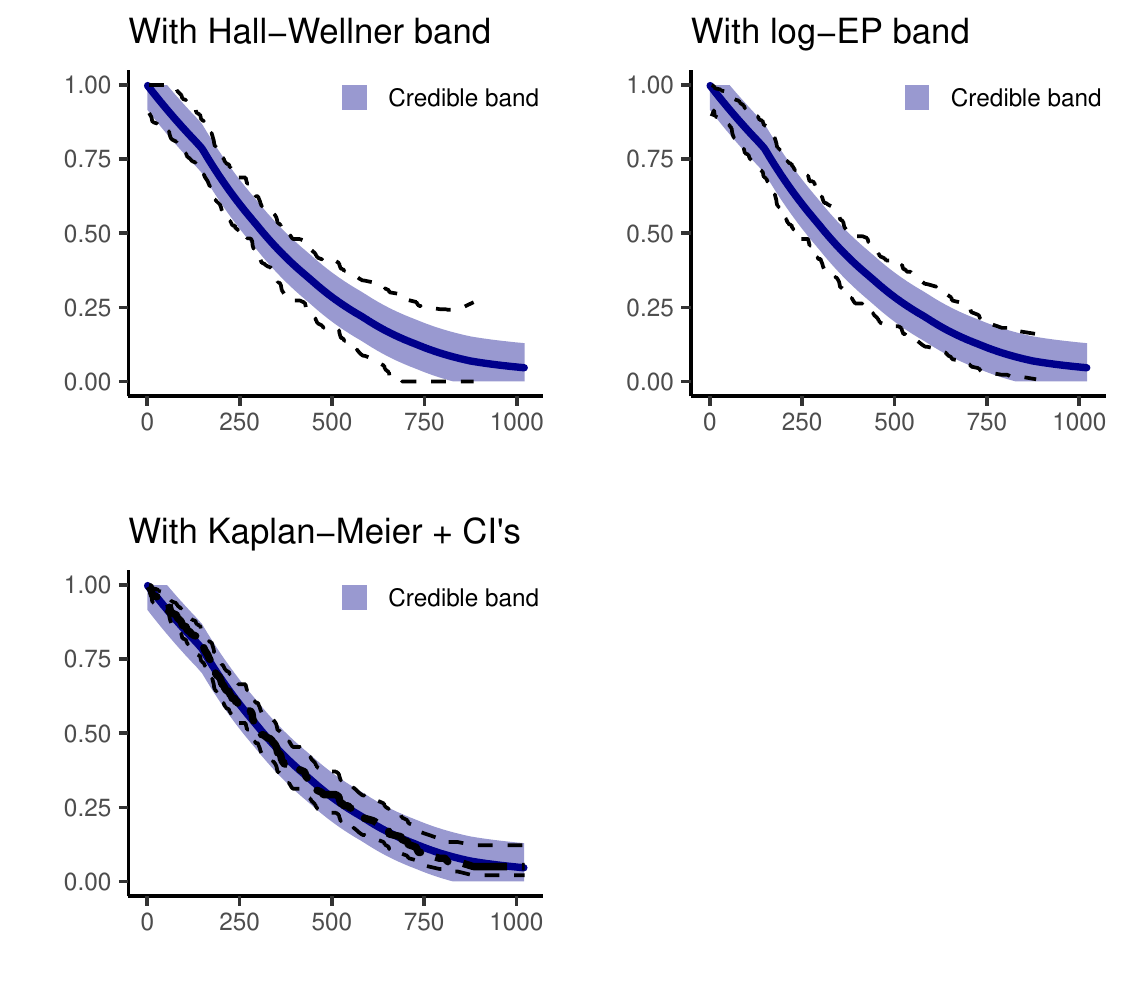}
\end{figure}

We see that the posterior means for all three survival objects are close to their frequentist counterparts. The credible bands for the survival function reveal an interesting pattern, matching what was observed in Section \ref{sec:simulation}. The area of the credible band is noticably smaller than that of the two non-Bayesian bands. Interestingly, the credible band is quite similar in size to the Kaplan-Meier pointwise confidence intervals, despite the much stronger guarantees now available for the credible band. This illustrates the conclusions from Section \ref{sec:simulation}, that the Bayesian credible band for the survival function is at an attractive point on the spectrum that trades off size and coverage.

\section{Discussion} \label{sec:discussion}

This work derives inference results for Bayesian procedures in the nonparametric  right-censoring model. Our results in particular provide theoretical back-up of practically used histogram priors on the hazard. We see that Bayesian methods are competitive with the standard frequentist options, providing natural uncertainty quantification, with credible sets reaching exact asymptotic coverage while having an optimal size in terms of efficiency.  Our methods could also be used to evaluate other classes of priors, not considered here, such as (truncated) Gaussian processes on log-hazards. More generalized frameworks like the generalized transformation model \cite{Younes1997} could also be studied with a similar approach, where one would first need to investigate the LAN properties of this generalization and then employ a Laplace transform approach similar to the one presented here.

While the Kaplan-Meier estimator with pointwise confidence intervals is  a highly popular method to quantify uncertainty in survival analysis (with non-guaranteed coverage for the `band' that arises visually by combining the pointwise intervals), our results show that the Bayesian paradigm offers an attractive alternative option, with automatic reliable uncertainty quantification. The credible bands are easily computed and turn out to be quite narrow compared to common frequentist methods for obtaining confidence bands. Any of the priors studied in this paper is guaranteed to yield good results, with the most crucial choice to be made being the number of intervals. We recommend taking the number \eqref{eln} as a guideline, with $\gamma = 1/2$ as a default choice in the absence of information on the smoothness of the underlying hazard.

The results also extend several recently obtained results in the Bayesian nonparametrics literature, in terms of rates for hazards, but also in terms of required regularity conditions. We briefly discuss  expected consequences for two models: density estimation and the Cox model.

In terms of density estimation, one can formulate results similar to the ones presented here. One main simplification in terms of proofs in density estimation is that LAN remainder terms are less complex, and thus easier to handle. A main novelty here for the density estimation framework with respect to \cite{ic14}--\cite{cn14} is in terms of regularity conditions. Using the scheme of proof of Theorem \ref{thm:sn}, parts (a) and (b), one can extend the corresponding results in \cite{ic14}--\cite{cn14} in density estimation on $[0,1]$, improving upon  minimal required regularities by at least $1/2$ (i.e. the cited works require $\be>1/2$ for wavelet histogram priors, this condition is removed here; also, $\be>1$ was required for smooth wavelet priors, here the condition becomes just $\be>1/2$ and is even completely removed for projected posteriors). Some of these refinements are similar in spirit to the idea of getting improved rates successively by an iterative argument, as was used recently by Richard Nickl and co--authors in inverse problems or diffusion  settings in \cite{nicklschroedinger}, \cite{nicklsoehl19}, \cite{nickl_ray_diffusions20} (note that the argument in these papers is not aiming at decreasing minimal regularity requirements though, which in inverse problems contexts are typically higher due to the `inverse problem' operator involved) -- a less sharp version of this idea also featured in \cite{ic14} p. 2083, to get consistency rates in the $\|\cdot\|_\infty$ norm for $\be\in(1/2,1]$ in density estimation, but those rates, unlike here, were not yet the optimal ones. 

Regarding the Cox model,  the paper \cite{Castillo2012} derived Hellinger posterior rates, as well as  consistency rates in terms of the LAN norm (for a slightly different presentation, see also Chapter 12 of the book \cite{gvbook}, where the  results of that paper are also presented). Those were enough for obtaining a BvM for $\te$ in the Cox model (a somewhat related result in the present paper is Theorem \ref{thm:bvmlin}). The required H\"older--regularity condition on the hazard was $\be>3/2$. Here we are able to go down to at least to $\be>1/2$ for $L^1$, LAN--norm and supremum norm rates (getting optimal rates for those, up to logarithmic terms for the first two norms), and to just $\be>0$ for histograms or truncated posteriors. 
For Hellinger and LAN norm rates, new arguments in Section \ref{sec:suffrates} enable this improvement. Note that pushing down the regularity constraint enables us to deal with histograms priors, 
which were ruled out in the treatment of \cite{Castillo2012, gvbook} (where a fast enough rate is needed for preliminary concentration in the LAN norm, and where supremum--norm rates are not discussed). Applying the present arguments in the Cox model is expected to enable lowering regularity requirements there. 
Another important  novelty in the present setting, already only in the setting of Theorem \ref{thm:bvmlin}, is dealing with functionals $\psg f,\psi_{lk} \psd$ for unbounded $l,k$, which in particular requires substantially more general bounds of remainder terms in the LAN expansion compared to \cite{Castillo2012}, see Lemmas \ref{lem:rn1}--\ref{lem:rn2}.

Here we studied conditions under which both BvM and supremum--norm rates can be achieved simultaneously, and under which such rates are optimal. If one is interested only in a BvM result for $\Lambda$, or satisfied with a BvM statement and supremum--norm consistency only, the class of priors with which this can be achieved grows larger and one may not need the techniques developed here. It would be interesting to develop theory for more classes of priors, possibly allowing for non-conjugate ones (as we do here).

The present work only addresses a certain set of questions. There are many other interesting ones to consider and the present contribution is intended as a platform in the simplest nonparametric survival analysis model from which to derive other results. Future interesting directions within the survival analysis field include the use of  covariates \cite{Friedman1982}, with possibly nonproportional hazards \cite{Laird1981}, and dealing with other classes of priors, e.g. survival trees \cite{Bou2011}.

\subsection*{Acknowledgements} I. C. would like to thank Richard Nickl for insightful discussions, in particular pertaining to low regularities treatment. S. P. would like to thank Leonhard Held for drawing our attention to the median survival as a quantity of interest, and Judith Rousseau for a question on lower smoothness levels. The authors would also like to thank the Associate Editor  and referees, as well as Bo Ning, for insightful comments.

\section{Proofs}  \label{sec:proofs}

Let us write $r = \log \la$ and $r_0 = \log \la_0$ for two hazards $\la, \la_0$, and $a = \sqrt{n}(r-r_0)$ for the scaled difference.

Let us recall the notation $M_0(u) = E_{\la_0}\1\{u \leq Y\} = (1-G(u))e^{-\La_0(u)}$, and, with $P_{L_n}$ the orthogonal projection onto $\text{Vect}\{\psi_{lk}, 0 \leq k < 2^l, l \leq L_n\}$,
\begin{equation*}\label{eq:psindef}
\psin = P_{L_n}(b/M_0).
\end{equation*}

The density of the pair $(Y, \delta) = (T \wedge C, \delta)$ (with respect to $\text{Leb}[0,\ta]+\delta_0$) is
\begin{align*}
p^{(Y,\delta)}_\la(y, d) &= \underbrace{\{g(y)S(y)\}^{1-d} }_{\text{Censored before time } \tau } \underbrace{ \{(1 - G(y-))\la(y)S(y)\}^d}_{\text{Event before time }\tau} \1\{y < \tau\}\\
&\quad\quad\quad + \underbrace{\{(1-G(\tau-))S(\tau)\}}_{\text{Censored at time } \tau}\1\{d = 0, y = \tau\},
\end{align*}
where $S$ is the survival function defined by the hazard $\la$. The part regarding the censoring distribution factorises in the likelihood, so needs not to be modeled with a prior distribution. 

The log--likelihood ratio is given by, with $a=\rn(r-r_0)$,
 \[
 \ell_n(r) - \ell_n(r_0) = \frac{1}{\sqrt{n}} \sum_{i=1}^n \delta_i a(Y_i) -  \frac{1}{\sqrt{n}}\sum_{i=1}^n \sqrt{n}\left[\La(Y_i) - \La_0(Y_i)\right].
 \]

\subsection{LAN expansion} \label{sec:lan}
The log-likelihood can be rewritten to feature a limiting Gaussian experiment via a LAN expansion. For the Cox model, the LAN expansion was considered in \cite{Castillo2012} and the following can be seen as the special case where the Cox model parameter is $0$:
 \begin{equation}\label{eq:LAN}
  \ell_n(r) - \ell_n(r_0) = -\tfrac{1}{2}\|a\|_L^2 + W_n(a) + R_n(r,r_0),
 \end{equation}
 where the LAN-norm $\|\cdot\|_L$ stems from the inner product
 \[ \psg a_1, a_2 \psd_L = \La_0\{a_1a_2M_0\} = \int_0^1 a_1(u)a_2(u)M_0(u)\la_0(u)du\]
that is $\|a\|_L^2 = \La_0\{ a^2M_0\}$, where
 \begin{equation}\label{eq:dispWna}
 W_n(a) = \frac{1}{\sqrt{n}} \sum_{i=1}^n[\delta_i a(Y_i) - \La_0a(Y_i)] 
 \end{equation}
and where the remainder term can be decomposed as, $R_n(r,r_0) = R_{n,1}(r,r_0) + R_{n,2}(r,r_0)$,  with  $\mathbb{G}_nf = \frac{1}{\sqrt{n}}\sum_{i=1}^n[f(Y_i) - P_0f]$,  
 \begin{align*}
R_{n,1}(r,r_0) & = \mathbb{G}_n( (\La_0 a)(\cdot) - \sqrt{n}(\La - \La_0)(\cdot)), \\
R_{n,2}(r,r_0) & = n\La_0\{M_0(1 + r - r_0 + \tfrac{1}{2}(r-r_0)^2 - e^{r-r_0})\}(1).
 \end{align*}

\subsection{Proof of Theorem \ref{thm:glimpse}}
For the cut--off \eqref{ctga}, the results are special cases of Theorem \ref{thm:applifull} (dealing with Laplace priors on coefficients). 
 For the cut--off  $L_n=L_n^U$ in \eqref{ctunder}, one proceeds following similar arguments as in the proof of Theorem \ref{thm:applifull} for $L_n=L_n(\ga)$.

\subsection{Proof of Theorem \ref{thm:bvmlin}} \label{sec:proofbvmlin}
In the theorem statement, we denoted $\varphi_b(\la) = \int_0^1 b(u) \la(u) du$ to make the dependence on $b$ explicit, which will be useful in the sequel. Since here we consider one particular $b$, for simplicity of notation in the next lines $\psi(\la)$ will stand for $\varphi_b(\la)$, for $b\in L^2(\Lambda)$.  First, we wish to relate the difference $\psi(\la) - \psi(\la_0)$ to a LAN--inner product involving  $r - r_0$. Let us recall that $\eif=b/M_0$ and $\psin =P_{L_n}\eif$.

  Let us  write
\begin{align*}
\psi(\la) - \psi(\la_0) &= \int \frac{\la-\la_0}{\la_0} \frac{b}{M_0} M_0 \la_0
= \left \langle \frac{\la - \la_0}{\la_0}, \eif\right\rangle_L.
\end{align*}
We define $B(\la, \la_0) = \langle r - r_0 - \frac{\la - \la_0}{\la_0}, \eif \rangle_L$, so that we have 
\begin{equation}\label{eq:defBcensored}
\psi(\la) - \psi(\la_0) = \langle r - r_0, \eif \rangle_L - B(\la, \la_0).
\end{equation}

\emph{\label{sec:BvMforlambda}Bernstein-von Mises for $\lambda$.}\label{sec:BvMcensored} 
We intend to show that, for any real $t$,
\begin{equation}\label{eq:BvMgoalcensored}
E\left[ e^{t\sqrt{n}[\psi(\la) - \widehat\psi]} \mid X, A_n \right] \overset{P_{\lambda_0}}{\rightarrow}  e^{\frac{t^2}{2}\|\eif\|_L^2},
\end{equation}
for some $\widehat\psi$ yet to be defined and where $A_n$ is as in \ref{c:l1}. 
This convergence of the Laplace transforms implies that the distribution $\Pi[\cdot \given X,A_n]\circ L^{-1}_{\hat\psi}$ converges in terms of the bounded Lipschitz metric to a $\cN(0, \|\eif\|_L^2)$ distribution (see Lemmas 1 and 2 in \cite{cr15supp} for more details). In turn, this implies the same result for $\Pi[\cdot \given X]\circ L^{-1}_{\hat\psi}$ by using $\Pi[A_n\given X]=\op$ and the definition of the bounded Lipschitz metric, which is the desired result, provided $\widehat\psi$ is an efficient estimator of $\psi(\la)$.
The following uses some elements of the proof of Theorem 4.1 of \cite{Castillo2015}, but one main difference is that later, we wish to apply  the results to  many $b$'s simultaneously as in \cite{ic14}--\cite{cn14} in a (essentially) non--asymptotic fashion. In the following paragraphs, we make in passing a few useful notes in preparation of  the proof of Proposition \ref{prop:lap} below.

We start by expanding the left hand side of \eqref{eq:BvMgoalcensored}, with $\psi(\la_0)$ instead of $\widehat\psi$. Applying Bayes' formula together with \eqref{eq:defBcensored} leads to
\begin{align*}
E\left[ e^{t\sqrt{n}[\psi(\la) -\psi(\la_0)]} \mid X, A_n \right] 
&= \frac{\int_{A_n} e^{t\sqrt{n}\langle r - r_0, \eif\rangle_L- t\sqrt{n}B(\la, \la_0) + \ell_n(r) - \ell_n(r_0)} d\Pi(r) }
{\Pi[A_n\given X]\int e^{\ell_n(r) - \ell_n(r_0)} d\Pi(r)}.
\end{align*}
The idea is now to merge the terms appearing on the exponential on the numerator of the last display. To do so, it is helpful to introduce a term $\ell(r_t)$, with $r_t=r-tn^{-1/2}a_n$ for some suitable function $a_n$.  One natural choice is $a_n=\psi_b$. However, for later treatment, as $\eif$ in general does not have a finite expansion onto the basis $(\psi_{lk})$, it is helpful to project it onto the space $\cV_{L_n}$ spanned by the prior. So one rather sets $a_n=\psin=P_{L_n}\psi_b$ and one defines $\la_t(u) = \la(u) e^{- t\psin/\rn}$ for  $u\in[0,1]$ (and interpolate to $\la(u)$ for $u > 1$ to ensure $\int_0^\infty \la_t(u)du = \infty$, so that $\la_t$ is a proper hazard). Using the LAN expansion \eqref{eq:LAN}, we compute
\begin{align*}
\ell_n&(r) - \ell_n(r_0) - [\ell_n(r_t) - \ell_n(r_0)]\\
&= \frac{t^2}{2} \|\psin\|_L^2 - t\sqrt{n}\langle r-r_0,\psin\rangle_{L} + tW_n(\psin) + R_n(r, r_0) - R_n(r_t^n, r_0).
\end{align*}
Plugging this into the last but one display, one obtains
\begin{align}
 E\left[ e^{t\sqrt{n}[\psi(\la) -\psi(\la_0)]} \mid X, A_n \right] \cdot e^{-\frac{t^2}{2} \|\psin\|_L^2- tW_n(\psin)} 
 \cdot \Pi[A_n\given X]\qquad \label{laphist}\\
\quad=\frac{\int_{A_n} e^{  \ell_n(r_t^n) - \ell_n(r_0) + R_n(r, r_0) - R_n(r_t^n, r_0) +t\sqrt{n}[\psg r-r_0,\eif-\psin\psd_L-B(\la, \la_0)]} d\Pi(r) } 
{\int e^{\ell_n(r) - \ell_n(r_0)} d\Pi(r)}.\notag
\end{align}

Let us split the term $B(\la,\la_0)$ in two parts as follows
\[
 \psg  r - r_0 - \frac{\la - \la_0}{\la_0}, \eif -\psin \psd_L
+  \psg  r - r_0 - \frac{\la - \la_0}{\la_0}, \psin \psd_L =: 
B_n^c(\la,\la_0) + B_n(\la,\la_0).
\]
Rearranging the expression, with $B_n(\la,\la_0)$ defined in the last display, 
\[ \psg r-r_0,\eif-\psin\psd_L - B(\la, \la_0)
 = -B_n(\la,\la_0) + \psg \frac{\la-\la_0}{\la_0} ,\eif-\psin\psd_L.
\]
The last term induces a semiparametric bias $t\rn \psg \frac{\la-\la_0}{\la_0} ,\eif-\psin\psd_L$ in \eqref{laphist}, because of the approximation of $\eif$.  On the set $A_n$, this term is bounded by 
\begin{align*}
 |\psg \frac{\la-\la_0}{\la_0} & ,\eif-\psin\psd_L| =|\int (\la-\la_0)(\eif-\psin) M_0 | \\& \le \|\la-\la_0\|_1\|\eif-\psin\|_\infty \|M_0\|_\infty
  \le C \veps_n \|\eif-\psin\|_\infty.
\end{align*}
Using \ref{c:funb},  this expression is a $o(1/\rn)$, which shows that for any {\em fixed} $t$,
\[ t\rn |\psg \frac{\la-\la_0}{\la_0}  ,\eif-\psin\psd_L| =o(1). \]

\emph{The LAN remainder terms.} \label{secrelan}
We keep in mind that the remainder terms (and/or their differences) only need to be bounded on the sets  $A_n$ as in \ref{c:l1}. Going back to \eqref{laphist}, one notes that 
\begin{align*}
R_n&(r, r_0)  - R_n(r_t^n, r_0) +t\sqrt{n}\left[\psg r-r_0,\eif-\psin\psd_L-B(\la, \la_0)\right] \\
& = R_n(r, r_0) - R_n(r_t^n, r_0)-t\rn B_n(\la,\la_0) + t\rn \psg \frac{\la-\la_0}{\la_0} ,\eif-\psin\psd_L,
\end{align*}
and the last bias term has been shown to be under control above.

We first look at  $R_{n,1}(r, r_0) - R_{n,1}(r^n_t, r_0)$. With 
$\La^n_t(\cdot) = \int_0^\cdot e^{r - \tfrac{t}{\sqrt{n}}\psin}$,
\begin{align*}
R_{n,1}&(r, r_0) - R_{n,1}(r^n_t, r_0)
= \sqrt{n} \mathbb{G}_n\left\{ \frac{t}{\sqrt{n}}\La_0\psin(\cdot) - (\La - \La^n_t)(\cdot) \right\}\\
&= \mathbb{G}_n\left\{t(\La_0-\La)\psin(\cdot) - \sqrt{n}\left(\La\left[1 - \frac{t}{\sqrt{n}}\psin\right] - \La^n_t\right)(\cdot) \right\}\\
&= \mathbb{G}_n\left\{t(\La_0-\La)\psin(\cdot) + \sqrt{n}\int_0^\cdot e^r\left( e^{- \tfrac{t}{\sqrt{n}}\psin} - 1 + \frac{t}{\sqrt{n}}\psin \right) \right\}.
\end{align*}

This part, we can control using empirical process tools: we handle each term in the sum from the last display separately.  
We write $f_n = (\La_0-\La)\psin$ and $g_n = \sqrt{n}\int_0^\cdot e^r( e^{- \tfrac{t}{\sqrt{n}}\psin} - 1 + \frac{t}{\sqrt{n}}\psin )$. 
In order to apply Lemma \ref{lem:rn3}, one first checks that $\psin$ belongs to the set $\cH_n$ defined in \eqref{hach}. This follows from Lemma \ref{lem:psi}, as here $b$ is a fixed element in $L^\infty[0,1]$, which implies $\|\psin\|_\infty\leqa  
L_n\|b\|_\infty\leqa L_n$  and $\|\psin\|_2\leqa \| b\|_2\leqa 1$, so that one can set $\mu_n=CL_n$ for some $C>0$. Second, one notices that on $A_n$ the hazard $\la$ verifies the conditions defining the set $\cL_n^1$ in \eqref{ellnb} with $\veps_n$ as in \ref{c:l1}. This shows $f_n\in\cF_n^1$ and $g_n\in\cG_n^1$, for $\cF_n^1, \cG_n^1$ as in \eqref{fnb}--\eqref{gnb}. By Lemma \ref{lem:rn3}, noting that $|t|\mu_n/\rn$ is bounded (even goes to zero here) for fixed $t$, one gets
\[  \sup_{f_n\in\cF_n^1,\, g_n\in\cG_n^1}  |\mathbb{G}_n(tf_n + g_n)|  =O_{P_{\la_0}}\left(|t|\veps_n\mu_n+ \frac{t^2}{\rn}(1+\veps_n\mu_n^2) \right). \]
The last bound is a $o_{P_{\la_0}}(1+t^2)=\op$ using  $\mu_n\leqa L_n$ and $\veps_n=o(L_n^{-2})$ by assumption.

We now turn to $R_{n,2}(r, r_0) - R_{n, 2}(r^n_t, r_0) - t\sqrt{n}B_n(\la, \la_0)$. We write:
\begin{align*}
R&_{n,2}(r, r_0) - R_{n, 2}(r^n_t, r_0) \\
&= n \La_0\left\{M_0\left(e^{r - \frac{t}{\sqrt{n}}\psin -r_0} - e^{r-r_0}  - (r^n_t-r) - \tfrac{1}{2}[(r^n_t-r_0)^2 -(r-r_0)^2] \right)\right\}\\
&= n \La_0\left\{ M_0\left(e^{r-r_0}\left(e^{- \frac{t}{\sqrt{n}}\psin} - 1\right) + \frac{t}{\sqrt{n}}\psin - \frac{1}{2}\frac{t^2}{n}\psin^2  + \frac{t}{\sqrt{n}}\psin(r-r_0) \right)\right\}
\end{align*}

The term $t\sqrt{n}B_n(\la, \la_0)$ may be rewritten in terms of $r$ and $r_0$ as $t\sqrt{n}B_n(\la, \la_0) = n\La_0\left\{ M_0\left[ r - r_0 - (e^{r-r_0} - 1) \right] \frac{t}{\sqrt{n}}\psin\right\}$. This cancels out partly as follows
\begin{align*}
R_{n,2}(r, r_0)&  - R_{n, 2}(r^n_t, r_0) -  t\sqrt{n}B_n(\la, \la_0)\\
&= n \La_0\left\{M_0\left(e^{r-r_0}\left(e^{- \frac{t}{\sqrt{n}}\psin} - 1 + \frac{t}{\sqrt{n}}\psin\right)  - \frac{1}{2}\frac{t^2}{n}\psin^2  \right)  \right\}.
\end{align*}

By the same argument as above, $h=\psi_{b,L_n}$ verifies that $|t|\|h\|_\infty/\rn$ is bounded and one sets $\mu_n=CL_n$, so that $h\in\cH_n, \la\in\cL_n^1$ as in \eqref{hach}--\eqref{ellnb}. By Lemma \ref{lem:rn3}, the previous display is  $O(t^2\veps_n\mu_n^2)=o(t^2)=o(1)$ for fixed $t$, using as before that $\veps_n=o(L_n^{-2})$.  

Given that the semiparametric bias is negligible on $A_n$ (as justified above),
\begin{align*}
&E_{\la_0}\left[ e^{t\sqrt{n}[\psi(\la) -\psi(\la_0)]} \given X, A_n \right] \Pi[A_n\given X] e^{-\frac{t^2}{2} \|\psin\|_L^2 - tW_n(\psin)}\notag\\
&\quad= \frac{\int_{A_n} e^{ R_n(r, r_0) - R_n(r_t^n, r_0) - t\sqrt{n}B_n(\la, \la_0) + o(1)+ \ell_n(r_t^n) - \ell_n(r_0)}  d\Pi(r) } 
{\int e^{\ell_n(r) - \ell_n(r_0)} d\Pi(r)}.
\end{align*}
From the previous computations $|R_n(r, r_0) - R_n(r_t^n, r_0) - t\sqrt{n}B_n(\la, \la_0)|$ is a $\op$ uniformly over $A_n$. On the other hand, with $r_t^n=r-t\psin/\sqrt{n}$,
\begin{equation} \label{ratiocv}
 \frac{\int_{A_n} e^{\ell_n(r_t^n) - \ell_n(r_0)}  d\Pi(r) } 
{\int e^{\ell_n(r) - \ell_n(r_0)} d\Pi(r)}
\end{equation}
goes to $1$ in probability as $n\to\infty$ 
  by \ref{c:cvarb}. Also,  Lemma \ref{lem:psi} implies $\|\psin\|_L\to\|\eif\|_L$ and $W_n(\psin-\eif)=\op$. From this one concludes that 
 \[ E_{\la_0}\left[ e^{t\sqrt{n}[\psi(\la) -\psi(\la_0)-W_n(\eif)/\rn]} \given X, A_n \right]\to e^{t^2\|\eif\|_L^2/2}
\]
under $P_{\la_0}$, which coincides with \eqref{eq:BvMgoalcensored} if one sets $\hat\psi=\psi(\la_0)+W_n(\eif)/\rn$. The latter expression is the first--order expansion of any efficient estimator of $\psi(\la)$ (see Section \ref{sec:backgroundeff} for some background on efficiency), which concludes the proof  of Theorem \ref{thm:bvmlin}. 

\subsection{A key proposition} \label{sec:prop}
We defer the proofs of Theorems \ref{thm:donsker}, \ref{thm:sn} and \ref{thm:appli} to later, but would like to present here a key technical tool underpinning both of them,  namely a version of (the `tightness part' of) the proof of Theorem \ref{thm:bvmlin}, which allows the norm $\|b\|_\infty$ to increase with $n$ and is essential for dealing with  many $b$'s simultaneously, such as wavelet basis functions with $l\to\infty$. Later we use it for $b=\psi_{LK}$ and $L\le L_n$. Recall condition \ref{c:cvarsn} from Section \ref{sec:donsker} and let us denote by $D_n$ the sets 
\begin{equation} \label{den}
D_n=\{\la:\ \|\la-\la_0\|_1\le \veps_n,\ \|\la-\la_0\|_\infty\le \zeta_n\}.
\end{equation}
Under \ref{c:l1}--\ref{c:sn}, we have $\Pi[D_n\given X]=1+\op$.

\begin{prop}[Laplace transform control for linear functionals] \label{prop:lap}

Let $b\in L^2(\La)$ be a possibly $n$--dependent functional representer that together with $\Pi$ satisfies \ref{c:cvarsn}. Suppose that for some constants $c_1,c_2>0$,
\[ \|b\|_2\le c_1,\qquad \|b\|_\infty\le c_2 2^{L_n/2}. \]
Then,  
 for $\veps_n$ as in \ref{c:l1}, 
we have for a constant $C>0$, for all $|t|\le \log{n}$,
\begin{align*}
 \log & E\left[ e^{t\rn\psg \la - \la_0 , b \psd } \given X, D_n \right] \\
& \le  tW_n(\psi_b) + C(1+t^2+|t|\rn \veps_n \| \eif - \psin\|_\infty)+|t|\op,
\end{align*}
where $\psi_b=b/M_0$, with $D_n$ the event in \eqref{den},  $W_n$ is defined in \eqref{wun}, and $\op$ is uniform with respect to $t, b$ verifying the above conditions.
\end{prop}
\begin{proof}[Proof of Proposition \ref{prop:lap}]
The proof follows in spirit that of Theorem \ref{thm:bvmlin}, but this time one keeps track of the dependence in $t$ as well as allows for possibly $n$--dependent $b$.

As before, one controls the Laplace transform $E\left[ e^{t\sqrt{n}\psg \la-\la_0,b \psd} \mid X, D_n \right]$. 
Using the bound on  semiparametric bias derived in the proof of Theorem \ref{thm:bvmlin},
\begin{align*}
 |\psg \frac{\la-\la_0}{\la_0}  ,\eif-\psin\psd_L|   \le C \veps_n \|\eif-\psin\|_\infty.
\end{align*}

To control remainder terms $R_{n,1}, R_{n,2}$, 
one now argues in a similar way as in the proof of Theorem \ref{thm:bvmlin}  (below we freely refer to functions $f_n, g_n$ appearing in $R_{n,1}$ and defined in that proof), but now with $b$ verifying the growth conditions of the Proposition. First, using Lemma \ref{lem:psi}, we have $\|\psin\|_\infty\leqa  
L_n\|b\|_\infty\leqa L_n2^{L_n/2}$  as well as $\|\psin\|_2\leqa \| b\|_2\leqa 1$, so that one can set $\mu_n:=CL_n2^{L_n/2}$ for some $C>0$. Second,   $\la$ belongs on $D_n$ to the set $\cL_n$ in \eqref{elln} once setting $v_n=\zeta_n$. This shows $f_n\in\cF_n$ and $g_n\in\cG_n$, for $\cF_n, \cG_n$ as in \eqref{fn}--\eqref{gn}. By Lemma \ref{lem:rn1}, noting that $|t|\mu_n/\rn$ is bounded if $|t|\le \log{n}$,
\[  \sup_{f_n\in\cF_n,\, g_n\in\cG_n}  |\mathbb{G}_n(tf_n + g_n)|  =O_{P_{\la_0}}\left(|t|\zeta_n+ \frac{t^2}{\rn}(1+\zeta_n) \right). \]
The last bound is a $o_{P_{\la_0}}(|t|)$ using that $\zeta_n=o(1)$ and $|t|\le \log{n}$. 
By applying Lemma \ref{lem:rn2} with $v_n=\zeta_n$,  the term $R_{n,2}$ is bounded by 
\[ |R_{n,2}(r, r_0)  - R_{n, 2}(r^n_t, r_0) -  t\sqrt{n}B_n(\la, \la_0)|= O\left(t^2\left\{\zeta_n + |t| \frac{1+\zeta_n}{\rn} \right\}\right).\]
As $|t|/\rn$ is bounded, one deduces, uniformly over the set $D_n$, the bound 
 $|R_n(r, r_0) - R_n(r_t^n, r_0) - t\sqrt{n}B(\la, \la_0)| \leqa t^2\zeta_n+O_P(|t|\{1+\zeta_n\})$. 
 
Using the Laplace transform control from the proof of Theorem \ref{thm:bvmlin}, 
\begin{align*}
&E_{\la_0}\left[ e^{t\sqrt{n}[\psi(\la) -\psi(\la_0)]} \given X, D_n \right] \Pi[D_n\given X] e^{|t|O(\rn \veps_n \| \eif - \psin\|_\infty)} \notag\\
&\quad=e^{\frac{t^2}{2} \|\eif\|_L^2+ tW_n(\eif)+ |t|\cdot \op + t^2 O(1) + O(|t|)} \frac{\int_{D_n} e^{ \ell_n(r_t^n) - \ell_n(r_0)}  d\Pi(r) } 
{\int e^{\ell_n(r) - \ell_n(r_0)} d\Pi(r)}.
\end{align*}
Note that  \ref{c:cvarsn} provides an upper bound of the ratio in the last display, as by definition $D_n\subset A_n$ and the integrand is positive. 
Now combining this with $\Pi[D_n\given X]=1+\op$ leads to
\begin{align*}
 E_{\la_0}& \left[  e^{t\sqrt{n}[\psi(\la) - \psi(\la_0)-W_n(\eif)/\rn ]} \given X, D_n \right]\\
 & \le (1+\op)e^{C(1+t^2)+|t|\{\rn \veps_n \| \eif - \psin\|_\infty+\op\}},
\end{align*}
as announced, which concludes the proof of Proposition \ref{prop:lap}.  \qedhere

\end{proof}

\subsection{Proofs of Corollaries \ref{cor:donsker}, \ref{cor:bands}, \ref{cor:median}}  Recall that the statements are formulated in the space $\mathcal{D}[0,1]$ of c\`adl\`ag functions on $[0,1]$ equipped with the supremum norm and the $\sigma$--algebra generated by open balls (see e.g. \cite{gillbook93}, Section II.8). 
To see that Corollary \ref{cor:donsker} follows from Theorem \ref{thm:donsker}: for the statement on $\La$, it is enough to check that $\rn\|\mathbb{T}_n-\hat\La_n\|_\infty=\op$ under the conditions of Corollary \ref{cor:donsker}. This is verified below. The statement on $S$ follows by Hadamard--differentiability of the negative exponential map $\La\to e^{-\La}$ from $\mathcal{D}[0,1]$ to $\mathcal{D}[0,1]$. Corollary \ref{cor:bands} follows by the continuous mapping theorem, as the map $g\to \|g\|_\infty$ is continuous from $\mathcal{D}[0,1]$ to $\RR^+$, for $\mathcal{D}[0,1]$ equipped with the supremum norm.  Finally, 
the result for the median functional in Corollary \ref{cor:median} is obtained using Hadamard--differentiability of the quantile transformation (on $\mathcal{D}[0,1]$ tangentially to the set of continuous functions at the considered point) as established in Lemma 21.3 of \cite{aad98}.

It now remains to check that $\rn\|\mathbb{T}_n-\hat\La_n\|_\infty=\op$. One first notes that Nelson--Aalen's estimator $\hat \La_n$ is an `efficient' estimator of $\La$ in that it is asymptotically linear in the efficient influence function. Namely, one has, see e.g. \cite{gillbook93}, p. 626, recalling the definition of $W_n$ in \eqref{eq:dispWna},
\[ \sup_{t\in[0,1]} \left| \rn(\hat\La_n-\La_0)(t) - W_n\left( \frac{\1_{\cdot\le t}}{M_0(\cdot)} \right) \right| = \op, \]
where we rewrite the integral with respect to the martingale process in \cite{gillbook93}, p. 626 in terms of $W_n$. 
It is thus enough to check that $\rn\|\mathbb{T}_n-\La^*\|_\infty=\op$, where we have set $\La^*(t)= \La_0(t)+n^{-1/2}W_n\left(\1_{\cdot\le t}/M_0(\cdot) \right)$. This follows from Lemma \ref{lem:diff}, which concludes the verifications for Corollary \ref{cor:donsker}.

\section{Details on histogram and wavelet bases} \label{sec:wave}

\subsection{Wavelets}  The following applies for $(\psi_{lk})$ either the Haar or the CDV wavelet basis. Although some properties of the Haar basis are lost when using CDV (e.g. simple explicit expression or the fact that for a given $l$, the supports of $\psi_{lk}$'s are disjoint for Haar), most convenient localisation properties and characterisation of spaces are maintained. Recall that the CDV basis is still denoted  $(\psi_{lk})$, with indexes $l\ge 0$, $0\le k\le 2^l -1$ (with respect to the original construction in \cite{cdv93}, one starts at a sufficiently large level $l\ge J$, with $J$ fixed large enough; for simplicity, up to renumbering, one can start the indexing at $l=0$). 
Let  $\alpha >0$ be fixed. Then the following properties and notation are used for both bases.
\begin{itemize}
	\item[(W1)] $(\psi_{lk})$ forms an orthonormal basis of $L^2[0,1]$
	\item[(W2)] $\psi_{lk}$ have 
	support $S_{lk}$, with diameter at most a constant (independent of $l,k$) times $2^{-l}$, and 
	$\|\psi_{lk}\|_{\infty} \leqa 2^{l/2}$. The $\psi_{lk}$'s are 
	in the H\"{o}lder class   $\cH(S, D)$, for some $S\ge \al$, $D > 0$. 
	\item[(W3)] 
	At fixed level $l$, given a fixed $\psi_{lk}$ with support $S_{lk}$, 
	\begin{itemize}
		\item[$\diamond$] the number of wavelets of the level $l'\le l$ with support intersecting $S_{lk}$ is bounded by a universal constant (independent of $l',l,k$)
		\item[$\diamond$] the number of wavelets of the level $l'>l$ with support intersecting $S_{lk}$ is bounded by $2^{l'-l}$ times a universal constant. 
	\end{itemize}
	The following localisation property holds: 
	$\sum_{k=0}^{2^l-1} \|\psi_{lk}\|_{\infty} \leqa 2^{l/2}$, where the inequality is up to a fixed universal constant.   
	
\end{itemize}
The basis $(\psi_{lk})$ characterises Besov spaces $B^{s}_{\infty,\infty}[0,1]$, any $s \le \al$, in terms of  wavelet coefficients. That is, $g\in B_{\infty,\infty}^s[0,1]$ if and only if
\begin{equation} \label{besovnorm}
	\|g\|_{\infty,\infty,s} := \sup_{l\ge 0,\ 0\le k\le 2^{l}-1} 2^{l(\frac12+s)} |\psg g,\psi_{lk} \psd_2| <\infty. 
\end{equation}
Also, recall that $B_{\infty,\infty}^s$ coincides with the H\"older space $\cC^s$ when $s$ is not an integer, and that when $s$ is an integer the inclusion 
$\cC^s\subset B_{\infty,\infty}^s$ holds. If  the Haar-wavelet is considered, the fact that $f_0$ is in  $\cH(s, L)$, $0<s\le 1$, $L>0$,  implies that the supremum in \eqref{besovnorm} with $\psi_{lk}=\psi_{lk}^H$ is finite. Also by definition  the H\"older class $\cH(s,L)$ introduced below \eqref{cve} is a subset of $B_{\infty,\infty}^s$.

\subsection{Histograms\label{sec:relatinghistogramsandhaar}} 
The priors of classes {$\bf(H_1)$} and {$\bf(H_2)$} are defined through the step heights of dyadic histograms, while the priors of class {$\bf(H_3)$} are defined on the wavelet coefficients. In the sequel, proofs for classes {$\bf(H_1)$} and {$\bf(H_2)$} (e.g. in Section \ref{sec:changevar}) can be simplified by relating the step heights of the histograms to the wavelet coefficients through the Haar transform. Specifically, we may write $r_W = 
W r_H$, where $r_{W} = (r_{-1}, r_{00}, r_{10}, r_{11}, \ldots, r_{L(2^L-1)})'$ is a vector of the wavelet coefficients for the wavelets up to level $L$, $r_H = (r_1, \ldots, r_{2^{L+1}})'$ contains the step heights of the histogram, and the entries of $W$ are given by $W_{-1,j} = 2^{-(L+1)}$ and
\begin{align*}
	W_{lk,j} &= 2^{-(L+1)+l/2}\left[\1_{I_{j-1}^{L+1} \subset I_{2k}^{l+1}} - \1_{I_{j-1}^{L+1} \subset I_{2k+1}^{l+1}} \right],
\end{align*}
where $W_{lk, \cdot}$ is the row corresponding to the wavelet coefficient $r_{lk}$. We remark that $2^{\frac{L+1}{2}}W$ is an orthogonal matrix.

\section{Nonparametric BvM theorem for the hazard rate and Donsker's theorem} 
\label{sec:npbvm}

\subsection{Background and theorem statement}

In this Section we derive a nonparametric BvM theorem for the hazard rate $\la$, Theorem \ref{thm:bvm}, following the approach of \cite{cn13}--\cite{cn14}. Since $\la$ is a nonparametric quantity and, unlike the cumulative  hazard $\La$, is typically estimable only at rates much slower than the parametric rate $1/\rn$ in the usual loss functions (such as $L^2$ or $L^\infty$--losses), the formulation of such a result needs some care. In order for the rate to be of order $1/\rn$, one weakens the loss function, which we take as a norm on multiscale spaces $\cM$ as defined below. The use of such spaces is 
also motivated by properties of the mapping $\la\to\La(\cdot)=\int_0^\cdot \la$, which will be continuous from $\cM$ to the space of continuous functions and will therefore allow a transfer from the nonparametric BvM result for $\la$ to a Donsker--BvM result on $\La$, using the continuous mapping theorem.  

We introduce such \emph{multiscale spaces} now, and refer to \cite{cn13}--\cite{cn14} for more details and background.

Let $(w_l)$ be a sequence such that $w_l/\sqrt{l}\uparrow \infty$. Call this an {\em admissible} sequence. Define the {\em multiscale} space $\cM$ of $\la$'s identified  from their sequence of wavelet coefficients as 
\[ \cM := \cM(w) = \left\{ \la=\{\psg \la,\psi_{lk} \psd\},\ \ \sup_{l} 
\max_{k} \frac{|\psg \la , \psi_{lk} \psd|}{w_l} <\infty  \right\},  \]
equipped with the norm  $\|\la\|_\cM=\sup_{l} \max_{k} |\psg \la , \psi_{lk} \psd| / w_l$  and consider the following separable subspace of $\cM$
\begin{equation}\label{eq:multi0}
	\cM_0 := \cM_0(w) = \left\{ \la=\{\psg \la,\psi_{lk} \psd\},\ \ \lim_{l\to\infty} 
	\max_{k} \frac{|\psg \la , \psi_{lk} \psd|}{w_l} = 0  \right\}.  
\end{equation}

{\em Limiting distribution.} Recalling $M_0(\cdot)=P_{\la_0}[Y\ge \cdot]$, let us define  $Q_0$, probability measure on $[0,1]$ with density $u_0:=\la_0/M_0$  with respect to Lebesgue's measure, that is $dQ_0(x)=u_0(x)dx$, and define the zero-mean Gaussian process $\mZ_{Q_0}$ (call it  $Q_0$--white noise process) indexed by the Hilbert space $L^2(Q_0)=\{f:\ \int_0^1 f^2dQ_0<\infty\}$, with covariance function
\begin{equation} \label{def:gpz}
	E[\mZ_{Q_0}(g)\mZ_{Q_0}(h)]=\int_0^1 ghdQ_0. 
\end{equation}

{\em Centering $T_n$.} The centering $T_n$ was defined in \eqref{tn} in the main paper, as
\begin{equation}
	\psg T_n , \psi_{lk} \psd =
	\begin{cases}
		\ \psg \la_0 , \psi_{lk} \psd+W_n\left(\psi_{lk}/M_0\right)/\rn \quad & \text{if } l\le L_n\\
		\ 0&  \text{if } l>L_n.
	\end{cases}
\end{equation} 
with $W_n$ as in \eqref{wun}, that is, for a bounded function $g$ on $[0,1]$, 
\begin{equation} 
	W_n(g) =W_n(X;g) = \frac{1}{\sqrt{n}} \sum_{i=1}^n[\delta_i g(Y_i) - \La_0g(Y_i)]. 
\end{equation}

For $z\in \mathcal M_0$, the map $\tau_z: \la \mapsto \sqrt n (\la- z)$ maps $\mathcal M_0 \to\mathcal M_0$, and below we consider the shifted posterior $\Pi[\cdot\given X] \circ \tau_{T_n}^{-1}$, with centering $T_n$.

Let us recall condition \ref{c:cvarsn}:  for suitable directions $b$, and $A_n$ as in \ref{c:l1},
\begin{enumerate}[label=\textbf{(T)}]
	\item
	with $r = \log{\lambda}, r_0 = \log\lambda_0, r_t^n = r - \frac{t}{\sqrt{n}}\psi_{b,L_n}$ and $C_1>0$, suppose 
	\begin{equation*}
		\log\frac{\int_{A_n} e^{\ell_n(r_t^n) - \ell_n(r_0)} d\Pi(r)}
		{\int e^{\ell_n(r) - \ell_n(r_0)} d\Pi(r)} \le C_1(1+t^2)
	\end{equation*}
	holds for any $|t|\le \log{n}$.
\end{enumerate}

In the next statement, $\Pi_{L_n}[\cdot\given X]$ denotes the posterior distribution on $\la$ projected  onto the first $L_n$ levels of wavelet coefficients. Let us also recall the definition of $\cV_L$ in \eqref{cve}.

\begin{thm}\label{thm:bvm}
	Let  $X=(X_1, \dots, X_n)$ be a sample of law $P_0$ with hazard rate $\la_0$ under conditions \ref{c:mod}. 
	Let $\mathcal M_0=\mathcal M_0(w)$ for some $w_l\uparrow\infty$ with $w_l\ge l$.
	Let $T_n$ be as in \eqref{tn}. Suppose the prior $\Pi$ is such that \ref{c:l1}--\ref{c:sn} are satisfied with cut--off $L_n$ and rate $\veps_n$ verifying 
	\[ \rn\veps_n2^{-L_n} = o\left( \min_{l\le L_n} \left\{  l^{-1/4} 2^{-l/2} w_l \right\}\right).\] 
	Suppose  \ref{c:cvarb} is satisfied for any  $b\in \cV_{\cL}$ and any fixed  $\cL\ge 0$, and that \ref{c:cvarsn} holds uniformly for $b=\psi_{LK}$ with $L\le L_n,$ $0\le K<2^L$. 
	
	{\em (Case $1$)}. If  $\la_{lk}=\psg \la,\psi_{lk} \psd=0$ for $l>L_n$ under $\Pi[\cdot\given X]$, then  
	\begin{equation} \label{iidbvm}
		\cB_{\mathcal M_0}(\Pi[\cdot\given X]\circ\tau_{T_n}^{-1}, \mZ_{Q_0})\to^{P_{0}} 0,
	\end{equation}
	where $\mZ_{Q_0}$ is as in \eqref{def:gpz} and $\cB_{\cM_0}$ is the bounded--Lipschitz metric on $\cM_0$. 
	
	{\em (Case $2$)}. If the posterior distribution does not set all $\la_{lk}$ to $0$, then \eqref{iidbvm} continues to hold for the projected posterior $\Pi_{L_n}[\cdot\given X]$. 
	It also holds for the original posterior $\Pi[\cdot\given X]$ provided  $w_{L_n}^{-1}\zeta_n2^{-L_n/2}+\|\la_0^{L_n^c}\|_{\cM(w)}=O(1/\rn)$, where $\la_0^{L_n^c}=\la_0-P_{L_n}\la_0$.
\end{thm}

The conditions on rates in Theorem \ref{thm:bvm} are  satisfied for typical prior choices as seen in Theorems \ref{thm:appli} and \ref{thm:applifull}, which shows that it is enough for the prior to `undersmooth', a condition that can even be further weakened by taking $(w_l)$ increasing significantly faster to infinity than $l$.

\subsection{Proof of Theorem \ref{thm:bvm}} \label{sec:proof:bvmnp}

Let us set $\bar{w}_l=w_l/l^{1/4}$. The sequence $(\bar{w}_l)$ verifies $w_l/\bar{w}_l\uparrow\infty$ and $\bar{w}_l\ge \sqrt{l}$. One first notes that the assumption on $(w_l)$ implies
\begin{equation} \label{techrates}
	\rn\veps_n2^{-L_n}\leqa \bar{w}_l2^{-l/2}. 
\end{equation} 

We follow the approach of the proof of Theorem 3 in \cite{cn14}: by Proposition 6 in \cite{cn14},  it is enough to prove tightness in  $\cM_0(\bar{w})$ (with $\bar{w}=(\bar{w}_l)$ such that $w_l/\bar{w}_l\uparrow\infty$ and $\bar{w}_l\ge \sqrt{l}$, which is the case for our choice of $\bar{w}$ above) as well as convergence of finite--dimensional distributions. We first deal with Case 1, that is we  assume that all $\la_{lk}=\psg \la,\psi_{lk}\psd$ for $l>L_n$ are zero under the posterior distribution. 
Let us note that by definition of $T_n$ in \eqref{tn},
\begin{equation*} 
	T_n =\la_{0,L_n}+\frac{1}{\rn}\sum_{L\le L_n} \sum_{0\le K<2^L} W_n(\psi_{LK}/M_0) \psi_{LK}. 
\end{equation*}
Let us start with tightness, proceeding similarly as in 5.4 (ii) of \cite{cn14}, taking $T_n$ as the  centering, and denoting by $E_X, P_X$ expectation and probability under $\Pi[\cdot \mid X,D_n]$, for $D_n$ as in \eqref{den}. In Case 1, only  frequencies for $l\le L_n$ are relevant and for $M > 0$ and $z_l:=\bar{w}_l/\sqrt{l}$,
\begin{align*}
	\sqrt{n}E_X[ \|\la &- T_n\|_{\cM_0(\bar w)}] \leq M+ \int_M^\infty P_X(\sqrt{n} \|\la -T_n\|_{\cM_0(\bar w)} > u) \ du\\
	&= M + \int_M^\infty P_X\left(\max_{l \leq L_n}\, \bar{w}_l^{-1} \rn \max_k |\psg \la-T_n,\psi_{lk} \psd_2| > u \right)\ du\\
	&\leq M + \sum_{l \leq L_n} \sum_{k < 2^l} \int_M^\infty P_X\left(z_l^{-1} \rn |\psg \la-T_n,\psi_{lk} \psd_2| > \sqrt{l} u\right) \ du\\
	&\leq M + \sum_{l \leq L_n} \sum_{k < 2^l} \int_M^\infty e^{-\sqrt{l}\cdot\sqrt{l} u} E_X\left[e^{\sqrt{l}\cdot z_l^{-1}\sqrt{n}  |\langle \la-T_n,\psi_{lk} \rangle_2| }\right] \ du,
\end{align*}
where the last line follows from Markov's inequality. 
We now wish to apply Proposition \ref{prop:lap} with $b=\psi_{lk}$ and $t=\sqrt{l}/z_l$. By Lemma \ref{lem:checkpsi}, we have $\|\psi_b-\psi_{b,L_n}\|_\infty\leqa 2^{l/2-L_n}$  
so that, for $l\le L_n$,
\begin{align*}
	|t|\rn \veps_n \| \eif - \psin\|_\infty & \leqa |t|\rn \veps_n 2^{-L_n}2^{l/2}.
\end{align*}
By using $t=\sqrt{l}/z_l$ and  $\rn\veps_n2^{-L_n}\le  \bar{w_l}2^{-l/2}=z_l\sqrt{l}2^{-l/2},$ which follows by combining \eqref{techrates} and the definition of $z_l$, one obtains that the last display is bounded by $Cl$.
Note that this bound  holds for $b=\psi_{lk}$, uniformly  for $l\le L_n$ and $0\le k<2^l$. Proposition \ref{prop:lap} implies that for 
$c_1, c_2$ independent of $l,k$, for $t$ as above, 
\[ 
E_X\left[ e^{t\sqrt{n} \langle \la - T_n, \psi_{LK}\rangle}\right] \leq c_1 e^{c_2 l}(1 + \op),
\]
where the $\op$ is  uniform in $l,k$. 
This results in
\begin{align*}
	\sqrt{n}E_X[ \|\la - T_n\|_{\cM_0(\bar w)}] &\lesssim M +  \sum_{l \leq L_n} \sum_{k < 2^l} 2e^{2c_2l}\int_M^\infty e^{-l u} \ du(1+\op)\\
	&\leqa M +  \sum_{l \leq L_n} 2^l 2e^{2c_2l}\int_M^\infty e^{-lu} \ du(1+\op),
\end{align*}
where the sum  is bounded by a constant for $M$ large enough. So we conclude:
$$ E_X[ \|\la - T_n \|_{\cM_0(\bar w)}] = O_{P_0}(n^{-1/2}).$$

Now that tightness is established, one now wishes to check that BvM holds for finite--dimensional projections. By Cram\'er--Wold, it is enough to do so  for $\langle \la, \psi_T \rangle_2$ with $\psi_T := \sum_{(l,k) \in T} t_{lk}\psi_{lk}$, for any finite set of indices $T$ and $t_{lk}$ any values in $\mathbb{R}$.  

It is enough to check that one can apply Theorem \ref{thm:bvmlin} for the functional representer $b=\psi_T$, as this guarantees the BvM theorem holds for the linear functional $\psg b,\la\psd$. By assumption, \ref{c:l1}--\ref{c:sn} are satisfied, and also  \ref{c:cvarb} for any  $b=\psi_T$, as $\psi_T\in \cV_{\cL}$ for large enough  $\cL\ge 0$, so it remains to check that \ref{c:funb} is verified for $b=\psi_T$. This holds  by invoking Lemma \ref{lem:checkpsi} with $L=\cL$ bounded, which gives $\|\psi_b-\psin\|_\infty\leqa 2^{-L_n}$.  As $\rn\veps_n2^{-L_n}=o(1)$ follows from the  assumption on $\veps_n$ (by bounding the minimum in the condition from above by the first term $l=1$), this concludes the proof in the case $\la_{lk}=0$ for $l>L_n$ under the posterior.

We now deal with Case 2.  The argument for finite-dimensional distributions is unchanged. For the tightness argument, one notes
\begin{align*}
	\sup_{l>L_n} &\max_{0\le k<2^l} w_l^{-1} |\psg \la-T_n,\psi_{lk}\psd|
	=  \sup_{l>L_n}  \max_{0\le k<2^l} w_l^{-1}|\psg \la,\psi_{lk}\psd|\\
	& \le w_{L_n}^{-1} \sup_{l>L_n}  \max_{0\le k<2^l} \|\la-\la_0\|_\infty \|\psi_{lk}\|_1 + \|\la_0^{L_n^c}\|_{\cM(w)},
\end{align*}
which is bounded by $w_{L_n}^{-1}  2^{-L_n/2}\zeta_n+ \|\la_0^{L_n^c}\|_{\cM(w)}$ on $D_n$. Noting that the $\cM(w)$--norm is the maximum of the last display and of the corresponding quantity with $l\le L_n$, for which the arguments for Case 1 apply,  concludes the proof of Theorem \ref{thm:bvm}. \qed

\subsection{Donsker's theorem for general priors} We now state a generalisation of Theorem \ref{thm:donsker} presented in the main paper. As its statement suggests, its proof quite directly follows from the nonparametric BvM Theorem \ref{thm:bvm}. 

Define, for a given centering $T_n\in L^2$, its primitive $\mathbb{T}_n(t)=\int_0^t T_n(u)du$.

\begin{thm} \label{thm:donskergen} 
	Let $\Pi$ be a prior on hazards as in \eqref{priorgen} supported in $L^2$ and suppose the nonparametric Bernstein--von Mises \eqref{iidbvm}  holds true in $\mathcal M_0(w)$ for some sequence $(w_l)$ such that $\sum_l w_l 2^{-l/2}<\infty$, and centering   $T_n\in L^2$.
	
	Let $\cL(\La\in \cdot\given X)$  denote the distribution induced on the cumulative hazard $\La$ when  $\la \sim \Pi[\cdot\given X]$. 
	
	Let $G_{\La_0}(t) = W(U_0(t))$ with $W$ Brownian motion and $U_0(t)=\int_0^t (\la_0/M_0)(u)du$. Then, with $\mathbb{T}_n(t)=\int_0^t T_n$,  as $n\to \infty$,
	\begin{equation} \label{funlim}
		\mathcal{B}_{\cC[0,1]}\left(\,\cL(\rn(\La-\mT_n)\given X)\ ,\, \cL(G_{\La_0})\right) \to^{P_{0}} 0.
	\end{equation}
	
\end{thm}
Theorem \ref{thm:donskergen} on the cumulative hazard is obtained by combining the nonparametric BvM theorem for the hazard, Theorem \ref{thm:bvm}, with the fact that `integration' is a continuous mapping, as we see in the next subsection.

\subsection{Proof of Theorem \ref{thm:donskergen}} \label{sec:proofdonsker}
One proceeds as in \cite{cn14}, proof of Theorem 4, by considering the `integration' map
\begin{equation} \label{defl}
	L:\{h_{lk}\} \mapsto L_t(\{h_{lk}\}):=\sum_{l,k} h_{lk} \int_0^t \psi_{lk}(x)dx, ~t \in [0,1],
\end{equation}
which is shown in \cite{cn14}, p. 1955 to be  linear and continuous from $\cM_0(w)$ to $L^\infty([0,1])$ (and also $\cC[0,1]$). The continuous mapping theorem applied to $L$ and $\|\cdot\|_\infty\circ L$ implies the two claimed convergences in distribution, upon checking that the limiting distribution under the map $L$, that is $\mZ_{P_0}\circ L^{-1}$, coincides with $G_{\La_0}$,  which follows from Lemma \ref{lemrkhs}. \qed 

\begin{lem} \label{lemrkhs}
	The Gaussian processes $[0,1]\ni t\to G_{\Lambda_0}(t)=W(U_0(t))$ and $[0,1]\ni t\to \mZ_{P_0}\circ L_t^{-1}$  coincide, where $L_t$ is the integration map \eqref{defl}.
\end{lem}
\begin{proof}
	As both are centered Gaussian processes, the result follows by checking that their respective RKHS coincide.
\end{proof}

\subsection{Proof of Theorem \ref{thm:donsker}}
Let us note that Theorem \ref{thm:donsker} is in fact (almost) a special case of Theorem \ref{thm:donskergen}. Indeed, under the conditions of Theorem \ref{thm:donsker}, the condition on rates in the statement of the BvM Theorem \ref{thm:bvm} is satisfied if one takes $w_l=2^{l/2}/(1+l^2)$ since that condition asks,  for this choice of $(w_l)$,
\[ \rn\veps_n2^{-L_n}=o(L_n^{-9/4}),\]
which is certainly satisfied if $\rn\veps_n2^{-L_n}=O(L_n^{-3})$ as assumed. By Theorem \ref{thm:bvm} and since one considers histogram priors (Case 1 of Theorem \ref{thm:bvm}), we deduce that the nonparametric BvM Theorem \eqref{iidbvm} holds, with $\sum_l w_l 2^{-l/2}<\infty$ by construction. Hence Theorem \ref{thm:donsker} follows by applying Theorem \ref{thm:donskergen}.

\section{Supremum norm results}

\subsection{Generic $\|\,\cdot\,\|_\infty$--bound}

Let us denote, for $L_n$ the prior's cut--off, the rate $\veps_n$ as in Condition \ref{c:l1}, and $\be>0$,
\begin{equation} \label{rate:first}
	\La_n= 2^{L_n}\veps_n + 2^{-\be L_n}.
\end{equation}
Let us define the $\ell_\infty$--metric between bounded functions $f,g$ as
\begin{equation} \label{ellinfty}
	\ell_\infty(f,g):= \sum_{l} 2^{l/2}\max_{0\le k<l} 
	|\psg f-g,\psi_{lk}\psd|. 
\end{equation}
The following standard bound follows from the localisation property (W3) of the wavelet basis, for bounded functions $f,g$,
\[ \|f-g\|_\infty \leqa \sum_{l} 2^{l/2} \max_{k} |\psg f-g,\psi_{lk}\psd|
=\ell_\infty(f,g). \]
\begin{lem} \label{lem:snf}
	Suppose $\la_0\in \cH(\be,D)$ for some $\be,D>0$ and that \ref{c:l1} holds with rate $\veps_n$. Then for $\La_n$ as in \eqref{rate:first},
	\[ \Pi[\|\la-\la_0\|_\infty > \La_n\given X] =\op. \]
	The same results also holds for the $\ell_\infty$--metric. 
\end{lem}
\begin{proof}
	Let us recall the notation $g_{L_n}=P_{L_n}g$, the $L^2$--projection of a given function $g$ onto the subspace $\cV_{L_n}$ generated by the first $L_n$ levels of wavelet coefficients. For any $x\in[0,1]$, 
	\begin{align*}
		(\la - \la_0)(x) & = \sum_{l\le L_n}\sum_{k=0}^{2^l-1} (\la_{lk}-\la_{0,lk})\psi_{lk}(x)- \sum_{l> L_n}\sum_{k=0}^{2^l-1} \la_{0,lk}\psi_{lk}(x) \\
		& = (\la_{L_n}-\la_{0,L_n})(x) + R(x),
	\end{align*}
	where $|R(x)|\le  \sum_{l>L_n} 2^{l/2}\max_{k}|\la_{0,lk}|\le \sum_{l>L_n} 2^{-l\be}\leqa 2^{-L_n\be}$ using the H\"older condition on $\la_0$
	and  with $g_{lk}=\psg g , \psi_{lk} \psd$ for a given  $g\in L^2[0,1]$. Also,
	
	\begin{align*}
		&\left|\la_{L_n}(x) -\la_{0,L_n}(x)\right| 
		\le
		C \ell_\infty(\la_{L_n},\la_{0,L_n})
		= C\sum_{l\le L_n} \max_{k}|\la_{lk}-\la_{0,lk}| 2^{l/2}\\
		& \quad \leqa \|\la_{L_n}-\la_{0,L_n}\|_2 \sqrt{ \sum_{l\le L_n} 2^l} \leqa \|\la-\la_0\|_2 2^{L_n/2} \\
		& \quad 
		\leqa \|\la-\la_0\|_\infty^{1/2} \|\la-\la_0\|_1^{1/2} 2^{L_n/2}.
	\end{align*}
	Reinserting this in the first identity on $(\la-\la_0)(x)$ above and taking the supremum in $x$ one gets, invoking the $\|\cdot\|_1$ concentration of the posterior,
	\[ \|\la-\la_0\|_\infty \leqa \sqrt{ \|\la-\la_0\|_\infty }\sqrt{\veps_n} 2^{L_n/2} + 2^{-L_n\be}. \]
	To bound the first term on the right-hand side, one uses the inequality $2ab\le a^2/c+cb^2$, for any $a,b,c>0$, to obtain, for suitably large $c$,
	\[ \|\la-\la_0\|_\infty \leqa \veps_n 2^{L_n} + 2^{-L_n\be}, \]
	as required. 
	
	We  note that the previous bounds also hold for the related $\ell_\infty$--norm. 
	Indeed, one can reproduce the previous bounds starting directly from $\ell_\infty(\la,\la_0)$ instead of $(\la-\la_0)(x)$, leading to 
	$\ell_\infty(\la,\la_0) \leqa \La_n$.
\end{proof}

\subsection{Proof of Theorem \ref{thm:sn}, part (a)}

As \ref{c:l1} holds, one can invoke Lemma \ref{lem:snf} to obtain a first supremum norm bound $\La_n$ under the posterior (note that, except for high regularities $\be$, this does not yet entail posterior consistency, i.e. $\La_n$ may go to $\infty$ with $n$).

Recall that $P_{L_n}\la_0=\la_{0,L_n}$ denotes the $L_2$--projection of $\la_0$ onto $\cV_{L_n}$, and for any $l,k$, let $\psi_{lk,n}= P_{L_n}(\psi_{lk}/M_0)$ and define a function $\la^*= \la^*(L_n)$ by the sequence of its wavelet coefficients
\begin{equation} \label{lastar}
	\psg \la^* , \psi_{lk} \psd =
	\begin{cases}
		\ \psg \la_0 , \psi_{lk} \psd+W_n\left(\psi_{lk,n} \right)/\rn \quad & \text{if } l\le L_n\\
		\ 0&  \text{if } l>L_n,
	\end{cases}
\end{equation}
for any $k$ and where $W_n$ is defined in \eqref{wun}. We show in Section \ref{sec:backgroundeff}, Lemma \ref{lem:tn}, that $T_n$ and  $\la^*$ are very close, so that one can indifferently consider a centering $T_n$ or $\la^*$. 

As part (a) is about the projected posterior, it is enough to consider $\la_{L_n}=P_{L_n}\la$. More precisely, by definition of $\Pi_{L_n}[\cdot\given X]$, for any given rate $r_n$, 
\begin{align*}
	\Pi_{L_n}[\|\la-\la_0\|_\infty >r_n\given X] & =\Pi_{L_n}[\|\la_{L_n}-\la_0\|_\infty >r_n\given X] \\
	& = \Pi[\|\la_{L_n}-\la_0\|_\infty >r_n\given X],
\end{align*}
as indeed the distribution of $\la_{L_n}$ is the same under $\Pi_{L_n}[\cdot\given X]$ and $\Pi[\cdot\given X]$. Once this is noted, it is enough to work with the original posterior, but considering only the first $L_n$ wavelet levels.

We bound:
\begin{align*}
	\|  \la_{L_n} - \la_0 \|_\infty \leq 
	\underbrace{\| \la_0 - \la_{0,L_n} \|_\infty}_{=(i)}
	+ \underbrace{\| \la_{0, L_n} - \la_{L_n}^* \|_\infty}_{=(ii)} 
	+ \underbrace{\|\la_{L_n}^*-\la_{L_n}  \|_\infty}_{=(iii)}
\end{align*}
and study the expectation under $\la_0$ of the posterior expectation of each component separately. We start by terms (i) and (ii), that do not depend on the posterior and are shown to go to zero at the required rate.

\emph{Part (i).} This part is deterministic. Recalling the notation $\la_{0,lk}=\langle \la_0, \psi_{lk}\rangle_2$,
\begin{align*}
	\| \la_0-\la_{0,L_n} \|_\infty & \leq  \sum_{l=L_n + 1}^{\infty} \max_{0 \leq k < 2^l -1} |\la_{0,lk}| \left\| \sum_{k=0}^{2^l - 1} |\psi_{lk}| \right\|_\infty\\
	&\lesssim \sum_{l=L_n + 1}^{\infty}2^{-l(1/2+\beta)}2^{l/2} = \sum_{l=L_n + 1}^{\infty}2^{-l\beta} \lesssim 2^{-L_n\beta},
\end{align*}
where we have used that $\la_0\in \cH(\beta, D)$ for some $D > 0$.

\emph{Part (ii).} By Lemma \ref{lem:parttwo}, the expectation $E_{\la_0} \|\la_{L_n}^* - \la_{0,L_n}\|_\infty$ is bounded above by a constant multiple of $(L_n 2^{L_n}/n)^{1/2}$. This implies that the sum $(i)+(ii)$ is a $O_{P_0}(\veps_n^{\be,L_n})$.

\emph{Part (iii).}
For any $\la\in \cV_{L_n}$, one bounds $\la-\la_{L_n}^*$ from above as follows
\begin{align*}
	\|\la-\la_{L_n}^*\|_\infty &  \le \Big| \sum_{l=0}^{L_n}  \sum_{0\le k<2^l} \psg \la-\la^*,\psi_{lk}\psd \psi_{lk}  \Big|\\
	&  \le \sum_{l=0}^{L_n} \max_{0\le k<2^l} |\psg \la-\la^*,\psi_{lk}\psd| \Big\| \sum_{0\le k<2^l} |\psi_{lk}| \Big\|_\infty \\ 
	&\le \frac{1}{\sqrt{n}}\sum_{l=0}^{L_n} 2^{l/2}\sqrt{n}\max_{0\le k<2^l} |\psg \la-\la^*,\psi_{lk}\psd|.
\end{align*}
Let us define the sets $E_n$ as, with $A_n$ as in \ref{c:l1} and $\La_n$ as in \eqref{rate:first},
\begin{equation} \label{en}
	E_n = A_n \cap \{ \la:\, \|\la-\la_0\|_\infty\le \La_n\}.
\end{equation}
Let $\Pi_n=\Pi[\cdot\given X,E_n]$ be a shorthand for the posterior distribution conditioned on $E_n$, and let $E^{\Pi_n}$ be the expectation under $\Pi_n$.  Bounding the expected maximum by Laplace transforms following the approach of
\cite{ic14}--\cite{cn14}, we find for any $t > 0$:
\begin{align*}
	E^{\Pi_n}  & \max_{0\le k<2^l} \sqrt{n} |\psg \la-\la^*,\psi_{lk}\psd| \\
	& \le \frac{1}{t} \log\left(\sum_{k=0}^{2^l - 1}  E^{\Pi_n}\left[e^{t\sqrt{n}\psg \la-\la^*,\psi_{lk}\psd} +  e^{-t\sqrt{n}\psg \la-\la^*,\psi_{lk}\psd} \right]  \right).
\end{align*}
We now follow the steps of the proof of Theorem \ref{thm:bvmlin}, but in place of 
$A_n$ we work instead on $E_n$. Similar to  \eqref{laphist}, we have, in terms of the sets $E_n$, and setting $b=\psi_{LK}$ as a shorthand, recalling $\psi_b=b/M_0$ and that $\psin$ is its projection onto $\cV_{L_n}$,
\begin{align} \label{laplacetransfo}
	& E\left[ e^{t\sqrt{n}\psg \la-\la^*,\psi_{LK}\psd} \given X, E_n \right] \cdot e^{-\frac{t^2}{2} \| \psin \|_L^2}\Pi[E_n\given X]  \notag \\
	&\quad=\frac{\int_{E_n} e^{  \ell_n(r_t^n) - \ell_n(r_0) + R_n(r, r_0) - R_n(r_t^n, r_0) +t\sqrt{n}[\psg r-r_0,\eif-\psin \psd_L-B(\la, \la_0)]} d\Pi(r) } 
	{\int e^{\ell_n(r) - \ell_n(r_0)} d\Pi(r)}.
\end{align}
Now we rearrange the term in brackets similarly as in the proof of Theorem  \ref{thm:bvmlin} by introducing the term $B_n(\la,\la_0)$ as below \eqref{laphist}. 

On the set $A_n$, we have 
\begin{align*}
	|\psg \frac{\la-\la_0}{\la_0},\eif-\psin \psd_L| &  \leqa \veps_n \|\eif-\psin\|_\infty \leqa \veps_n 2^{L/2-L_n},
\end{align*}
where one uses Lemma \ref{lem:psi}. This shows that the last term in the last display is bounded from above by $\veps_n 2^{-L_n/2}$ uniformly over $L\le L_n$ and $K$. As $\veps_n\leqa  \veps_{n}^{\be,L_n}$ by assumption, 
one deduces
\[ t \rn  | \psg \frac{\la-\la_0}{\la_0},\eif-\psin \psd_L |
\leqa t \rn 2^{-L_n/2} \veps_{n}^{\be,L_n}.
\]
We now focus on bounding $|R_n(r, r_0) - R_n(r_t^n, r_0)-t\rn B_n(\la,\la_0)|$ from above. The control of this term is similar in spirit to that in the proof of Proposition \ref{prop:lap}. We include it for completeness. Note that both bounds on the $R_{n,1}$ and $R_{n,2}$ parts as obtained there do not use supremum--norm consistency of the posterior, but only an upper-bound on $\|\la-\la_0\|_\infty$, so one may reuse these bounds here  replacing $\zeta_n$ by the generic `rate' (or rather, bound, as it does not go to $0$ in general) $\La_n$ obtained in Lemma \ref{lem:snf} (or for future use in further iterations by any such bound on $\|\la-\la_0\|_\infty$).

Let us first deal with the terms $R_{n,1}$, recalling the notation 
\begin{align*}
	R_{n,1}&(r, r_0) - R_{n,1}(r^n_t, r_0)\\
	&= \mathbb{G}_n\left\{t(\La_0-\La)\psin(\cdot) + \sqrt{n}\int_0^\cdot e^r\left( e^{- \tfrac{t}{\sqrt{n}}\psin} - 1 + \frac{t}{\sqrt{n}}\psin \right) \right\},
\end{align*}
which we write $t \mathbb{G}_n f_n + \mathbb{G}_n g_n$, setting $f_n = (\La_0-\La)\psin$ as well as  
$g_n = \sqrt{n}\int_0^\cdot e^r( e^{- \tfrac{t}{\sqrt{n}}\psin} \\ - 1 + \frac{t}{\sqrt{n}}\psin )$.

First, using Lemma \ref{lem:checkpsi}, we have $\|\psin\|_\infty\leqa  
L_n\|b\|_\infty\leqa L_n2^{L_n/2}$  as well as $\|\psin\|_2\leqa \| b\|_2\leqa 1$, so that one can set $\mu_n=CL_n2^{L_n/2}$ for some $C>0$. Second,  the hazard $\la$ belongs to the set $\cL_n$ in \eqref{elln} once setting $v_n=\La_n$. This shows $f_n\in\cF_n$ and $g_n\in\cG_n$, for $\cF_n, \cG_n$ as in \eqref{fn}--\eqref{gn}. By Lemma \ref{lem:rn1}, noting that $|t|\mu_n/\rn$ is bounded if $|t|\le \log{n}$,
\[  \sup_{f_n\in\cF_n,\, g_n\in\cG_n}  |\mathbb{G}_n(tf_n + g_n)|  =O_{P_{\la_0}}\left(|t|\La_n+ \frac{t^2}{\rn}(1+\La_n) \right). \]
One now has to bound $R_{n,2}(r, r_0)  - R_{n, 2}(r^n_t, r_0) -  t\sqrt{n}B_n(\la, \la_0)$, which using Lemma \ref{lem:rn2} is bounded by 
\[ |R_{n,2}(r, r_0)  - R_{n, 2}(r^n_t, r_0) -  t\sqrt{n}B_n(\la, \la_0)|= O\left(t^2\left\{\La_n + |t| \frac{1+\La_n}{\rn} \right\}\right).\]

As by assumption $|t|/\rn$ is bounded (in fact goes to $0$ fast), one concludes that 
$|R_n(r, r_0) - R_n(r_t^n, r_0) - t\sqrt{n}B(\la, \la_0)| \leqa t^2\La_n+O_P(|t|\{1+\La_n\})$. 
Reinserting these bounds into the Laplace transform expression \eqref{laplacetransfo} and using that $\|\psin\|_L$ is bounded by Lemma \ref{lem:psi} and that $\Pi[E_n\given X]=\op$ leads to, with $y_n=:=\rn 2^{-L_n/2} \veps_{n}^{\be,L_n}$,
\begin{align*} 
	E&\left[ e^{t\sqrt{n}\psg \la-\la^*,\psi_{LK}\psd} \given X, E_n \right] \\
	& = (1+\op)  e^{C[(1+\La_n) t^2 +(\sqrt{L_n}+y_n)t] + O_{P_0}(t\{1+\La_n\})} \frac{\int_{E_n} e^{  \ell_n(r_t^n) - \ell_n(r_0)} d\Pi(r) } 
	{\int e^{\ell_n(r) - \ell_n(r_0)} d\Pi(r)} \\
	& \leqa (1+\op)  Ce^{C[(1+\La_n) t^2 +(\sqrt{L_n}+y_n)t] + O_{P_0}(t\{1+\La_n\})},
\end{align*}
where one uses assumption \ref{c:cvarsn} noting that the numerator in that assumption bounds from above the numerator in the last display, as $E_n\subset A_n$.

Now inserting these bounds on remainders within the Laplace transform argument one gets, with $t=t_l=\sqrt{l}$ as before, and $\Pi_n=\Pi[\cdot\given X,E_n]$ the  conditioned posterior as before, 
\begin{align*} 
	\int & \|\la-\la^*\|_\infty d\Pi_n(\la)  \\
	& \le \frac{1}{\sqrt{n}} \sum_{l=0}^{L_n} 2^{l/2} \frac{1}{t_l}\log\left\{ \sum_{k=0}^{2^{l}-1} 2(1+\op)  Ce^{C[(1+\La_n) t^2 +(\sqrt{L_n}+y_n)t] + O_{P_0}(t\{1+\La_n\})} \right\} \\
	&\leqa  \frac{1}{\sqrt{n}} \sum_{l=0}^{L_n} 2^{l/2} \frac{1}{t_l}\left[
	O_{P_0}(1)+l+(1+\La_n) t_l^2 +(\sqrt{L_n}+y_n)t_l + O_{P_0}(t_l\{1+\La_n\}) \right]\\
	& \leqa O_{P_0}(1) \sum_{l=0}^{L_n} \frac{2^{l/2}}{\rn}(1+\La_n)
	+  \sum_{l=0}^{L_n} \frac{2^{l/2}}{\rn}(1+\La_n)t_l
	+  \frac{2^{L_n/2}}{\sqrt{n}} (\sqrt{L_n}+y_n)
	\\
	& \leqa O_{P_0}(1) \sqrt{\frac{2^{L_n}}{n}} (1+\La_n) +  \sqrt{L_n\frac{2^{L_n}}{n}}(1+\La_n) + \veps_{n}^{\be,L_n}\\
	& \leqa O_{P_0}(1) \frac{\veps_n^{\be, L_n}}{\sqrt{L_n}}(1+\La_n) + \veps_n^{\be, L_n}(1+\La_n).
\end{align*} 
From this one deduces using Markov's inequality that for $M_n\to\infty$ arbitrary,
\[ \Pi_{L_n}[\|\la-\la^*\|_\infty>M_n\La_n^{(1)}\given X]=\op,\]
where we have set 
\[ \La_n^{(1)} := \veps_n^{\be, L_n}(1+\La_n). \] 
Combining this with steps (i)--(ii), and using Markov's inequality to get 
$P_0(\|\la^*-\la_0\|_\infty>M_n\veps_n^{\be,L})=o(1)$, for any $M_n\to\infty$, one obtains that the projected posterior contracts at rate $\La_n^{(1)}$ around $\la_0$, for arbitrary $M_n\to\infty$,
\[  \Pi_{L_n}[\|\la-\la_0\|_\infty>M_n\La_n^{(1)}\given X]=\op. \]  

As $\veps_n^{\be, L_n}=o(1)$, one observes that 
\[ \veps_n^{\be,L_n} \leqa \La_n^{(1)} =o(\La_n), \] 
and the new obtained rate is  faster than $\zeta_n$.

We can now reproduce identically the argument of this subsection, but now using the improved rate $\La_n^{1}:=M_n\La_n\veps_n^{\be,L_n}$ for $\|\la-\la^*\|_\infty$ in \eqref{en} (with given arbitrary diverging $M_n\to\infty$), which once we apply the Laplace transform argument again leads to, for arbitrary $M_n\to\infty$, 
\[ \Pi_{L_n}[\|\la-\la_0\|_\infty>M_n\La_n^{(2)}\given X]=\op,\]

where we have set $\La_n^{(2)} := \veps_n^{\be, L_n}(1+\veps_n^{\be, L_n}\La_n)$. 
Further iterating the argument, one obtains the rate, for fixed given $p\ge 1$,
\[ \La_n^{(p+1)} :=  (1+(\veps_n^{\be, L_n})^{p} \La_n)\veps_n^{\be, L_n}. \]
Noting that $\La_n$ from \eqref{rate:first} verifies $\La_n\le 2^{L_n}$, one sees, as $2^{L_n}$ is a given power of $n$, that for an integer $p$ large enough, we have $(\veps_n^{\be, L_n})^{p} \La_n=O(1)$, so that the overall obtained rate at that iteration is $\veps_n^{\be, L_n}$ as requested, which concludes the proof of Theorem \ref{thm:sn}, part (a).

\subsection{Proof of Theorem \ref{thm:sn}, part (b)}

Note that 
\[ r-r_0 = \log(\la/\la_0)=\log\left[1+\frac{\la-\la_0}{\la_0} \right]. \]
Therefore a rate  $\|r-r_0\|_\infty=o(1)$ automatically translates into the same rate for $\|\la-\la_0\|_\infty$ and vice-versa. Within this proof, we set  
\[ \zeta_n:=2^{L_n/2}\veps_{n}+2^{-L_n\be}.\]  
As
$\veps_n\leqa \veps_{n,\be}$ and $\be\wedge \ga>1/2$ by assumption, we have, using the explicit expressions of $\veps_{n,\be}$ and $L_n$, that $\zeta_n\asymp 2^{L_n/2}\veps_{n}$ and $\zeta_n=o(n^{-\rho})$ for some $\rho>0$. 

Combining this with Lemma \ref{lem:consistsn}, we have, on the set $A_n$,

\[ \|r-r_0\|_\infty \leqa \zeta_n\]
under the posterior distribution 
and in particular $\|r-r_0\|_\infty$ is bounded.  
Now define a centering function as, with $r_{0,L_n}=P_{L_n}r_0$,
\begin{equation} \label{centerwave}
	\tilde r_n= \tilde r_{L_n}:=  r_{0,L_n}+\frac{1}{\rn} \sum_{l\le L_n, k} W_n(\psi_{lk}), 
\end{equation} 
and proceeding similarly as in Lemma \ref{lem:parttwo}, one checks that 
\begin{equation} \label{centerpartb}
	\|\tilde r_n - r_{0,L_n}\|_\infty \leqa \ell_\infty(\tilde{r}_n,r_{0,L_n})= O_{P_0}( \veps_n^{\be,L_n}).
\end{equation}

Next one writes, recalling the expression of the LAN--norm $\|\cdot\|_L$ in \eqref{eq:LAN},
\begin{align*}
	\|r-&\tilde{r}_n\|_\infty  \le \|(r-\tilde r_n)M_0\la_0 \frac{1}{M_0\la_0}\|_\infty 
	\le \| (r-\tilde r_n)M_0\la_0 \|_\infty \|\frac1{M_0\la_0}\|_\infty \\
	& \leqa \sum_{l\ge 0} 2^{l/2} \max_{k} \left|\psg (r-\tilde r_n)M_0\la_0 , \psi_{lk} \psd_2\right| \\
	& \leqa \sum_{l\le L_n} 2^{l/2} \max_{k} \left|\psg r-\tilde r_n, \psi_{lk} \psd_L\right|
	+ \sum_{l> L_n} 2^{l/2} \max_{k} \left|\psg (r-\tilde r_n)M_0\la_0 , \psi_{lk} \psd_2\right|\\
	& \qquad =: \qquad\qquad(I) \qquad\qquad \ \ +\ \ \qquad\qquad (II).
\end{align*}
Combining Lemma \ref{lem:consistsn} with Lemma \ref{lem:biasr}, one obtains, on the set $A_n$,
\begin{align*}
	(II) & \leqa 2^{-L_n\delta} \ell_\infty(r,\tilde{r}_n)\leqa 2^{-L_n\delta}(\ell_\infty(r,r_0)+\ell_\infty(r_0,\tilde{r}_n)) \\
	& \leqa  2^{-L_n\delta} \zeta_n + 2^{-L_n\delta} O_{P_0}(\veps_{n}^{\be,L_n}),  
\end{align*}
where $\delta=1\wedge \be$. 
This shows that $(II)=O_{P_0}(\veps_{n}^{\be,L_n})$, since $\be\ge 1/2$.

We now control the  term (I) corresponding to levels $l\le L_n$. Let us recall that, setting $r_t=r-t\psi_{lk}/\rn$, one can expand the log-likelihood as follows
\begin{align}
	\ell_n(&r) - \ell_n(r_0) - [\ell_n(r_t) - \ell_n(r_0)] \notag\\
	&= \frac{t^2}{2} \|\psi_{lk}\|_L^2 - t\sqrt{n}\langle r-r_0,\psi_{lk} \rangle_{L} + tW_n(\psi_{lk}) + R_n(r, r_0) - R_n(r_t, r_0) \label{expwave}\\
	&= \frac{t^2}{2} \|\psi_{lk}\|_L^2 - t\sqrt{n}\langle r-\tilde{r}_n,\psi_{lk} \rangle_{L} + R_n(r, r_0) - R_n(r_t, r_0). \notag
\end{align}
We now control uniformly the Laplace transforms, for $\tilde r_n$ given by \eqref{centerwave}, $E[\exp\{t\rn(\psg r-\tilde{r}_n,\psi_{lk}\psd_L\}\given X]$.

This is quite similar as for part (a). There are a few differences. The bracket at stake is with $\psi_{lk}$ in terms of the LAN inner product, instead of the Euclidean inner product: we do not start with $\psg \la-\la_0,\psi_{lk}\psd_2$, but with $\psg r-r_0,\psi_{lk}\psd_L$, so the term $B(\la,\la_0)$  in particular does not cancel out when studying the remainder term. On the other hand, the semiparametric bias coming from the approximation of  $\psi_{lk}/M_0$ is not present this time.

We have the following analog of \eqref{laplacetransfo} (this time analysing $r$ rather than $\la$), for $l\le L_n$ and any admissible $k$, for $\tilde{r}_n$ as in \eqref{centerwave},
\begin{align}
	& E\left[ e^{t\sqrt{n}\psg r-\tilde r_{n},\psi_{lk}\psd_L} \mid X, A_n \right] \cdot e^{-\frac{t^2}{2} \|\psi_{lk}\|_L^2}\cdot \Pi[A_n\given X] \notag\\
	&\quad=\frac{\int_{A_n} e^{  \ell_n(r_t) - \ell_n(r_0) + R_n(r, r_0) - R_n(r_t, r_0) 
		} 
		d\Pi(r) }{\int e^{\ell_n(r) - \ell_n(r_0)} d\Pi(r)}. \label{lapnewave}
\end{align} 
Note that here there is no need to further project $\psi_{lk}$ for $l\le L_n$, as the latter already belongs to $\cV_{L_n}$. 
We now study the difference $R_n(r, r_0) - R_n(r_t, r_0)$. Recalling $R_n(r, r_0)=R_{n,1}(r, r_0)+R_{n,2}(r, r_0)$, let us start with the term $R_{n,2}$
\begin{align*}
	R&_{n,2}(r, r_0) - R_{n, 2}(r_t, r_0) \\
	& =  t\sqrt{n}B(\la, \la_0) + 
	n \La_0\left\{M_0\left(e^{r-r_0}\left(e^{- \frac{t}{\sqrt{n}}\psi_{lk}} - 1 + \frac{t}{\sqrt{n}}\psi_{lk}\right)  - \frac{1}{2}\frac{t^2}{n}\psi_{lk}^2  \right)  \right\},
\end{align*}
where this time the term $B(\la,\la_0)$ has to be studied separately
\begin{align*}
	t\sqrt{n}B(\la, \la_0) &= t\sqrt{n}\left\langle r - r_0 - \frac{\la - \la_0}{\la_0}, \psi_{lk} \right\rangle_L
	= t\sqrt{n}\left\langle r - r_0 - (e^{r-r_0} - 1), \psi_{lk} \right\rangle_L\\
	&= n\La_0\left\{ M_0\left[ r - r_0 - (e^{r-r_0} - 1) \right] 
	\frac{t}{\sqrt{n}}\psi_{lk}\right\}.
\end{align*}

Combining the posterior supremum--norm consistency following from Lemma \ref{lem:consistsn} as noted above and the inequality $|e^u-1-u|\le Cu^2$ for bounded $u$, 
\begin{align*}
	|t\rn B(\la,\la_0)|
	& \leqa \rn t\int \|r-r_0\|_\infty^2 |\psi_{lk}|  \leqa \rn t\zeta_n^2 2^{-l/2}.
\end{align*}
One now deals with $R_{n,1}(r_t,r_0)-R_{n,1}(r,r_0)$ and 
$R_{n,2}(r_t,r_0)-R_{n,2}(r,r_0)-t\rn B(\la,\la_0)$ in a similar way as we did for part (a) (and in the proof of Theorem \ref{thm:bvmlin}). The difference here is that we work with $h=\psi_{lk}$ which satisfies similar bounds as $\psin$, namely $\|\psi_{lk}\|_\infty\leqa 2^{l/2}\leqa 2^{L_n/2}$ for $l\le L_n$ and $\|\psi_{lk}\|_2\leqa 1$, so that for $t\ge 0$,
\[ |R_n(r, r_0) - R_n(r_t^n, r_0) - t\sqrt{n}B(\la, \la_0)| \leqa t^2\zeta_n+
O_{P_0}(t\{1+\zeta_n\}), \]
which combining with the bound on $B(\la, \la_0)$ above leads to, with $\zeta_n=o(1)$,
\[|R_n(r, r_0) - R_n(r_t^n, r_0)| \leqa  
(\rn2^{-l/2}\zeta_n^2)t + t^2\zeta_n + O_{P_0}(t).
\]
Following similar steps as for part (a) above, the Laplace transform method gives us,  using \ref{c:cvarsn}, that
\begin{align*} 
	E&\left[ e^{t\sqrt{n}\psg r-\tilde{r}_{n},\psi_{LK}\psd_L} \given X, A_n \right] \\
	& =   e^{C[(1+\zeta_n) t^2 + (\rn2^{-l/2}\zeta_n^2) t] + O_{P_0}(t)} \frac{\int_{A_n} e^{  \ell_n(r_t^n) - \ell_n(r_0)} d\Pi(r) } 
	{\int e^{\ell_n(r) - \ell_n(r_0)} d\Pi(r)} \\
	& \leqa e^{C[(1+\zeta_n) t^2 + (\rn2^{-l/2}\zeta_n^2) t] + O_{P_0}(t)}.
\end{align*} 

Now for $\Pi_n[\cdot\given X]=\Pi[\cdot\given X,A_n]$ the conditioned posterior and $t=t_l=\sqrt{l}$,
\begin{align*} 
	\int & \|r-\tilde{r}_{n}\|_\infty d\Pi_n(r\given X)  \\
	& \le \frac{1}{\sqrt{n}} \sum_{l=0}^{L_n} 2^{l/2} \frac{1}{t_l}\log\left\{ \sum_{k=0}^{2^{l}-1} 2e^{C[(1+\zeta_n) t^2 + (\rn2^{-l/2}\zeta_n^2) t_l] + O_{P_0}(t)}
	\right\} \\
	&\leqa  \frac{1}{\sqrt{n}} \sum_{l=0}^{L_n} 2^{l/2} \frac{1}{t_l}\left[
	l+(1+\zeta_n) t_l^2 +(\rn2^{-l/2}\zeta_n^2)t_l + O_{P_0}(t_l) \right]\\
	& \leqa O_{P_0}(1) \frac{1}{\sqrt{n}} \sum_{l=0}^{L_n} 2^{l/2}
	+  \frac{1}{\sqrt{n}} \sum_{l=0}^{L_n} 2^{l/2}\left[ (1+\zeta_n)\sqrt{l} +\rn2^{-l/2}\zeta_n^2 \right]
	\\
	& \leqa O_{P_0}(1) \sqrt{\frac{2^{L_n}}{n}}
	+  \sqrt{L_n\frac{2^{L_n}}{n}}(1+\zeta_n)+ L_n\zeta_n^2\\
	& \leqa O_{P_0}(1) \frac{\veps_n^{\be, L_n}}{\sqrt{L_n}}
	+ \veps_n^{\be, L_n}(1+\zeta_n)+(L_n\zeta_n)\zeta_n.
\end{align*} 

From this one deduces using Markov's inequality, proceeding as for part (a) above,  that for $M_n\to\infty$ arbitrary,
\[ \Pi[\|r-\tilde{r}_{n}\|_\infty>M_n \zeta_n^{(1)}\given X]=\op,\]
where we have set 
\[ \zeta_n^{(1)} := \veps_n^{\be, L_n}+ (L_n\zeta_n)\zeta_n. \] 
Combining this with \eqref{centerpartb}, one obtains that the posterior contracts at rate $\La_n^{(1)}$ around $r_0$. By using the remark at the beginning of the proof, one obtains the same for the posterior of $\la$ around $\la_0$: for arbitrary $M_n\to \infty$,
\[ \Pi[\|\la-\la_0\|_\infty>M_n\zeta_n^{(1)}\given X]=\op.\]  
Since, as noted earlier, $\zeta_n=o(n^{-\rho})$ for some $\rho>0$, we have $L_n\zeta_n=o(1)$,  so that the new rate improves upon $\zeta_n$. Let us now set
\[ E_n=A_n\cap \{\la:\ \|\la-\la_0\|_\infty\le M_n\zeta_n^{(1)}\},\]
where $M_n$ is a given (arbitrary) sequence going to $\infty$. 
By iterating the argument using $E_n$ instead of $A_n$ (and invoking $E_n\subset A_n$ just before using \ref{c:cvarsn}), one obtains the posterior rate, for $p\ge 1$,
\[ \zeta_n^{(p)} := \veps_n^{\be, L_n} \vee (L_n\zeta_n)^p \zeta_n . \]
As $\zeta_n$ decreases polynomially with $n$, for an integer $p$ large enough, the first term dominates in the maximum on the last display, so the overall obtained rate is $\veps_n^{\be, L_n}$ as requested, which concludes the proof of Theorem \ref{thm:sn}, part (b).

\begin{lem} \label{lem:biasr}
	Suppose $r_0\in  \cH(\be, D)$ for some $\be,D>0$. 
	For $\tilde{r}_n$ as in \eqref{centerwave} and $\delta= 1\wedge \be$,
	\[ \sum_{l> L_n} 2^{l/2} \max_{k} \left|\psg (r-\tilde{r}_n)M_0\la_0 , \psi_{lk} \psd_2\right|
	\leqa 2^{-L_n\delta}\ell_\infty(r,\tilde{r}_n).
	\]
\end{lem}
\begin{proof}
	For any $l>L_n$ and admissible $k$, by expanding $r-\tilde{r}_n$ onto the wavelet basis, and recalling that both $r_{LK}=0$ and $\tilde{r}_{n,LK}=0$ for levels $L>L_n$ by the definitions of the prior and of $\tilde{r}_n$,
	\begin{align*}
		\psg (r-\tilde{r}_n)M_0\la_0 , &\psi_{lk} \psd_2  = \sum_{L\le L_n, K}(r_{LK}-\tilde{r}_{n,LK})\psg \psi_{LK}M_0\la_0,\psi_{lk}\psd_2.
	\end{align*}
	Let us denote by $\overline{M_0\la_0}^{lk}$ the mean value of the function $M_0\la_0$ over the support $S_{lk}$ of $\psi_{lk}$. By Lemma \ref{lem:lip}, the function $M_0$ is a Lipschitz function, so $M_0\la_0$ belongs to $\cH(\delta,d)$ for some large enough $d$, if one sets $\delta:=1\wedge \be$.  One now bounds the inner products in the last display for $l>L_n, L\le L_n$ and admissible $k, K$ as follows, noting that $\psg \psi_{LK},\psi_{lk}\psd = 0$,
	\begin{align*}
		|\psg \psi_{LK}M_0\la_0,\psi_{lk}\psd_2| & = 
		|\psg \psi_{LK} \{M_0\la_0-\overline{M_0\la_0}^{lk}\},\psi_{lk}\psd_2|\\
		& \leqa |S_{lk}|^\delta 2^{L/2} 2^{-l/2},
	\end{align*}
	where we use $\|\psi_{LK}\|_\infty \leqa 2^{L/2}$ and $\|\psi_{lk}\|_1\le 2^{-l/2}$ and that $M_0\la_0$ is $\delta$--H\"older. Deduce, using that for a given $L\le L_n$ the support of $\psi_{LK}$ intersects that of $\psi_{lk}$ at most a constant number of times, that the quantity $\psg (r-\tilde{r}_n)M_0\la_0 , \psi_{lk} \psd_2$ is bounded in absolute value by a constant times
	\begin{align*}
		|S_{lk}|^\ga  2^{-l/2} \sum_{L\le L_n} 2^{L/2} \max_{K}|r_{LK}-\tilde r_{n,LK}|
		\le  2^{-(1/2+\delta)l} \ell_\infty(r,\tilde{r}_n),
	\end{align*}
	which gives the result by inserting this bound in the sum of the statement.
\end{proof}

\section{Examples of priors, proof of Theorems \ref{thm:appli} and \ref{thm:applifull}} \label{sec:proofappli}
We first provide examples of priors meeting the conditions required for our main results in Section \ref{sec:specificationpriors}. The results for the dependent and independent Gamma priors were already stated as Theorem \ref{thm:appli} in the main paper. In Section \ref{sec:fullthmappli} we state a result covering all four classes of priors considered in the paper, as Theorem \ref{thm:applifull}, of which Theorem \ref{thm:appli} is a special case. The proof of Theorem \ref{thm:applifull} is subsequently given first for the independent Laplace prior in Section \ref{sec:proofindeplaplace}. The modifications required for the remaining priors are described in \ref{sec:proofotherpriors}.

\subsection{Specification of priors}\label{sec:specificationpriors}

For classes {\bf (H$_3$)} and {\bf (S)}, referred to as `wavelet priors', the prior on the log--hazard $r=\log{\la}$ is given as in \eqref{priorgen}, 
\begin{align*} 
	r & = \sum_{l\le L_n,\, k}\sigma_l Z_{lk} \psi_{lk},
\end{align*}
with cut--off $L_n$ as in \eqref{eln}, $\sigma_l>0$ and $Z_{lk}$ independent random variables. We consider two common distributions for $Z_{lk}$: either a standard Laplace,  or Gaussian $\cN(0,1)$, assuming,
\begin{equation}\label{c:sig} 
	\sigma_l = 2^{-l/2} \ (0\le l \le L_n) \quad \text{or} \quad
	\sigma_l = 1 \ (0\le l \le L_n).
\end{equation}

The choices of $\sigma_l$ above are for simplicity of presentation, and either diverging or slightly larger or smaller scaling factors could be considered as well.  
Class {\bf (H$_3$)} arises by selecting $\psi_{lk}$ to be Haar, while class  {\bf (S)} arises by selecting  the smoother CDV wavelet basis.

In classes {\bf (H$_1$)} and  {\bf (H$_2$)}, the prior is a random histogram, with either dependent or independent heights. These priors may be viewed as versions of Haar wavelet priors, but  are more conveniently expressed in terms of histogram heights for the hazard itself as
\begin{equation} \label{histheight}
	\la = \sum_{k=0}^{2^{L_n+1}-1}  \la_k \1_{I_{k}^{L_n+1}},
\end{equation} 
where $\la_k$ are random heights whose distribution is specified below. 

We consider three prior distributions for  $\la_k$'s: a Gamma, log-normal and log-Laplace prior. For each distribution, we verify our conditions for independent (leading to class  {\bf (H$_1$)}) and dependent (leading to class {\bf (H$_2$)}) $\la_k$'s.

We provide details for the dependent class {\bf (H$_2$)} formulations, which follow the autoregressive idea in \cite{Arjas1994}. The priors are constructed so that the prior mean and variance on the $\lambda$-scale  satisfy , for $k = 1, \ldots, 2^{L_n+1}-1$: 
\begin{equation}\begin{aligned} \label{eq:ARstructure}
		E[\la_k \mid \la_{k-1}, \ldots, \la_0] &= \la_{k-1}\\
		\text{Var}(\la_k \mid \la_{k-1}, \ldots, \la_0) &= \sigma^2 (\la_{k-1})^2,
\end{aligned} \end{equation}
for some constant $\sigma^2>0$ to be specified.

The specification for the Gamma priors was already given in Section \ref{sec:appli}. We now list the parameter specifications for the dependent versions of the log-normal and log-Laplace priors. In the independent {\bf (H$_1$)} case, under the prior each $\la_k$ is i.i.d. with common distribution the same as that of $\la_0$  as specified below. 
\begin{enumerate}

	\item The dependent log-normal prior. With $X \sim \operatorname{LN}(\mu, \sigma^2)$ we refer to the distribution of $X=e^Y$, where $Y$ follows a normal distribution with mean $\mu$ and variance $\sigma^2$.  For some $\mu_0, \sigma_0, \sigma > 0$ to be freely chosen, the structure \eqref{eq:ARstructure} is obtained by choosing:
	\begin{align*}
		\lambda_0 &\sim \operatorname{LN}(\mu_0, \sigma_0^2)\\
		\lambda_k \mid \la_0, \ldots, \la_{k-1} &\sim \operatorname{LN}\left(\log\left(\tfrac{\la_{k-1}}{\sqrt{1+\sigma^2}}\right), \log(1+\sigma^2)\right), 
	\end{align*}
	for $k = 1, \ldots, 2^{L_n+1}-1$.

	\item The dependent log-Laplace prior. With $X \sim \operatorname{LL}(\mu, \theta)$ we refer to the distribution of $X=e^Y$, where $Y$ is Laplace distributed with  location $\mu$ and rate $\theta$. For some $\mu_0 > 0, \theta_0 > 2$ and $\sigma > 0$  to be freely chosen, the structure \eqref{eq:ARstructure} is obtained by choosing:
	\begin{align*}
		\lambda_0 &\sim \operatorname{LL}(\mu_0, \theta_0)\\
		\lambda_k \mid \la_0, \ldots, \la_{k-1} &\sim \operatorname{LL}\left( \log\left(\la_{k-1} \tfrac{g(\sigma) - \sigma^2}{g(\sigma)}\right), \sqrt{\tfrac{g(\sigma)}{\sigma^2}} \right),
	\end{align*}
	for $k = 1, \ldots, 2^{L_n+1}-1$, where $g(\sigma) = 2\sigma^2 + 1 +  \sqrt{4\sigma^4 + 5\sigma^2 + 1}$.
	
\end{enumerate}

\subsection{Full theorem statement}\label{sec:fullthmappli}
We now state the more general theorem, of which Theorem \ref{thm:appli} is a special case. Theorem \ref{thm:applifull} below provides details on all four classes of priors considered in this paper.

\begin{thm} \label{thm:applifull}
	Let  $X=(X_1, \dots, X_n)$ be a sample of law $P_0$ with hazard rate $\la_0$ under conditions \ref{c:mod}. For  $\be,L>0$, suppose $\log\la_0\in \cH(\be,L)$.
	
	Suppose the prior $\Pi$ is of the type \eqref{priorgen} with $L_n$ chosen as in \eqref{eln} with $\ga=\be$ and parameters specified as above. Then, for histogram priors (of types {\bf (H$_1$)}, {\bf (H$_2$)} or {\bf (H$_3$)}) and any $0< \be\le 1$  (except for a few examples listed at the end of this statement for which we require $\be>1/2$), the posterior distribution satisfies the nonparametric BvM theorem \eqref{iidbvm} in $\cM_0(w)$ with the choices $w_l=l$ or $w_l=2^{l/2}/(l+2)^2$ and centering $T_n$ as in \eqref{tn}. Also, 
	\[ \beta_{\mathcal{D}[0,1]}\left(\,\cL(\rn(\La-\hat\La_n)\given X)\ ,\, \cL(G_{\La_0})\right)  \to^{P_{0}} 0,\]
	for $\hat\La_n$ Nelson Aalen's estimator, as well as, for $\veps_{n,\be}^*$ as in \eqref{ratesup}, 
	\[ \Pi\left[\|\la-\la_0\|_\infty >M_n \veps_{n,\be}^* \given X\right] =\op. \] 
	For smooth wavelet priors {\bf(S)}, for any given $\be>0$, the previous results hold for the projected posterior $\Pi_{L_n}[\cdot\given X]$, provided the wavelet basis is regular enough. If $\be>1/2$, the results also hold for the original posterior distribution $\Pi[\cdot\given X]$ under the same assumptions.
	
	For the following examples of priors we require $\be>1/2$: Gaussian, Gamma and independent log--Laplace histograms. For all wavelet priors  (except Gaussian wavelets in case $\sigma_l=2^{-l/2}$) as well as for dependent log--Laplace histograms,  the results hold for any $\be>0$ (with $\be\le 1$ in the histogram case). 
\end{thm}

For simplicity, we have stated Theorems \ref{thm:appli} and Theorem \ref{thm:applifull} in the case of matched regularity $\ga=\be$. It extends to the case of arbitrary $\ga$ as follows, say first in the case of Haar wavelets: the supremum--norm rate becomes $\veps_{n}^{\be,L_n}$ as in Theorem \ref{thm:sn}. 
The nonparametric BvM theorem holds in the undersmoothing case $\ga\le \be$  when $w_l=l$ and under the weaker condition $\ga<\be+1/2$ for $w_l=2^{l/2}/(l+2)^2$. Also, the Donsker BvM theorem holds if one sets $\ga=1/2$, regardless of $\be>0$ (and more generally as soon as $\ga<\be+1/2$). This is verified along the proof of Theorem \ref{thm:appli} in Section \ref{sec:proofappli}. The case of CDV wavelets is similar, under the further condition $\ga>1/2$ if one works with the un-projected posterior distribution.

\begin{remark} \label{rem:beta}
	The condition $\be>1/2$ assumed for some examples of priors in Theorems  \ref{thm:appli} and \ref{thm:applifull} can be seen to arise from checking the change of variables conditions \ref{c:cvarsn} and \ref{c:cvarb}. In many examples, it can be removed if one allows for individual prior variances (either on histogram heights or wavelet coefficients) that go to infinity fast enough with $n$. For dependent Gaussian histogram priors for instance, by replacing $\log(1+\sigma^2)$ by $2^{L_n}\log(1+\sigma^2)$, one can check that the condition $\be>1/2$ can be removed.
\end{remark}

\subsection{Proof for independent Laplace coefficients}\label{sec:proofindeplaplace}

Let us consider the case of independent Laplace priors on coefficients $Z_{lk}$ in \eqref{priorgen}, with $(\psi_{lk})$ either the Haar basis (in which case we use $\be\le 1$), or the CDV wavelet basis.

First,  the rate condition \ref{c:l1} is verified in Section \ref{sec:l1} 
with $\veps_n\leqa \veps_{n,\be}^{L_n}$. Also, the change of variables condition \ref{c:cvarsn} for $b=\psi_{LK}$ is verified  thanks to the bounds obtained in Section \ref{sec:cvarlap}, where it is shown that it suffices to control, for $0\le t\le\log{n}$, 
\[ t\sum_{l\le L_n,\, k} \frac{|\psi_{n,lk}|}{\rn \sigma_l} \leqa
\frac{t}{\rn} \sum_{l\le L_n} \frac{1}{\sigma_l} 2^{-|l-L|/2}\leqa 
\frac{t}{\rn}(2^{L/2}+L_n2^{L/2}), 
\]
using the first bound of Lemma \ref{lem:cvarlap}. The last bound is at most $tL_n2^{L_n/2}/\rn$, which goes to $0$ if $0\le t \le \log{n}$.  One can now apply Theorem \ref{thm:sn}, which yields the supremum norm contraction rate of $\zeta_n=M_n\veps_n^{\be,L_n}$, for arbitrary $M_n\to\infty$, under no further assumptions for the Haar basis, and, for the CDV basis, either for the projected posterior, or for the full posterior with the additional condition  $2(\be\wedge \ga)>1$, that is if $\be>1/2$, $\ga>1/2$. This shows that \ref{c:sn} holds under these conditions.

We now verify the nonparametric BvM Theorem \ref{thm:bvm} with the choices $w_l=l$ and $\ga\le \be$. First, \ref{c:l1}--\ref{c:sn} hold as verified above. Then, since $\veps_n\leqa \veps_{n,\be}^{L_n}\leqa (L_n2^{L_n}/n)^{1/2}$ for $\ga\le \be$ (see above), we have $\rn\veps_n2^{-L_n/2}\le \sqrt{L_n}=o(w_{L_n})$. 

Now for $w_l=l$, we have
\[ m_n:=\min_{l\le L_n} \left\{ w_l2^{-l/2}l^{-1/4}\right\}
=w_{L_n}2^{-L_n/2}L_n^{-1/4},\] 
so the condition $\rn\veps_n2^{-L_n}=o(m_n)$ is satisfied. By the same reasoning, now using that $\zeta_n\leqa M_n\veps_{n,\be}$ for arbitrary $M_n\to\infty$ and that
$\|\la_0^{L_n^c}\|_{\cM(w)} \leqa 2^{-(1/2+\be)L_n}/w_{L_n}$, one sees that the condition required for controlling the full CDV posterior in Case 2  is also fulfilled in the case $\ga\le \be$. 
Finally, condition \ref{c:cvarb} is verified in a similar way as for \ref{c:cvarsn} above: now $b=\sum_{i=1}^N b_i \psi_{l_i k_i}$, for some $l_i\le L, k_i\le 2^{l_i}$ and $N\ge 1$ fixed. So the previous bounds and the triangle inequality imply that the bound obtained in last display two paragraphs above is bounded from above by the same quantity up to a different multiplicative constant, from which \ref{c:cvarb} is obtained. 
This shows that one can apply Theorem \ref{thm:bvm} when working on the full posterior (i.e. with Haar or CDV with $\be>1/2$). 

If one chooses instead $w_l=2^{l/2}/l^2$, one only needs the condition $\ga< \be+1/2$ to be met. Indeed, in this case $m_n:=\min_{l\le L_n} w_l2^{-l/2}l^{-1/4}=L_n^{-9/4}$, and
\[ \rn\veps_n = \sqrt{L_n2^{L_n}}+\rn2^{-L_n\be} \leqa  \sqrt{L_n}2^{L_n/2}
+ (\log{n})^\delta n^{(1/2+\ga-\be)/(1+2\ga)}, \]
for some $\delta>0$, 
and this is $o(2^{L_n}/L_n^{9/4})$ as soon as $\ga<\be+1/2$. The condition on $\zeta_n$ in Case 2 leads to the same condition, using that one can take $\zeta_n\le M_n\veps_{n,\be}$ for arbitrarily slow $M_n\to\infty$. 

Let us now check that the Donsker--BvM theorem holds for the posterior on $\La$  as soon as $\ga<\be+1/2$.  This follows by applying Theorem \ref{thm:donskergen} with the sequence $w_l=2^{l/2}/l^2$, which verifies the required summability condition as well as the nonparametric BvM in $\cM_0(w)$, as checked above. This gives a limiting result with centering at $\mathbb{T}_n$, and the Donsker--BvM with centering at $\hat\La_n$ (respectively $\hat S_n$ for $\cL(S\given X)$) by Corollary \ref{cor:donsker}. 

Finally, we justify the result stated in Remark \ref{rem:lnu}. When the cut-off is $L_n^U$, one first uses Theorem \ref{thm:sn} to obtain the rate 
$\veps_n^{\be,L_n^U}$. Indeed, Condition \ref{c:l1} with $\veps_n\leqa \veps_n^{\be,L_n^U}$ is obtained as in Section \ref{sec:l1}, using that $L_n^U 2^{L_n^U}\le \veps_n^{\be,L_n^U}$ and $2^{-L_n\be}\le \veps_n^{\be,L_n^U}$ by definition. Verification of \ref{c:cvarsn} is as before, which means that one can indeed apply Theorem \ref{thm:sn}. Similarly, one checks that Theorems \ref{thm:bvm} and Theorem \ref{thm:donskergen} can be applied under the same conditions on $\be$ and the same choices of $w_l$ as for the cut-off $L_n=L_n(\ga)$ in \eqref{eln} as above.

\subsection{Proof for other priors} \label{sec:proofotherpriors}  The proof for the other priors is essentially the same, after verifying Hellinger concentration and the change of variable conditions. The Hellinger rates for all priors with independent coefficients are verified in Section \ref{sec:hellingerindep}, and for the dependent histogram priors verification takes place in Section \ref{sec:hellingerdep}.  

The change of variable conditions, specifically \ref{c:cvarsn} for $b = \psi_{LK}$ and \ref{c:cvarb} are verified in Section \ref{sec:changevar}. For the dependent and independent log-Laplace histogram priors, the conditions are verified in Section \ref{sec:loglaplacechangeofvar}. The conditions are verified for all priors with Gaussian coefficients in \ref{sec:changevargauss} and finally for the dependent and independent Gamma histogram priors in \ref{sec:changevargamma}.

\section{Intermediate rate results\label{sec:suffrates}}

In this Section, we show how to derive posterior contraction rates for the hazard rate in terms of $\|\cdot\|_1, \|\cdot\|_2$ and $\|\cdot\|_\infty$ losses (for $\|\cdot\|_\infty$ the rate is only intermediate, i.e. non--optimal), starting from a Hellinger contraction rate $\veps_n=o(1)$  for the densities in the model (in slight abuse of notation we denote this rate by $\veps_n$, which is used in the main paper for the $L^1$--rate, but one shows below that the Hellinger rate on densities gives an $L^1$--rate on the hazard of the same order).  The derivation of the Hellinger rate is done following the general rate theory of \cite{Ghosal2000}, see Section \ref{sec:l1} below.

Recall that $h^2(P,Q)=\int (\sqrt{p}-\sqrt{q})^2 d\mu=:h^2(p,q)$ is the square--Hellinger distance between two probability distributions $P,Q$ with densities $p,q$ with respect to a common dominating measure $\mu$. For two positive integrable  functions $\la_1, \la_2$, let us denote  
\begin{equation} \label{pseudoh}
	H^2(\la_1,\la_2)= \int_0^1  ( \sqrt{\la_1} - \sqrt{\la_2} )^2.
\end{equation} 
the squared `pseudo-Hellinger' (as we do not assume $\la_1, \la_2$ to be densities) distance between $\la_1$ and $\la_2$. For a given  hazard $\la$, we write
$\bar\la$ for $\la/\int_0^1 \la$, the normalised density defined from $\la$.

\subsection{\label{sec:helltol1}$\|\cdot\|_1$--rates from Hellinger rates}
For $p_{\la}=p_{\la,g}$ the density of the observations under the survival model with hazard $\la$ and censoring density $g$, and $P_{\la}=P_{\la,g}$ the corresponding distribution.  
Suppose one knows that 
\begin{equation} \label{c:hrate}
	\Pi[ \{\la: h^2(P_{\la},P_{\la_0}) \le r_n \} \given X] = 1+\op. 
\end{equation}
We now show that this implies a contraction rate $\veps_n=Mr_n$ in $L^1$-distance for the hazard, up to a large enough multiplicative constant $M>0$.

The Hellinger distance $h^2(p_{\la},p_{\la_0})=h^2(p_{\la,g},p_{\la_0,g})$ can be written as
\begin{align} 
	h^2(p_{\la},p_{\la_0})  = &\ \underbrace{ \int_0^1 \left[\sqrt{gS}-\sqrt{gS_0}\right]^2(y) dy }_{= (i)}  
	+ \underbrace{ \int_0^1 \bar{G}(y)\left[ \sqrt{\la S} - \sqrt{\la_0 S_0} \right]^2(y) dy }_{=(ii)} \notag\\
	&\  + \underbrace{ \bar{G}(1)\left[\sqrt{S}-\sqrt{S_0}\right]^2(1) }_{=(iii)},
	\label{eq:hellingerdef}
\end{align}
with the notation $\bar{G}(y)=1 - G(y-)$ as noted below \ref{c:mod}.

\begin{lem}\label{lem:lab} Under assumptions \ref{c:mod}, there exist constants $c,c' > 0$ (depending only on the universal constants $c_i$ in  \ref{c:mod}) such that if $h^2(p_{\la},p_{\la_0})\leqa \veps_n^2=o(1)$, then 
	$c'\le \La(1) \le c$.
\end{lem}
\begin{proof}
	
	Since $\bar{G}(1)$ is by \ref{c:mod} a constant bounded away from zero, the upper bound $\bar{G}(1)\left[\sqrt{S}-\sqrt{S_0}\right]^2(1) \lesssim \veps_n^2$ from part (iii) in \eqref{eq:hellingerdef} yields the inequality $|e^{-\La(1)/2} - e^{-\La_0(1)/2}| \lesssim \veps_n$. Since it follows from \ref{c:mod} that $\La_0(1)$ is bounded away from $0$ and infinity, the previous inequality implies the same for $\La(1)$.
\end{proof}

\begin{lem}\label{lem:h_lambda_bar_bounded} 
	Under \ref{c:mod},   on a set where $h^2(p_{\la},p_{\la_0})\leqa \veps_n^2$, we have, for $H$ as in \eqref{pseudoh}, that $H^2(\la, \la_0) \lesssim \veps_n^2$.
\end{lem}
\begin{proof} 
	By Lemma \ref{lem:lab}, the quantity $S(1)$ is bounded from below, so using the  inequality $(a+b)^2 \leq 2(a^2+b^2)$, 
	\begin{align}
		H^2(\la,\la_0) & \leq \frac{1}{S(1)}\int_0^1 ( \sqrt{\la S} - \sqrt{\la_0 S} )^2 \lesssim \int_0^1 ( \sqrt{\la S} - \sqrt{\la_0 S} )^2 \notag \\
		&\le 2 \int_0^1 ( \sqrt{\la S} - \sqrt{\la_0 S_0} )^2
		+2 \int_0^1 ( \sqrt{\la_0 S_0} - \sqrt{\la_0 S} )^2. \label{eq:hbar1}
	\end{align}
	Since $G$ is a cdf, the term (ii) in \eqref{eq:hellingerdef} is bounded from below by
	\[ (ii) \ge \bar{G}(1)\int_0^1 \left[ \sqrt{\la S} - \sqrt{\la_0 S_0} \right]^2(y) dy. \]
	By  assumption \ref{c:mod} on $\bar{G}(1)$, and because $(ii) \lesssim \veps_n^2$ by assumption, we obtain that the first term of \eqref{eq:hbar1} is bounded above by a constant multiple of $\veps_n^2$. 
	By  assumption \ref{c:mod} on $g$, one gets, again referring to \eqref{eq:hellingerdef}:
	\[ (i) =  \int_0^1 \left[\sqrt{gS}-\sqrt{gS_0}\right]^2(y) dy\ge c_4  \int_0^1 \left[\sqrt{S}-\sqrt{S_0}\right]^2.\]
	
	Using that $(i)\leqa \veps_n^2$ by assumption, and noticing that the second term of \eqref{eq:hbar1} is bounded by $2\|\la_0\|_\infty \int_0^1(\sqrt{S}-\sqrt{S_0})^2$ one gets by  assumption \ref{c:mod} on $\|\la_0\|_\infty$ that the second term of \eqref{eq:hbar1} is also less than $C\veps_n^2$. 
\end{proof}

\begin{lem}\label{lem:h_lambda_bar_and_h_bar}
	Suppose \ref{c:mod} holds.
	
	{Then, on a set where $h^2(p_{\la},p_{\la_0})\leqa \veps_n^2$, we have that $H^2(\bar{\la}, \bar{\la_0}) \lesssim \veps_n^2$ and $\|\la - \la_0\|_1^2 \lesssim \veps_n^2$. }
\end{lem}
\begin{proof}
	By using the inequality $(a+b)^2 \leq 2(a^2+b^2)$,
	\begin{align*}
		H^2(\bar\la,\bar\la_0) 
		&\leq 2  \int_0^1 \left( \sqrt{\frac{\la}{\La(1) }} -  \sqrt{\frac{\la_0}{\La(1) }} \right)^2
		+  2\int_0^1 \left( \sqrt{\frac{\la_0}{\La(1) }} -  \sqrt{\frac{\la_0}{\La_0(1) }} \right)^2\\
		&= 2\La(1)^{-1} H^2(\la, \la_0) +  2\La_0(1) \left( \La(1)^{-1/2} - \La_0(1)^{-1/2} \right)^2.
	\end{align*}
	By Lemmas \ref{lem:lab} and \ref{lem:h_lambda_bar_bounded}, we have $2\La(1)^{-1} H^2(\la, \la_0)  \lesssim \veps_n^2$. Continuing with the second term of the above display, again using {Lemma \ref{lem:lab},} $(a+b)^2 \leq 2(a^2+b^2)$ and the Cauchy-Schwarz inequality,
	\begin{align*}
		\La_0(1) &\left( \La(1)^{-1/2} - \La_0(1)^{-1/2} \right)^2 = \La(1)^{-1}\left( \frac{\La_0(1)-\La(1)}{\sqrt{\La_0(1)} + \sqrt{\La(1)}} \right)^2 \\
		& \leq (\La_0(1)-\La(1))^2
		\lesssim \left(\int_0^1 (\la_0 - \la)\right)^2\\
		&
		\leq \int_0^1 (\sqrt{\la}-\sqrt{\la_0})^2 \cdot \int_0^1 (\sqrt{\la}+\sqrt{\la_0})^2
		\leq 2 [\La(1) + \La_0(1)] H^2(\la, \la_0),
	\end{align*}
	which is also bounded above by a constant multiple of $\veps_n^2$ by Lemmas \ref{lem:lab} and \ref{lem:h_lambda_bar_bounded}. { By noting $\|\la - \la_0\|_1^2 = (\int_0^1 |(\sqrt{\la} - \sqrt{\la_0})(\sqrt{\la} + \sqrt{\la_0})|)^2$ and applying Cauchy-Schwarz as above, we find that $\|\la - \la_0\|^2_1$ is bounded by a constant multiple of $\veps^2_n$ as well. }
\end{proof}

\subsection{\label{sec-strongn} Log--hazard $r$: $\|\cdot\|_2$--rate and intermediate $\|\cdot\|_\infty$--rate} $ $
We now show how to derive results for stronger losses than $\|\cdot\|_1$ for the log--hazard $r$. 
We focus on the $\|\cdot\|_2$--norm, or equivalently, as $\la_0M_0$ is bounded away from $0$ and $\infty$ by \ref{c:mod}, on the LAN $\|\cdot\|_L$--norm defined below \eqref{eq:LAN}.

\begin{lem} \label{lem:consistsn}
	Suppose $r_0=\log{\la_0}\in   \cH(\be, D)$ with $\be>1/2$, $D > 0$. Assume, for $\veps_n$ the rate in \ref{c:l1}, that the cut--off $L_n$ verifies $2^{L_n/2}\veps_n=o(1)$. 
	Then 
	\begin{align*}
		\|r-r_0\|_2^2& \leqa \|r-r_0\|_L^2 \leqa \veps_n^2,\\ 
		\|r-r_0\|_\infty & \leqa \ell_\infty(r,r_0)\leqa 2^{L_n/2}\veps_n +2^{-\be L_n}.
	\end{align*}
\end{lem}
\begin{proof} 
	To bound $\|r-r_0\|_L^2$ from above, it is enough to bound, as $M_0\le 1$, the quantity $\int_0^1 (r-r_0)^2\la_0=\int_0^1 \la_0\log^2(\la/\la_0)$,
	
	\begin{align*}
		\int_0^1 \la_0\log^2(\la/\la_0) & = \La_0(1)  \int_0^1 \bar\la_0 \log^2\left(\frac{\bar\la}{\bar\la_0}\frac{\La(1)}{\La_0(1)}\right) \\
		& \le 2\La_0(1)\int_0^1 \bar\la_0 \log^2\frac{\bar\la_0}{\bar\la}
		+ 2\La_0(1)\log^2\left(\frac{\La(1)}{\La_0(1)}\right).
	\end{align*}
	Further, using the bound on $\|\cdot\|_1$ from Lemma \ref{lem:h_lambda_bar_and_h_bar}, 
	\[ \left|\frac{\La(1)}{\La_0(1)}-1\right|
	\le (\La(1)-\La_0(1))/\La_0(1)\leqa \|\La-\La_0\|_\infty\leqa \|\la-\la_0\|_1\leqa \veps_n,\] 
	so that  expanding the logarithm in the last but one display one gets
	\[  \int_0^1 \la_0\log^2(\la/\la_0)  \leqa
	\int_0^1 \bar\la_0 \log^2\frac{\bar\la_0}{\bar\la} + \veps_n^2.
	\]
	Now, under \ref{c:mod}, for $c_1>0$ a universal constant,
	\[ \|r-r_0\|_2^2 \le c_1^{-1} \int_0^1 \la_0\log^2(\la/\la_0) \leqa
	\int_0^1 \bar\la_0 \log^2\frac{\bar\la_0}{\bar\la} + \veps_n^2.\]
	From Lemma 8 in \cite{Ghosal2007} applied with the densities $p=\bar\la_0$ and $q=\bar\la$,
	\[ \int_0^1 \bar\la_0 \log^2\frac{\bar\la_0}{\bar\la} \leqa h^2(\bar\la,\bar\la_0)
	\left(1+ \log\left\| \frac{\bar\la_0}{\bar\la}\right\|_\infty\right)^2.\]
	Proceeding as in the previous bounds,
	\[ \frac{\bar\la_0}{\bar\la} = \frac{\La(1)}{\La_0(1)}e^{r_0-r}\le (1+C\veps_n)e^{\|r-r_0\|_\infty},
	\]
	so that  
	\[\left(1+ \log\left\| \frac{\bar\la_0}{\bar\la}\right\|_\infty\right)^2  
	\leqa 1 + \|r-r_0\|_\infty^2. 
	\]
	Deduce, using the bound on $H^2(\bar{\la},\bar{\la_0})=h^2(\bar\la,\bar\la_0)$ from Lemma \ref{lem:h_lambda_bar_and_h_bar}, that
	\[ \| r - r_0 \|_2^2 \leqa  \veps_n^2(1 + \|r-r_0\|_\infty^2). \]
	
	For $M_n=2^{m_n}$ to be chosen, since $r_0\in  \cH(\be, D)$, and for any $\ga\le \be$ and $x\in[0,1]$, recalling that the end of study is at $\ta=1$,
	\begin{align*}
		|(r-&r_0)(x)|  \leqa \sum_{l\le m_n} \sum_k |r_{lk}-r_{0,lk}||\psi_{lk}(x)|+
		\sum_{l>m_n} 2^{l/2}\max_{k}(|r_{0,lk}|+|r_{lk}|)\\
		& \leqa \|r-r_0\|_2 \Big[\sum_{l\le m_n} \sum_k \psi_{lk}^2(x)\Big]^{1/2}
		+ M_n^{-\be} + \sum_{l>m_n} 2^{-l\ga}\max_{k} \left[2^{(1/2+\ga)l} |r_{lk}|\right]
		\\ 
		& \leqa \sqrt{M_n}\|r-r_0\|_2+M_n^{-\be}+ \sum_{l>m_n} 2^{-l\ga}\max_{k} \left[2^{(1/2+\ga)l} |r_{lk}|\right].
	\end{align*}
	Let us choose $M_n=2^{L_n}$ the prior cut-off.  
	Then for any prior truncated at $2^{L_n}$, the last term in the previous display is zero and one gets the bound
	\[ \|r-r_0\|_\infty 
	\le 2^{L_n/2}\|r-r_0\|_2 + 2^{-\be L_n}.
	\]  
	Putting this bound back into the former inequalities leads to
	\begin{align*}
		\| r-r_0 \|_2^2 & \leqa \veps_n^2 2^{L_n}\|r-r_0\|_2^2 + \veps_n^2.
	\end{align*}
	Since the cut--off is such that  $\veps_n^2 2^{L_n}=o(1)$ by assumption, one deduces
	\[ \|r-r_0 \|_2^2  \leqa \veps_n^2,\]
	as well as, using the previously obtained bounds, the $\|\cdot\|_\infty$--consistency
	\[ \|r-r_0\|_\infty  \leqa 2^{L_n/2}\veps_n + 2^{-\be L_n}= o(1),\]
	under the previous assumptions.  For the $\ell_\infty$--statement, one notes that
	\begin{align*}
		\ell_\infty(r,r_0) & = \sum_{l\le L_n} 2^{l/2} \max_k |r_{lk}-r_{0,lk}| + 
		\sum_{l> L_n} 2^{l/2} \max_k |r_{0,lk}|\\
		& \le \Big[\sum_{l \le L_n} 2^l\Big]^{1/2} \Big[\sum_{l \le L_n} (r_{lk}-r_{0,lk})^2\Big]^{1/2} + 2^{-\be L_n} \\
		& \leqa 2^{L_n/2}\|r-r_0\|_2 + 2^{-\be L_n} \leqa 2^{L_n/2}\veps_n + 2^{-\be L_n},
	\end{align*}
	using the previously obtained $L^2$--rate on $r$. 
\end{proof}

\section{\label{sec:l1}Hellinger rates}

We obtain Hellinger-rates for all considered priors.
We recall that by $h^2(p_\la, p_{\la_0})$ we mean $h^2(p_{\la, g}, p_{\la_0, g})$, that is, we vary the hazard while we keep the censoring distribution the same.  
Our goal is to obtain a Hellinger rate $\veps_n\leqa \veps_{n}^{\be,L_n}$. In this Section we use $\veps_{n,\be}$ as a shorthand for  $\veps_{n}^{\be,L_n}$. That is,
\[ \veps_{n,\be} = \sqrt{L_n2^{L_n}/n}+2^{-\be L_n}.\] 
Following \cite{Ghosal2000}, suitable tests exist for the Hellinger distance and it suffices to verify, as $n \to \infty$, the conditions of Theorem 1 in  \cite{Ghosal2000},
\begin{align}
	\log N(C_1\varepsilon_n, \mathcal{A}_n, h) & \le C_2n\varepsilon_n^2, \label{entropy}\\
	\Pi(L^2[0, 1]\backslash\mathcal{A}_n) & 
	\le e^{-n\varepsilon_n^2(C_3 + 4C_4^2)},
	\label{setan}\\
	\Pi(B_{KL}(r_0, C_4\varepsilon_n)) & \ge e^{-n\varepsilon_n^2C_3},
	\label{priormass}
\end{align}
for $\mathcal{A}_n$ a sequence of measurable sets, where $K(p,p_0)=-P_0\log(p/p_0)$ and $V(p,p_0)=P_0\log^2(p/p_0)$, and one denotes 
\[B_{KL}(p_0, \varepsilon)=\{r:\ K(p_\la,p_{\la_0})\le \veps^2, V(p_\la,p_{\la_0})\le \veps^2\}.\] 
Let us set, with $r_{lk}=\psg r, \psi_{lk}\psd$ the wavelet coefficients of $r$,
\[ \cA_n=\{r:\, r_{lk}=0 \ (\forall\, l>L_n, k),\quad |r_{lk}|\le n \ (\forall\, l\le L_n, k)\}. \]
By the same arguments as used in the Cox model in \cite{Castillo2012}, one can relate the Hellinger distance between the data distributions and the supremum norm distance between log--hazards, as well as  the set $B_{KL}$ to a supremum norm ball: by Lemmas 7 and 8 of \cite{Castillo2012} (setting the parameter of the Cox model to zero), we have 
\begin{align}
	h^2(p_{\la_1},p_{\la_2}) & \leqa \|r_1-r_2\|_\infty^2 e^{\|r_1-r_2\|_\infty}
	\quad (\|r_1-r_2\|<1/4) \label{boundh2} \\
	\{r : \|r - r_0\|_\infty \leq \veps\} & \subset \{r : K(p_{\la_0}, p_\la) \lesssim \veps^2, V(p_{\la_0}, p_\la) \lesssim \veps^2\}, \label{inclusion}
\end{align}
for any $\veps>0$ small enough.

\subsection{Independent priors on log--hazard Haar coefficients\label{sec:hellingerindep}} We consider  either independent Gaussian $r_{lk}\sim \cN(0,\sigma_l^2)$ or Laplace $\cL(0,\sigma_l)$ coefficients. We denote $\cL(0,\sigma_l)$ the Laplace distribution with scale $\sigma_l$, that is the distribution of $\sigma_l Y$ where $Y\sim\cL(0,1)$.  

Condition \eqref{setan} follows from a union bound by noticing that the marginal probabilities $\Pi[|r_{lk}|> n]\leqa e^{-n}$ by using $\sigma_l\le 1$ and the Gaussian or Laplace tail bound, so that the prior probability in \eqref{setan} is bounded from above by $C2^{L_n}e^{-n}$ which goes much faster to $0$ than $e^{-Cn\veps_{n,\be}^2}$.

To verify the entropy condition \eqref{entropy}, one notes that thanks to the bound \eqref{boundh2}, one can replace the Hellinger metric by the supremum norm on $r$'s. By identifying $r\in\cA_n$ with the vector of size $\sum_{l\le L_n} 2^l$ of its wavelet coefficients, we have, for any  $r_1, r_2$ in $\cA_n$,
\begin{align*}
	\|r_1-r_2\|_\infty & \leqa \ell_\infty(r_1,r_2)=\sum_{l\le L_n}2^{l/2}\max_{k}|r_{1,lk}-r_{2,lk}|\\
	& \leqa 2^{L_n/2} \max_{l\le L_n, k} |r_{1,lk}-r_{2,lk}|
	\leqa 2^{L_n/2} \|r_1-r_2\|_{\ell^2},
\end{align*}
where $\|\cdot\|_{\ell^2}$ is the Euclidean norm of a vector. From this deduce that
\begin{align*}
	\log N(\veps,\cA_n,h) & \leqa \log N(c_1\veps 2^{-L_n/2},B_{\ell^2}(0,c_2 2^{L_n/2}n),\|\cdot\|_{\ell^2}) \\
	&\leqa C2^{L_n}\log\left( \frac{c_3 2^{L_n/2}n}{\veps 2^{-L_n/2}}\right) \leqa 
	2^{L_n}\log(n/\veps),
\end{align*}
using that the $\veps$--entropy of a Euclidean ball of radius $R$ in dimension $C2^{L_n}$ is bounded from above by $C2^{L_n}\log(C'R/\veps)$. The last display is bounded from above by a constant times $L_n2^{L_n}\leqa n\veps_{n,\be}^2$ for $\veps\asymp\veps_{n,\be}$.

Finally,  the prior mass condition is obtained by first using the inclusion \eqref{inclusion} and then noting that, since $\|r_0-r_{0,L_n}\|_\infty\leqa 2^{-L_n \be}\leqa \veps_{n,\be}$ (recall $r_{0,L_n}=P_{L_n}r_0$) using that $r_0$ is $\be$--H\"older, if one chooses $\veps_n\asymp \veps_{n,\be}$,
\begin{align*}
	\Pi[B_{KL}(r_0,C\veps_n)] & \geqa \Pi[\|r-r_0\|_\infty \leqa \veps_n]\geqa 
	\Pi[\ell_\infty(r,r_{0,L_n}) \leqa \veps_n]\\
	& \geqa \Pi[|r_{lk}-r_{0,lk}|\leqa 2^{-l/2}\veps_n/L_n,\ \ \forall\, l\le L_n, k]\\
	& \geqa  \prod_{l\le L_n, k} \Pi[|r_{lk}/\sigma_l-r_{0,lk}/\sigma_l|\leqa 2^{-l/2}\sigma_l^{-1}\veps_n/L_n]
\end{align*}
by using independence of the coordinates $r_{lk}$ in this case. Further, each marginal prior probability in the last display can be bounded from below using the minimum of the density on the corresponding interval. As $\sigma_l\geq 2^{-l/2}$ by \eqref{c:sig}, the term $r_{0,lk}/\sigma_l$ is bounded, so the density at stake is bounded from below by a constant $c_0>0$. One deduces  that
\begin{align*}
	\Pi[B_{KL}(r_0,C\veps_n)] 
	&\geqa \prod_{l\le L_n,k}  c_0 \left(2 2^{-L_n/2} \frac{\veps_n}{L_n} \right)
	\geqa e^{-C' 2^{L_n}\log\left(2 2^{L_n/2} \frac{L_n}{\veps_n} \right)},
\end{align*}
which is bounded from below by $e^{-C2^{L_n}\log(L_n)}\geq
e^{-C'n\veps_{n,\be}^2}$ as requested.

\subsection{Priors defined on histogram heights\label{sec:hellingerdep}}

The priors on histogram heights from Section \ref{sec:appli} either draw heights independently, or in a dependent way specified in \eqref{eq:ARstructure}. Below we give the argument for dependent heights in detail, the case of independent heights being similar and easier (independence can be used to bound the prior mass directly, as for wavelet priors above).

In the next lines and in the sequel we freely use the following bounds, that follow from the definition of conditional probabilities: for measurable $A,B$,
\begin{align}
	\inf_{x\in A}\, & \Pi[r_k\in B  \given r_{k-1}=x] \notag \\
	& \le \Pi[r_k\in B\given r_{k-1}\in A]
	\le \sup_{x\in A}\, \Pi[r_k\in B\given r_{k-1}=x]. \label{cpr}
\end{align}
We begin by verifying \eqref{priormass}, for which we use the inclusion \eqref{inclusion}.  We start with a computation valid for all three dependent priors. Let $r_{0,k}  = |I^{L_n+1}_k|^{-1}\int_{I^{L_n+1}_k} r_0$ and $r_{0,L_n} = \sum_{k=0}^{2^{L_n+1}-1} r_{0,k} \1_{I^{L_n+1}_k}$,   the $L_2$-projection of $r_0$ onto the space of histograms with $2^{L_n+1}$ equispaced bins. For $n$ large enough, we have $\|r_0 - r_{0,{L_n}}\|_\infty \leq \varepsilon_n/2$ (as $\|r_0 - r_{0,{L_n}}\|_\infty \leqa 2^{-L_n\be}$ for $\be\le 1$ if $r_0\in\cH(\be,D)$ and taking $\veps_n\asymp \veps_{n,\be}$). First using the triangle inequality,
\begin{align*}
	\Pi&[ \|r - r_0\|_\infty < \varepsilon_n] 
	\geq \Pi [ \|r - r_{0,{L_n}}\|_\infty < \varepsilon_n/2]\\
	&= \Pi[ \forall k \in \{0, \ldots, 2^{L_n+1}-1\} :  |r_k - r_{0,k}| < \varepsilon_n/2]\notag\\
	&=  \Pi[|r_0 - r_{0,0}| < \varepsilon_n/2] \prod_{k=0}^{2^{L_n + 1}} \Pi[ |r_{k+1} - r_{0,k+1}| < \varepsilon_n/2 \mid  |r_{k} - r_{0,k}| < \varepsilon_n/2],\notag
\end{align*}
using the Markov structure of the prior. Since $r_0\in \cH(\be,D)$, we further obtain $ |r_{0,k} - r_{0,k+1}| \leq D|I^{L_n+1}_1|^\beta$ for all $k$. Using the so obtained inequalities, we verify \eqref{priormass} for each of the three dependent histogram priors.

\emph{Prior mass: dependent log-normal prior}\\
We study $\Pi[ |r_{k+1} - r_{0,k+1}| < \varepsilon_n/2 \mid  |r_{k} - r_{0,k}| < \varepsilon_n/2]$. We recall the prior structure $r_{k+1} \mid r_k \sim \mathcal{N}(r_k - s^2/2, s^2)$, where for clarity of exposition we abbreviate $s^2 = \log(1+\sigma^2)$ and we use the first inequality in \eqref{cpr}.

For a normally distributed random variable, the probability of being realized in an interval of size $\varepsilon_n/2$  is smallest farthest out in the tails. The most extreme interval is achieved if $|r_k - r_{0,k}| = \varepsilon_n/2$ and $|r_{0,k} - r_{0,k+1}| = D2^{-(L_n+1)\beta}$. Thus the probability of interest is at least the probability that $Z$ is in $[s^2/2 + D2^{-(L_n+1)\beta}, s^2/2 + \varepsilon _n + D2^{-(L_n+1)\beta}]$, where $Z \sim \mathcal{N}(0, s^2)$.

We use: $\int_a^b \frac{1}{\sqrt{2\pi}} e^{-x^2/2} dx \geq  (b-a)/\sqrt{2\pi} e^{-\max\{|a|, |b|\}^2/2}$, and find:
\begin{align*}
	\Pi &[ |r_{k+1} - r_{0,k+1}| < \varepsilon_n/2 \mid  |r_{k} - r_{0,k}| < \varepsilon_n/2]\\
	&\geq   \varepsilon_n\frac{1}{\sqrt{2\pi s^2}}e^{-\frac{1}{2 s^2} (\max\{ |s^2/2  + D2^{-(L_n+1)\beta}|, s^2/2 + \varepsilon_n + D2^{-(L_n+1)\beta}\})^2}
\end{align*}
As $s^2$ is fixed while $L_n \to \infty, \varepsilon_n \to 0$, we find that the bound is of the order $ \varepsilon_n e^{-1}/\sqrt{2\pi s^2}$. Therefore, $\Pi[\|r-r_0\|_\infty < \varepsilon_n]$ is, up to constants, bounded below by $e^{2^{L_n+1}\log{\varepsilon_n}} \geq e^{-C'n\veps^2_{n,\beta}}$ as required.

\emph{Prior mass: dependent log-Laplace prior}\\
Using the same argument as for the log-normal prior, we find we can lower bound the probability of interest by the probability of a $\cL(0, \theta)$ random variable being in either $[-\log(1-\theta^{-2}) - \varepsilon - D2^{-(L_n+1)\beta}, -\log(1-\theta^{-2}) - D2^{-(L_n+1)\beta}]$ or in $[-\log(1-\theta^{-2}) + D2^{-(L_n+1)\beta}, -\log(1-\theta^{-2}) + \varepsilon_n + D2^{-(L_n+1)\beta}]$, whichever leads to the lowest probability. This leads to the lower bound: 
\begin{align*}
	\pi &[ |r_{k+1} - r_{0,k+1}| < \varepsilon_n/2 \mid  |r_{k} - r_{0,k}| < \varepsilon_n/2]\\
	&\geq \min\left\{
	\varepsilon_n \frac{\theta}{2}e^{-\theta \max\{|-\log(1-\theta^{-2}) - \varepsilon_n - D2^{-(L_n+1)\beta}|, |-\log(1-\theta^{-2}) -D2^{-(L_n+1)\beta}| \} },\right.\\
	& \left.   \varepsilon_n\frac{\theta}{2}e^{-\theta (\max\{ |-\log(1-\theta^{-2})  +D2^{-(L_n+1)\beta}|, -\log(1-\theta^{-2}) + \varepsilon_n + D2^{-(L_n+1)\beta}\})^2}
	\right\}.
\end{align*}
Assuming $\theta$ is fixed while $L_n \to \infty, \varepsilon_n \to 0$, we find that the bound is of order $\varepsilon_n \theta e^{\theta \log(1-\theta^{-2})}/2$, after which the argument is finished as for the dependent log-normal prior.

\emph{Prior mass: dependent Gamma prior}\\
The pdf on the $r$-scale is given by $f_r(r | r_{k}) = \frac{\alpha^\alpha e^{-\alpha r_k}}{\Gamma(\alpha)} e^{\alpha[r - e^{r-r_k}]}.$ This density is not symmetric around $r_k$ but does attain its maximum at $r_k$. The tail is most heavy on the side where $r > r_k$. Thus, arguing as for the dependent log-normal prior, the probability of interest is lower bounded by the probability that $r_{k + 1}$ is in $[r_k + D2^{-(L_n+1)\beta}, r_k + D2^{-(L_n+1)\beta} + \varepsilon_n]$, which is bounded from below by
$\varepsilon_n \frac{\alpha^\alpha}{\Gamma(\alpha)} e^{\alpha[D2^{-(L_n+1)\beta} + \varepsilon_n - e^{D2^{-(L_n+1)\beta} + \varepsilon_n}]}.$ As $L_n \to \infty, \varepsilon_n \to 0$ as $\alpha$ remains fixed, this bound is approximately $\varepsilon_n \alpha^\alpha/\Gamma(\alpha) e^{-\alpha}$, and again we finish the argument as for the dependent log-normal prior.

\emph{Entropy and prior mass on sieve}\\
Conditions \eqref{entropy} and   \eqref{setan} are similarly checked as for the independent wavelet priors in Section \ref{sec:hellingerindep}, albeit with a different sieve. Denoting the histogram heights of $r$ by $r_k=\log{\la_k}$, we set:
$$\mathcal{A}_n = \{r : |r_k| \leq n^2, \ \forall k\}, $$
whose entropy is bounded in a similar way as in Section \ref{sec:hellingerindep}. 
The verification of \eqref{setan} requires an extra argument, because of the Markov structure of the priors. A union bound gives
\begin{align*}
	\Pi&[ \exists k \in \{0, 1, 2, \ldots, 2^{L_n+1} - 1\}: |r_k| > n^2 ] 
	\leq \sum_{k=0}^{2^{L_n+1}-1} \Pi[|r_k| > n^2],
\end{align*}
and we have
\begin{align*}
	\Pi[|r_k| > n^2] &= \Pi[ |r_k| > n^2, |r_{k-1}| > n^2 -2^{L_n}] + \Pi[ |r_k| > n^2, |r_{k-1}| \leq n^2 -2^{L_n}]\\
	&\leq  \Pi[ |r_{k-1}| > n^2 -2^{L_n}] + \Pi[ |r_k| > n^2 \mid |r_{k-1}| \leq n^2 -2^{L_n}].
\end{align*} 
Denoting $\delta_{x, y} = \Pi[ |r_x| > n^2 - y2^{L_n}]$ and $\psi^{v,w}_{x,y} = \Pi[ |r_x| > n^2 - y2^{L_n} \mid |r_v| \leq n^2 - w2^{L_n}]$, we just derived the following relationship.
$$\delta_{k,0} \leq \delta_{k-1, 1} + \psi_{k, 0}^{k-1, 1}, $$ 
By induction, one deduces that
\begin{equation*}
	\delta_{k,0} \leq \delta_{0, k} + \psi_{0,k}^{1, k-1} + \ldots + \psi_{k, 0}^{k-1, 1},
\end{equation*}
and we can further bound $\delta_{0,k} \leq \delta_{0, 2^{L_n+1}} =  \Pi[|r_0| > n^2 - 2^{L_n+1}2^{L_n}].$ Since the priors considered here have exponential tails on the $r$-scale, this latter term is exponentially small as long as $2^{2L_n} = o(n^2)$, which is the case for any $\beta > 0$. Similarly, the $\psi$ terms are all of order at most $\exp(-2^{L_n})$, using the upper-bound in \eqref{cpr} and a bound similar to the one for prior masses above. 

Putting it all together, since the summation over $k$ induces a multiplicative factor of at most $2^{L_n+1}$ (so that the tail probabilities just considered still dominate) we again find that  \eqref{setan} is met.

\section{\label{sec:changevar} Change of variables condition}

In this Section we proceed to verifying conditions \ref{c:cvarsn} and \ref{c:cvarb} for the priors we consider. To do so, we need some preliminary lemmas on decrease of wavelet coefficients of $\psi_{LK}/M_0$ (respectively, histogram heights).

\subsection{Approximation lemmata for wavelets and histograms}

\begin{lem}[Wavelet approximation $\psi_n$]  \label{lempsin}
	With  $(\psi_{lk})$ the considered wavelet basis (Haar or CDV), let $b=\psi_{LK}$ and set  $\psin=P_{L_n}(b/M_0)$ and $\psi_{n,lk}=\psg b/M_0,\psi_{lk} \psd$. Denoting by $S_{lk}$ the support of $\psi_{lk}$, and supposing $L\le L_n$,
	\begin{align*}
		|\psi_{n,lk}| 
		& \leqa 2^{-(L-l)/2}\qquad \text{if }\ l\le L\\
		|\psi_{n,lk}|& \leqa 2^{L/2-3l/2} \qquad \text{if }\ S_{lk}\cap S_{LK}\neq\emptyset,\, (l,k)\neq (L,K),\\
		|\psi_{n,lk}|& = 0 \qquad\qquad \quad \text{if }\ S_{lk}\cap S_{LK}=\emptyset.
	\end{align*}
\end{lem}
\begin{proof}
	Below we use properties (W1)-(W2)-(W3) of the wavelet basis, see section \ref{sec:wave}. We denote by $S_{lk}$ the support of $\psi_{lk}$. By taking absolute values and using $\|\psi_{lk}\|_\infty\leqa 2^{l/2}$ uniformly in $l,k$, 
	\[ | \psg \psi_{LK}/M_0,\psi_{lk} \psd|\leqa 2^{l/2} \int_{S_{LK}} \frac{2^{L/2}}{M_0} \leqa 2^{(l-L)/2},\]
	which gives the first bound. For the second bound, one notes that if $(l,k)\neq (L,K)$, then $\psg \psi_{lk},\psi_{LK}\psd=0$ by orthogonality of the basis, and writing $\bar{g}$ for the average of a function $g$ on $S_{lk}$,
	\begin{align*}
		| \psg \psi_{LK}M_0^{-1},\psi_{lk} \psd| & = 
		| \psg \psi_{LK}(M_0^{-1}-\overline{M_0^{-1}}),\psi_{lk} \psd| \\
		& \leqa \| \psi_{LK} \|_\infty \| \psi_{lk} \|_1 2^{-l} \leqa 2^{L/2-3l/2},
	\end{align*}
	where we use that $M_0^{-1}$ is   Lipschitz by Lemma \ref{lem:lip}. 
	The third bound is immediate as supports are assumed to be disjoint.
\end{proof}

\begin{lem} \label{lem:cvarlap}
	For $b\in L^\infty[0,1]$, let us recall the notation $\psin=P_{L_n}(b/M_0)$, and let us denote $\psi_{n,lk} = \langle b/M_0, \psi_{lk}\rangle$. For $l\le L_n$ an integer, we have
	\begin{enumerate}
		\item if $b=\psi_{LK}$, for any $L\le L_n$ and $K$, 
		\[ \sum_{0\le k<2^l}  |\psi_{n,lk}| \leqa 2^{-|l-L|/2};
		\]
		\item if $b  \in \cH(\mu, D)$ for some $\mu,D>0$, setting $\mu'=1\wedge \mu$,
		\[ \sum_{0\le k<2^l}  |\psi_{n,lk}| \leqa 2^{-(1/2-\mu')l}.\]
	\end{enumerate}
\end{lem}
\begin{proof}
	For the first inequality, one considers the two cases $l\le L$ and $l>L$, and use the properties (W3) of the wavelet basis. 
	In case $l\le L$, one notes that the support $S_{LK}$ of the wavelet $\psi_{LK}$ only intersects at most a constant number (independent of $l$) of the supports $S_{lk}$ for $1\le k\le l$. For any $l$ such that the support intersect, one uses the first bound of  Lemma \ref{lempsin}, which gives the desired bound. In the case $l>L$, the supports $S_{lk}$ for $0\le k<2^l$ intersect $S_{LK}$ at most a number $C2^{l-L}$ of times. By using the second bound of  Lemma \ref{lempsin}, one obtains
	\[ \sum_{0\le k<2^l}  |\psi_{n,lk}| \leqa \sum_{0\le k<2^l} 
	2^{L/2-3l/2} 2^{l-L} \leqa 2^{-(l-L)/2}\]
	which gives the first inequality of the lemma. For the second inequality, one notes that $b/M_0$ is a product of a H\"older-$\mu$ with a Lipschitz function (as $M_0^{-1}$ is Lipschitz by Lemma \ref{lem:lip}), so is a H\"older-$\mu'$ function with $\mu'=1\wedge \mu$.
\end{proof}

\begin{lem} \label{lempsinh}
	In this Lemma, $(\psi_{lk})$ is the Haar basis. Let $\psi_n=\psin =  P_{L_n}(b/M_0)$ for $b = \psi_{LK}$ for some $L, K$, and set $H = W^{-1}\psi_n$, with $W$ the matrix described in Section \ref{sec:relatinghistogramsandhaar}. Then for $L\le L_n$ and $1\le j <  2^{L_n+1}$,
	\begin{align*}
		|H_j| 
		& \le C 2^{L/2} \qquad \text{if }\  I^{L_n+1}_j \cap S_{LK} \neq \emptyset,\\
		H_j & = 0 \qquad \qquad\text{if }\  I^{L_n+1}_j \cap S_{LK} =\emptyset.
	\end{align*}
	
\end{lem}
\begin{proof}
	The quantity $H_j$ is the value of the function $\psin$ on interval $I^{L_n+1}_j$, that is, $\psin = \sum_{j=1}^{2^{L_n+1}-1} H_j \1_{I^{L_n+1}_j}$. By definition, $\psin$ equals the wavelet expansion of $\psi_{LK}/M_0$ up to level $l=L_n$. For any $j$ as above, denoting by $k_j:=\lfloor \frac{j-1}{2}\rfloor$ (so that $I^{L_n+1}_j \cap S_{L_nk_j} \neq \emptyset$) and $\psi_{n,lk} = \langle b/M_0, \psi_{lk}\rangle$,
	\[ H_j = \sum_{l\le L_n} 2^{l/2} \psi_{n,l k_j}. \]
	It now follows from Lemma \ref{lempsin} that $\psi_{n,l k_j}=0$ if $S_{lk_j}\cap 
	S_{LK}=\emptyset$ while, if $S_{lk_j}\cap S_{LK}\neq \emptyset$,
	\[   \psi_{n,lk_j} \le  
	\begin{cases}
		C2^{-(L-l)}
		& \text{if } l\le L,\\
		2^{L/2}2^{-3l/2} & \text{if }l > L.
	\end{cases}
	\]
	From this one deduces that if $\cI_j\cap S_{LK}=\emptyset$, then $H_j=0$, while if         $I^{L_n+1}_j\cap S_{LK}\neq \emptyset$, one gets
	\[ |H_j|\le \sum_{l\le L} 
	2^{l/2}C2^{-(L-l)} + \sum_{l=L+1}^{L_n} 2^{l/2}2^{L/2-3l/2}\leqa 2^{L/2}. \qedhere
	\]
\end{proof}

\begin{lem} \label{lemdeltah}
	Under the same notation as in Lemma \ref{lempsinh}, we have
	\[ \sum_{j=1}^{2^{L_n+1}} |H_{j}-H_{j-1}| \le C 2^{L/2}.\]
\end{lem}
\begin{proof}
	By definition the $H_i$'s are the heights of the histogram $\psi_n= P_{L_n}(\psi_{LK}/M_0)$. Let us denote the two halves of the support $S_{LK}$ of $\psi_{LK}$ by $S_{LK}^{-}$ and $S_{LK}^+$ respectively. By definition of $\psi_{LK}$ and linearity,
	\[ \psi_n = 2^{L/2}P_{L_n}(\1_{S_{LK}^{-}}M_0^{-1}) -  
	2^{L/2}P_{L_n}(\1_{S_{LK}^{+}}M_0^{-1}).
	\]
	Note that $P_{L_n}(\1_{S_{LK}^{-}}M_0^{-1})=\1_{S_{LK}^{-}}P_{L_n}(M_0^{-1})$ and similarly for $S_{LK}^+$ (indeed, $M_0^{-1}= \1_{S_{LK}^{-}}M_0^{-1} + \1_{(S_{LK}^{-})^c}M_0^{-1}$ and the $P_{L_n}$--projection  of the second part has support outside $S_{LK}^{-}$). That is
	\[ \psi_n = 2^{L/2}\1_{S_{LK}^{-}}(P_{L_n}M_0^{-1}) -  
	2^{L/2}\1_{S_{LK}^{+}}(P_{L_n}M_0^{-1}) \]
	(equivalently we have just shown that $\psi_n=\psi_{LK}P_{L_n}(M_0^{-1})$ holds for the Haar basis). Denoting by $\eta_j$ the heights of the histogram $P_{L_n}(M_0^{-1})$, one obtains the following sharper version of the result of Lemma \ref{lempsinh}
	\begin{align*}
		H_j = 
		\begin{cases}
			& 2^{L/2} \eta_j \qquad \text{if }\  I^{L_n+1}_j \subset S_{LK}^{-},\\
			&  -2^{L/2} \eta_j \quad \ \text{if }\  I^{L_n+1}_j \subset S_{LK}^{+},\\
			&  0 \qquad \qquad\ \text{if }\  I^{L_n+1}_j \cap S_{LK} =\emptyset.
		\end{cases}
	\end{align*}
	
	First, $\eta_j$ are bounded by a constant $D>0$ (as $M_0^{-1}$ is bounded); second, as $M_0^{-1}$ is Lipschitz by Lemma \ref{lem:lip}, we have $|\eta_j - \eta_{j-1}|\leqa 2^{-L_n}$. Deduce
	\begin{align*}  
		\sum_{j=1}^{2^{L_n+1}} |H_{j}-H_{j-1}| & \le  D 2^{L/2} + 2^{L/2}\sum_{j:\ I^{L_n+1}_j \subset S_{LK}^{-}} |\eta_j-\eta_{j-1}| \\
		& + 2^{L/2}(2D) +
		2^{L/2}\sum_{j:\ I^{L_n+1}_j \subset S_{LK}^{+}} |\eta_j-\eta_{j-1}| +  D 2^{L/2},
	\end{align*}
	where the far left and right terms on the right hand side of the last display account for the first and last jumps of the histogram, and the middle term $2^{L/2}(2D)$ for the jump at the passing from $S_{LK}^-$ to $S_{LK}^+$. Finally, the sums involving $|\eta_j-\eta_{j-1}|$ in the last display are bounded by a constant times $2^{L_n-L}2^{-L_n}\leqa 2^{-L}$. Deduce that the last display is bounded by a constant times $2^{L/2}$, which concludes the proof.
\end{proof}

\subsection{Laplace wavelet prior\label{sec:cvarlap}}
Here we consider the case of a wavelet prior with independent Laplace coefficients with rescaling $\sigma_l$. The argument below in fact  only assumes a $\log$-Lipschitz prior density with rescaling $\sigma_l$ (of which  
the Laplace prior just mentioned is a particular case). 
Denoting $\psi_n=\psi_{b,L_n}$ as a shorthand,
\begin{equation}\label{eq:ratiochangevarstart}
	\frac{N_n}{\Delta_n} := \frac{\int_{A_n} \exp( \ell_n(r_t^n) - \ell_n(r_0)) d\Pi(r)}{\int \exp( \ell_n(r) - \ell_n(r_0)) d\Pi(r)}.
\end{equation}
Plugging in $r_t^n = r - t\psi_n/\sqrt{n}$ and writing $d\Pi(r) = \prod_{l \leq L_n; k} \pi_{lk}(r_{lk})dr_{lk}$ and subsequently taking a prior of the form $\pi_{lk}(\cdot) = \sigma_l^{-1}\phi(r_{lk}\sigma_l^{-1})$, we obtain
$$N_n = \int_{A_n} \exp( \ell_n(r_t^n) - \ell_n(r_0)) \prod_{l \leq L_n; k} \frac{1}{\sigma_l} \phi\left( \frac{r_{lk}}{\sigma_l}\right) dr_{lk}, $$
which we rewrite as
$$N_n = \int_{A_n} \exp( \ell_n(r_t^n) - \ell_n(r_0)) \prod_{l \leq L_n; k} \frac{1}{\sigma_l} \phi\left( \frac{r_{lk} - \frac{t\psi_{n,lk}}{\sqrt{n}} + \frac{t\psi_{n,lk}}{\sqrt{n}} }{\sigma_l}\right) dr_{lk}.$$
We substitute $\rho_{lk} = r_{lk} - \frac{t\psi_{n,lk}}{\sqrt{n}}$. By invariance of the Lebesgue measure, we have $dr_{lk} = d\rho_{lk}$ and thus we arrive at:
$$N_n = \int_{A_n - \frac{t\psi_n}{\sqrt{n}}} \exp( \ell_n(\rho) - \ell_n(r_0)) \prod_{l \leq L_n; k} \frac{1}{\sigma_l} \phi\left( \frac{\rho_{lk} + \frac{t\psi_{n,lk}}{\sqrt{n}} }{\sigma_l}\right) d\rho_{lk}.$$
In comparison, we can write
$$\Delta_n = \int \exp( \ell_n(\rho) - \ell_n(r_0))   \prod_{l \leq L_n; k} \frac{1}{\sigma_l} \phi\left( \frac{\rho_{lk} }{\sigma_l}\right) d\rho_{lk}.$$
If in $N_n$ we didn't have the extra term $ \frac{t\psi_{n,lk}}{\sqrt{n}}$, the fraction $N_n/\Delta_n$ would be equal to the posterior of $A_n -  \frac{t\psi_{n,lk}}{\sqrt{n}}$ and this is what we will compare our current expression to.  

By the Lipschitz-property of $\phi$, we can bound from above (and similarly from below, if we want to):
\begin{align*}
	\phi\left( \frac{\rho_{lk} + \frac{t\psi_{n,lk}}{\sqrt{n}} }{\sigma_l}\right)
	&= \phi\left( \frac{\rho_{lk} }{\sigma_l}\right) 
	\cdot \exp\left( \log \phi\left( \frac{\rho_{lk} + \frac{t\psi_{n,lk}}{\sqrt{n}} }{\sigma_l}\right) - \log \phi\left( \frac{\rho_{lk}  }{\sigma_l}\right) \right)\\
	&\leq \phi\left( \frac{\rho_{lk} }{\sigma_l}\right) 
	\cdot \exp\left( D \frac{t|\psi_{n,lk}|}{\sigma_l \sqrt{n}} \right).
\end{align*}
Proceeding similarly for the lower bound, one obtains 
\begin{equation*}
	\exp\left( -\sum_{l \leq L_n; k} D \frac{t|\psi_{n,lk}|}{\sigma_l \sqrt{n}} \right) \le \frac{N_n}{\Delta_n}\Pi\left(A_n - \frac{t\psi_n}{\sqrt{n}} \mid X\right)^{-1} \leq \exp\left( \sum_{l \leq L_n; k} D \frac{t|\psi_{n,lk}|}{\sigma_l \sqrt{n}} \right).
\end{equation*}
From this one sees that to verify \ref{c:cvarsn}, it is enough to consider the upper--bound: one first bounds the posterior probability in the last display from above by $1$, and next uses Lemma \ref{lem:cvarlap} to bound the exponential term from above. For $\sigma_l$ bounded from below or decreasing to $0$ not too fast with $l$, the quantity in the exponential is bounded (in fact goes to $0$ for fixed or not too large $t$'s) by $C|t|\le C(1+t^2)$ as requested.  

To verify \ref{c:cvarb}, let us first note that for that statement the real $t$ is fixed and the functional $b$ is a fixed uniformly bounded function. Then, it is enough to show, first, that the sum in the exponential above goes to $0$, which for fixed $t$ follows for any sequence  $(\sigma_l)$ that does not go too fast to $0$. Second, one wishes to show that 
\begin{equation}\label{picvar}
	\Pi\left(A_n - \frac{t\psi_n}{\sqrt{n}} \mid X\right) =1+\op. 
\end{equation} 
As the posterior probability is at most $1$, it is enough to bound the last display from below by $1+\op$. To do so, one notes that $\|\psi_n\|_\infty \leqa L_n\|b\|_\infty\leqa L_n$ by Lemma \ref{lem:psi}, so that $t\|\psi_n\|_\infty/\rn \leqa L_n/\rn$. This is a $o(\veps_n)$, as $\veps_n$ is a nonparametric rate (in particular, for all priors and under the H\"older conditions we consider, it is always strictly slower than $\log{n}/\rn$). This means that $A_n'\subset A_n - \frac{t\psi_n}{\sqrt{n}}$, where $A_n'$ is as $A_n$ but with $\veps_n$ replaced by $\veps_n/2$. Conclude that up to choosing $\veps_n$ equal to twice its original value, we have $\Pi(A_n'\given X)=1+\op$.

\subsection{Log--Laplace dependent and independent histogram priors\label{sec:loglaplacechangeofvar}}

We again study the ratio \eqref{eq:ratiochangevarstart}, first for the dependent log-Laplace histogram prior. We note that the prior on the $\{\lambda_k\}$ implies a prior on the Haar wavelet coefficients.  We can relate the vector of Haar wavelet coefficients $r$ to the vector of histogram heights on the log-scale $h$ (i.e. $h_k = \log \la_k$) through the relationship $r = Wh$, with $W$ the matrix described in Section \ref{sec:relatinghistogramsandhaar}. Carrying out this transformation, then perturbing again by an additive factor $-\frac{t}{\sqrt{n}}\psi_n$, and finally substituting $\rho = r - \frac{t}{\sqrt{n}}\psi_n$ in the numerator, we arrive at the ratio
\begin{equation}\label{eq:ratiochangevardep}
	\frac{N'_n}{\Delta'_n} := \frac{\int_{A_n-\frac{t}{\sqrt{n}}\psi_n} \exp( \ell_n(\rho) - \ell_n(r_0)) |\det W|^{-1} \frac{F(W^{-1}\rho + z)}{F(W^{-1}\rho)} F(W^{-1}\rho)d\rho }
	{\int \exp( \ell_n(\rho) - \ell_n(r_0)) |\det W|^{-1} F(W^{-1}\rho) d\rho},
\end{equation}
where $z = \frac{t}{\sqrt{n}}W^{-1}\psi_n$ and $F(h) = f_{\theta_0}(h_0 \mid \mu_0)\prod_{j=1}^{2^{L_n+1}-1} f_\theta(h_j \mid \mu(h_{j-1}))$,  with $f_\theta( u \mid \mu) = \tfrac{1}{2}\theta e^{-\theta|u - \mu|}$. As specified in Section \ref{sec:appli}, we have $\mu(h_{j-1}) = h_{j-1} - c_\sigma$, with $c_\sigma > 0$ a constant only depending on $\sigma$. Since
\begin{align*}
	\log\left(\frac{F(W^{-1}\rho + z)}{F(W^{-1}\rho)} \right)
	&= \sum_{j=2}^{2^{L_n+1}-1} \left[ 
	\log f_\theta( (W^{-1}\rho)_j + z_j \mid \mu( (W^{-1}\rho)_{j-1} + z_{j-1} ) ) \right.\\
	& \quad \left. - \log f_\theta( (W^{-1}\rho)_j \mid \mu( (W^{-1}\rho)_{j-1} ) )\right]\\
	&\quad + \log f_\theta( (W^{-1}\rho)_1 + z_1 \mid \mu_0) - \log f_\theta((W^{-1}\rho)_1 \mid \mu_0).
\end{align*}
and
\begin{align*}
	& \left| \log f_\theta( (W^{-1}\rho)_j + z_j \mid \mu( (W^{-1}\rho)_{j-1} + z_{j-1} ) ) - \log f_\theta( (W^{-1}\rho)_j \mid \mu( (W^{-1}\rho)_{j-1} ) \right|\\
	&\quad = \theta\left| \left[ 
	| (W^{-1}\rho)_j - (W^{-1}\rho)_{j-1} + c_\sigma + z_j - z_{j-1}|  
	-
	| (W^{-1}\rho)_j - (W^{-1}\rho)_{j-1} + c_\sigma|  
	\right] \right|\\
	&\leq \theta|z_j - z_{j-1}|,
\end{align*}
we find, writing $z_0 = 0$, and arguing similarly for the lower bound:
\begin{equation}\label{eq:boundextrafactordeplaplace}
	\exp\left(-\theta \sum_{j=1}^{2^{L_n+1}-1} |z_j - z_{j-1}|\right) \leq \frac{F(W^{-1}\rho + z)}{F(W^{-1}\rho)} \leq \exp\left(\theta \sum_{j=1}^{2^{L_n+1}-1} |z_j - z_{j-1}|\right).
\end{equation}
For \ref{c:cvarsn}, we bound further from above by $\exp\left(2\theta \sum_{j=1}^{2^{L_n+1}-1} |z_j|\right)$. Recalling the notation $z = \tfrac{t}{\sqrt{n}}W^{-1}\psi_n$, with $\psi_n = \psin$ the projection of $b/M_0$ on the Haar system up to level $L_n$. For Theorem \ref{thm:appli}, it suffices to consider the case where $b = \psi_{LK}$ for some $L, K$.  By Lemma \ref{lemdeltah},  the quantity in the exponential in \eqref{eq:boundextrafactordeplaplace} is, in the notation of that Lemma, bounded by 
\[ \frac{|t|}{\rn}\te \sum_{j=1}^{2^{L_n+1}-1} |H_j - H_{j-1}|\leqa |t|\frac{2^{L/2}}{\rn}.\]
This bound is a $o(1)$ uniformly in $|t|\le \log{n}$ and $L\le L_n$, and is also bounded uniformly in $L\le L_n$, by $C|t|\le C(1+t^2)$ for some $C>0$ and thus \ref{c:cvarsn} is verified for the dependent Laplace prior, for any $\be>0$.

For the independent log--Laplace prior,  the right hand side of \eqref{eq:boundextrafactordeplaplace} is replaced by $\exp\left(\theta \sum_{j=1}^{2^{L_n+1}-1} |z_j|\right)$. We apply Lemma \ref{lempsinh} and obtain, for some $C>0$,
\begin{align*}
	\sum_j |z_j| & \le \sum_{j:\, I^{L_n+1}_j\cap S_{LK}\neq \emptyset} C\frac{|t|}{\rn} 2^{L/2} \le | \{j:\, I^{L_n+1}_j\cap S_{LK}\neq \emptyset\}| C\frac{|t|}{\rn} 2^{L/2}\\
	& \le 2^{L_n+1-L} C\frac{|t|}{\rn} 2^{L/2}\leqa C\frac{|t|}{\rn} 2^{L_n+1-L/2}.
\end{align*}
This is bounded uniformly in $L\le L_n$ by $C2^{L_n} |t|/\sqrt{n}=o(|t|)$  for $\be>1/2$. 

To verify \ref{c:cvarb}, we can argue similarly as for the log-Lipschitz wavelet prior in Section \ref{sec:cvarlap}: first one notes that for fixed $L=\cL$, the exponential terms above go to $1$ as can be seen from the obtained upper bounds and recalling $|t|\le\log{n}$. Second, verification of \eqref{picvar} is as before, the rate $\veps_n$ being the same as before.

\subsection{Gaussian wavelet prior and dependent and independent histogram priors\label{sec:changevargauss}}

For this class of priors, we write the main arguments, leaving a few details to the reader (we refer to \cite{Castillo2012} and \cite{Castillo2015}, where similar arguments are used for different Gaussian priors, and for more context on the notion of RKHS). Let $\mh$ be the RKHS associated to the Gaussian prior and $\|\cdot\|_\mh$ its associated norm. It can be checked that for independent Gaussian wavelet priors on $r$, we have $\|g\|^2_\mathbb{H} = \sum_{l,k} \sigma_l^{-2}g_{lk}^2$.  
For histogram priors, a multivariate normal prior $\mathcal{N}(\mu, V)$ on the histogram heights implies a prior $\mathcal{N}(\mu_r, V_r)$ on the Haar wavelet coefficients, where $\mu_r = W\mu$, $V_r = WVW^T$ and $W$ is the matrix described in Section \ref{sec:relatinghistogramsandhaar}. In case of independent heights, we have $\mu = 0$ and $V = I$. For the prior with dependent heights, we have $V_{ij} = \log(1+\sigma^2)\min\{i,j\}$. In that case $\|v\|_\mh^2 := v^T V_r^{-1}v$.
Also, below we use $\zeta:=t\psi_n/\rn$ as a shorthand.

For this class of priors, we need to modify conditions \ref{c:cvarsn} and \ref{c:cvarb} slightly, replacing the set $A_n$ by $A_n'=A_n \cap B_n$, where $B_n$ is given by $B_n = \{r: |\langle \zeta, r - \zeta\rangle_\mathbb{H}| \leq {M} \sqrt{n}\veps_n \|\zeta\|_\mathbb{H}\}$ for suitably large constant $M>0$. We check below that for the priors considered here,
\begin{equation} \label{prgrk}
	\| \zeta\|_\mh^2 = O(t^2).
\end{equation}

Combining this with the fact that $\langle r, \zeta \rangle_\mathbb{H} \sim \mathcal{N}(0, \|\zeta\|_\mathbb{H}^2)$, one deduces $\Pi[B_n^c] \leq e^{-Cn\veps_n^2}$ for some $C > 0$ that can be made arbitrarily large provided the constant $M$ above is large enough.  This implies  $\Pi[B_n^c \mid X] = o_{P_0}(1)$.

In the Gaussian counterpart of \eqref{eq:ratiochangevarstart}, one can 
change variables by setting $\rho = r-\zeta$,  to obtain
\[  \frac{\int_{A_n'} e^{\ell_n(r_t^n) - \ell_n(r_0)}  d\pi(r) } 
{\int e^{\ell_n(r) - \ell_n(r_0)} d\pi(r)}
= \frac{\int_{A_n^\ta} e^{\ell_n(\rho) - \ell_n(r_0)}  
	\exp\left\{ \frac{-\|\zeta\|_\mh^2}{2} - 
	\psg \zeta , \rho \psd_\mh \right\} 
	d\pi(\rho) } 
{\int e^{\ell_n(r) - \ell_n(r_0)} d\pi(r)},
\]
where $A_n^\ta=\ta(A_n')$ with $\ta$  the translation map $\ta:g\to g-\zeta$. By \eqref{prgrk} the term $\|\zeta\|_\mh^2$ is a $O(t^2)$. 
The introduction of the set $B_n$ allows us to bound $|\psg \zeta,\rho \psd_\mh|\le M\rn\veps_n\|\zeta\|_\mh\leqa |t|\veps_n\|\psi_n\|_\mh$.  This is further bounded in the next paragraphs for the considered Gaussian priors.

For the Gaussian wavelet prior, using the bounds for $\psi_{n,lk}=\psg \psi_{LK}/M_0,\psi_{lk} \psd$ in  Lemma \ref{lempsin} and noting that there are $2^{l-L}$ indices $l$ such that $S_{lk}\subset S_{LK}$,
\begin{align*}
	\|\psi_n\|_\mh^2 & \leqa \sum_{l\le L} \sigma_l^{-2} + \sum_{l=L+1}^{L_n}\sigma_l^{-2} 2^{l-L} 
	2^{L/2}2^{-3l/2}.
\end{align*}  
In case $\sigma_l=1$ for $l\le L_n$, one obtains $\|\psi_n\|_\mh^2\leqa L\leqa L_n$, and $t\veps_n\|\psi_n\|_\mh=O(|t|\veps_nL_n)=O(|t|)$. For the choice $\sigma_l = 2^{-l/2}$, this continues to hold as long as $\beta > 1/2$. Combining the previous bounds, one sees that the term $\exp\{ -\|\zeta\|_\mh^2/2 - 
\psg \zeta , \rho \psd_\mh \}$  induced by the change of variables goes to $1$, so that \ref{c:cvarsn} is verified.

For the Gaussian priors defined on heights directly (dependent or not), we consider the eigendecomposition $V = P\Delta P^T$, where $\Delta$ is a diagonal matrix with the eigenvalues of $V$ on the diagonal, and $P$ is matrix with the corresponding eigenvectors of $V$ as its columns. Setting $R_W := 2^\frac{L_n+1}{2}W$ and noting that it is a rotation, we find $V_r^{-1} =  2^{L_n+1}R_WP\Delta^{-1}P^{-1}R_W^{-1}$. We thus have $\|v\|_\mh^2 = v^T V_r^{-1}v = \|V_r^{-1/2}v\|_2^2$, where $V_r^{-1/2} = 2^\frac{L_n+1}{2} R_WP\Delta^{-1/2}P^{-1}R_W^{-1}$, and we bound: 
\begin{align*}
	\|v\|_\mh^2 &= 2^{L_n+1}\|\Delta^{-1/2}P^{-1}R_W^{-1}v\|_2^2
	\leq 2^{L_n+1} \max_{1 \leq i \leq 2^{L_n+1}} \frac{1}{|\la_i(V)|} \|v\|_2^2,
\end{align*} 
with $\la_{i}(V)$ the eigenvalues of $V$. We thus need a bound on the minimum eigenvalue of $V$. In the independent Gaussian case, this is equal to one. For the dependent prior, we find \cite{Anderson1997}:
\begin{align*}
	\|v\|_\mh^2 &\leq 2^{L_n+1} \max_{1 \leq i \leq 2^{L_n+1}} \frac{4\sin^2\left(\frac{(2i-1)\pi}{4\cdot 2^{L_n+1} + 2}\right)}{\log(1+\sigma^2)} \|v\|_2^2 \leq \frac{2^{L_n+3}}{\log(1+\sigma^2)}\|v\|_2^2.
\end{align*} 

From this we deduce the bound, valid for both the dependent and independent case:
$\|\tfrac{t}{\sqrt{n}}\psi_n\|_\mh^2 \leq   \frac{2^{L_n+3}}{\log(1+\sigma^2)}\frac{t^2}{n}\|\psi_n\|_2^2$. Thus we arrive at the bound 
$\|\zeta\|_\mh^2 \leq   \frac{2^{L_n+4}}{\log(1+\sigma^2)}\frac{t^2}{n}$,
valid for both the independent prior (where $2^{L_n+4}$ can be improved to $2^{L_n+2}$) and the dependent prior. This shows \eqref{prgrk}, and one should also check that $\veps_n\|\psi_n\|_\mh$ is bounded, which follows from the previous bound for $\be>1/2$. We can now follow the same arguments as for the case where the prior is defined on the wavelet coefficients, independently.

Finally, one verifies that condition \ref{c:cvarb} holds for Gaussian priors. This amounts to check that the intersecting set $B_n$ still has posterior mass going to one  after applying the  translation $g\to g-\zeta$: this is verified similarly as for the log-Laplace prior and details are left to the reader (see e.g. \cite{ic14} for a similar argument).

\subsection{Gamma dependent and independent histogram priors\label{sec:changevargamma}}

In the dependent case, we can use the same argument as for the dependent log-Laplace prior in Section \ref{sec:loglaplacechangeofvar}, with minor adaptations. The density on the $h$-scale is  given by $f(h_{i+1} \mid \alpha, h_{i}) = \frac{\alpha^\alpha }{\Gamma(\alpha)} e^{\alpha[h_{i+1} - h_i - e^{h_{i+1} - h_i}]}$. We  bound a single term
\begin{align}
	\log&\{ f(z_j + (W^{-1}\rho)_j \mid \alpha, z_j + (W^{-1}\rho)_{j-1} ) \} \notag\\
	&\quad - \log\{ f( (W^{-1}\rho)_j \mid \alpha, z_j + (W^{-1}\rho)_{j-1} ) \} \notag \\
	&= \alpha(z_j - z_{j-1}) + \alpha(1 - e^{z_j-z_{j-1}}) e^{(W^{-1}\rho)_j - (W^{-1}\rho)_{j-1}}. \label{trgamma}
\end{align}
In the log-Laplace case above, we already showed that $\sum_j |z_j-z_{j-1}|$ 
is bounded by  $|t|2^{L/2}/\rn$ which is $o(1)$ uniformly over $|t|\le \log{n}$ and $L\le L_n$.

We now have an extra term involving $(W^{-1}\rho)_j - (W^{-1}\rho)_{j-1}$. Using the relationship $\rho = W(h-z)$, we find
$ (W^{-1}\rho)_j - (W^{-1}\rho)_{j-1} = h_j - h_{j-1} - (z_j - z_{j-1})$. Now, for $\be>1/2$, one can use Lemma \ref{lem:consistsn}, which implies supremum--norm consistency for $r$ on $A_n$: in particular, the log--hazard $r$ and the corresponding histogram heights are bounded by a universal constant. We thus find that $|h_j-h_{j-1}|$ is a $O(1)$. Deduce that \eqref{trgamma} is bounded from above by a constant times (using $1-e^u\le 2u$ for small $u$)
\[ |z_j - z_{j-1}|+|z_j - z_{j-1}|e^{O(1)}. \]
By summing over $j$,
one sees that the proof proceeds like the log-Laplace case otherwise for both conditions \ref{c:cvarsn} and \ref{c:cvarb}.

In the independent case, the density on the $h$-scale is given by $f(h_i \mid \alpha_0, \beta_0) = \beta_0^{\alpha_0} e^{\alpha_0 h_i - \beta_0 e^{h_i}}/\Gamma(\alpha_0)$ for all $i$. The difference corresponding to \eqref{trgamma} then becomes
$$\alpha_0 z_j + \beta_0(1 - e^{z_j})e^{(W^{-1}\rho)_j},$$
and the argument proceeds similarly, using this time that, as proved for the independent log--Laplace prior, we have $\sum_j |z_j|\leqa |t|2^{L_n}/\rn=o(1)$ uniformly in $|t|\le\log{n}$ for $\be>1/2$, as well as $|W^{-1}\rho_j|\le |h_j|+|z_j|$, for which we can use the supremum--norm consistency as above to argue that $|h_j|$s are uniformly bounded.  From this one verifies \ref{c:cvarsn} and \ref{c:cvarb} as above, which concludes the verifications for the change of variable condition for all priors.

\section{\label{sec:cfun}Approximations of $\psi_b=b/M_0$}

\subsection{Approximation Lemmas}

\begin{lem} \label{lem:psi}
	Let $(\psi_{lk})$ be one of the considered wavelet bases (Haar or CDV).  Let $L_n$ be the prior cut--off and let $L\le L_n$ and $0\le K<2^L$. 
	For $b\in L^\infty$, we recall the notation $\psi_b=b/M_0$ and $ \psi_{b,L_n}=P_{L_n}(b/M_0)$. 
	For $\|\cdot\|_L$ the LAN norm, the following holds.
	\begin{enumerate}
		\item There exists a constant $C$, depending on the wavelet basis and $M_0$  only, such that for any $b\in L^\infty$, possibly depending on $n$,
		\[ \|\psi_{b,L_n}\|_\infty\le CL_n \|b\|_\infty,\qquad  \|\psin\|_2 \le C \|b\|_2.\]
		\item For any {\em fixed} bounded $b$, as $n\to \infty$ and $L_n\to\infty$,
		\begin{align*} \|\psi_{b,L_n}\|_L & \to \|\psi_{b}\|_L,\\
			W_n(\psi_{b,L_n}-\psi_b)& =\op. 
		\end{align*}
		In particular, $\|\psi_{b,L_n}\|_L$ and $\|\psi_{b,L_n}\|_2$ are uniformly bounded.   
	\end{enumerate}
\end{lem}
\begin{proof}
	For the first point,  one bounds from above
	\[ \|\psi_{b,L_n}\|_\infty \le C\sum_{l\le L_n} 2^{l/2} \max_{k} |\psg  \frac{b}{M_0}, \psi_{lk} \psd| \leqa L_n\|b\|_\infty\|M_0^{-1}\|_\infty ,\]
	where we use for a bounded function $h$ that $|\psg h,\psi_{lk}\psd|\le \|h\|_\infty\|\psi_{lk}\|_1\leqa \|h\|_\infty 2^{-l/2}$. Also,  $\|\psin\|_2^2 \le \|\eif\|_2^2\le \|b\|^2\|M_0^{-1}\|_\infty$. This gives the result as $1/M_0$ is bounded. 
	
	For the second point, one notices that the LAN norm is bounded from above by a constant times the $\|\cdot\|_2$--norm, as $M_0\la_0$ is bounded by assumption, and the first convergence follows from $\|\psi_b-\psi_{b,L_n}\|_2\to 0$ by definition (as here $b$ is a fixed element of $L^\infty\subset L^2$). Similarly, $W_n(\psi_{b,L_n}-\psi_b)$ is centered with variance under $P_{\la_0}$ equal to $\|\psi_{b,L_n}-\psi_b\|_L^2=o(1)$, so is a $\op$.  
\end{proof}

\begin{lem} \label{lem:checkpsi}
	Under the same notation as in Lemma \ref{lem:psi},
	\begin{enumerate}
		\item For $b=\psi_{LK}$, uniformly over $L\le L_n$ and $K$, for some $C>0$,
		\begin{align*}
			\|\psi_{b,L_n}\|_2 \vee \|\psi_{b,L_n}\|_L & \le C\\
			\|\psi_{b,L_n}\|_\infty & \le C 2^{L/2} L_n, \\
			\| \psi_b - \psi_{b,L_n} \|_\infty & \le C_2 2^{L/2}2^{-L_n}.
		\end{align*}
		\item For $b\in \cH(\mu, D)$ for some $\mu,D>0$, with $\mu'=\mu\wedge 1$,
		\[\| \psi_b - \psi_{b,L_n} \|_\infty  \leqa 2^{-\mu' L_n}. \]
	\end{enumerate}
\end{lem}
\begin{proof}
	For the first point and first two inequalities, one applies Lemma \ref{lem:psi} with $b=\psi_{LK}$, for which $\|\psi_{LK}\|_2\vee \|\psi_{LK}\|_L\leqa C$ and 
	$\|\psi_{LK}\|_\infty\leqa 2^{L/2}$. 
	
	For the last inequality of the first point,  one can bound
	\begin{align*}
		\| \psi_b - \psi_{b,L_n} \|_\infty & \le
		\sum_{l > L_n}
		2^{l/2} \max_{k}\left|\psg \frac{\psi_{LK}}{M_0} , \psi_{lk} \psd \right|.
	\end{align*}
	One further splits, for $h:=M_0^{-1}$, and denoting by $\bar{h}$ the mean of $h$ on the support of the wavelet $\psi_{lk}$,
	\[ \psg \frac{\psi_{LK}}{M_0} , \psi_{lk} \psd
	= \psg \psi_{LK} (h-\bar h), \psi_{lk} \psd + 
	\bar{h} \psg \psi_{LK} , \psi_{lk} \psd.
	\]
	As $l>L_n\ge L$ by definition, the last bracket is zero by orthogonality of the wavelet basis. We note that $h=M_0^{-1}$ is Lipschitz by Lemma \ref{lem:lip}.   So, for all $x$ in the support $S_{lk}$ of $\psi_{lk}$, there exists $c$ in $S_{lk}$ such that 
	\[ |h(x)-\bar{h}|=|h(x)-h(c)|\le C_0 |x-c|\le C_0 2^{-l}, \]
	using that $S_{lk}$ has diameter at most $C2^{-l}$, so that, using $\|\psi_{LK}\|_\infty\leqa 2^{L/2}$ and $\|\psi_{lk}\|_1\leqa 2^{-l/2}$,
	\[ |\psg \psi_{LK} (h-\bar h), \psi_{lk} \psd|
	\le C2^{L/2}2^{-l}2^{-l/2}=C2^{L/2-3l/2}.
	\]
	Combining the previous bounds leads to
	\[ \| \psi_b - \psi_{b,L_n} \|_\infty  \leqa
	\sum_{l > L_n} 2^{L/2-l} \leqa 2^{L/2-L_n}.  \]
	For the second point, for $b\in \cH(\mu, D)$, 
	\begin{align*}
		\|\psi_b-\psi_{b,L_n}\|_\infty &\leqa \sum_{l>L_n} 2^{l/2} \max_{k} 
		|\psg b/M_0 , \psi_{lk} \psd| \\
		&\leqa \sum_{l>L_n} 2^{l/2}  2^{-l(1/2+\mu')}\leqa 2^{-L_n\mu'}. \qedhere
	\end{align*}
	using that, since $M_0^{-1}$ is Lipschitz by Lemma \ref{lem:lip}, we have $b/M_0\in \cH(\mu', D)$.
	
\end{proof}

\subsection{Smooth linear functional example} \label{sec:linfunex}

Let us first check the change of variables condition \ref{c:cvarb}. By the arguments in Section \ref{sec:changevar}, for \ref{c:cvarb} to hold it is enough to bound, with $\psi_{n,lk} = \langle b/M_0, \psi_{lk}\rangle,$
\[ \frac{t}{\rn} \sum_{l\le L_n,\, k} \frac{|\psi_{n,lk}|}{\sigma_l}
\le \frac{t}{\rn} \sum_{l\le L_n} 2^{l(1/2-\mu')}{\sigma_l}, \]
using Lemma \ref{lem:cvarlap}, part 2. For $\sigma_l=1$, this always goes to $0$ as $2^{L_n/2}=o(\rn)$. For $\sigma_l=2^{-l/2}$, this goes to $0$ when  $2^{(1-\mu')L_n}=o(\rn)$, which happens if $\ga+\mu'>1/2$, that is also if $\ga+\mu>1/2$.

One now checks condition \ref{c:funb}.  By Lemma \ref{lem:checkpsi}, we have $\|\psi_b-\psi_{b,L_n}\|_\infty\leqa 2^{-L_n\mu'}$, with $\mu'=1\wedge \mu$. So, it is enough to check that 
$\rn \veps_n 2^{-L_n\mu'}=o(1)$. Using the upper-bound $\veps_n$ for the rate obtained  in Theorem \ref{thm:appli}, one obtains the condition $\mu'+ 
\be\wedge\ga > 1/2 + \ga$.

\section{Background on efficiency and centering} \label{sec:backgroundeff}

For a given function $b\in L^2(\Lambda_0)$, consider estimating the functional $\psi(b)=\int_0^1b\la_0=\La_0b$. The efficient influence function for estimating $\psi(b)$ can be seen to be (see \cite{aad98}, Chapter 25)
\[ \tilde\psi_b(Y,\delta)=\delta\frac{b}{M_0}(Y)-\left[\Lambda_0\left(\frac{b}{M_0}\right) \right](Y). \]
Note that $W_n(b/M_0)=W_n(\eif)=\tilde\psi_b(Y,\delta)$, for $\eif=b/M_0$. In particular, an estimator $\hat\psi_b$ of $\psi(b)$ is asymptotically efficient if, as $n\to\infty$,
\[ \hat\psi_b = \psi(b) + \frac1\rn  W_n(\psi_b) + \op. \]

{\em Comparing $T_n$ and $\la_{L_n}^*$.} By definition \eqref{lastar}, 
\[ \la_{L_n}^*=\la_{0,L_n}+\frac{1}{\rn}\sum_{L\le L_n} \sum_{0\le K<2^L} W_n(\psi_{n,LK}) \psi_{LK}. \]
Recall that the quantity $T_n$ is  defined similarly, with $\psi_{LK}$ in place of  $\psi_{n,LK}$, see  \eqref{tn}. We find that in the space $\cM_0$, one can use either the first or the second centering.
\begin{lem} \label{lem:tn}
	Let $T_n$ be defined by \eqref{tn} and $\la^*_{L_n}$ by \eqref{lastar}. Then for any admissible sequence $w=(w_l)$,
	\[ nE_{\la_0} \| T_n-\la_{L_n}^*\|_{\cM_0(w)}^2 = o(1). \]
	As a consequence, $\cB_{\mathcal M_0}\left(\Pi(\cdot\given X)\circ\tau_{T_n}^{-1}, \Pi(\cdot\given X)\circ\tau_{\la^*}^{-1}\right)=\op$.
\end{lem}
\begin{proof}
	By definition of both quantities, we have 
	\[ \rn(T_n-\la_{L_n}^*) = \sum_{L\le L_n, K} W_n(\psi_{LK}/M_0-\psi_{n,LK}) \psi_{LK}.\]
	Using the inequality $\|f\|_{\cM(w)}^2 \le \sum_{l,k} w_l^{-2} f_{lk}^2$, one sees that
	\begin{align*} 
		nE_{\la_0} \| T_n-\la_{L_n}^*\|_{\cM(w)}^2 & \le \sum_{L\le L_n, K} w_L^{-2} E_{f_0}[ W_n(\psi_{LK}/M_0-\psi_{n,LK})^2 ] \\
		& =  \sum_{L\le L_n, K} w_L^{-2} \|\psi_{LK}/M_0-\psi_{n,LK}\|_L^2. 
	\end{align*}
	Using the $L^\infty$--bound from Lemma \ref{lem:checkpsi}, one gets
	\begin{align*}
		\| \psi_{LK}/M_0-\psi_{n,LK} \|_L^2 & \leqa
		\| \psi_{LK}/M_0-\psi_{n,LK} \|_\infty^2  \leqa 2^{L}2^{-2L_n}.
	\end{align*}
	From this one concludes that for any admissible sequence $(w_l)$, that is such that $w_l/\sqrt{l}$ increases to infinity,
	\begin{align*}
		nE_{\la_0} \| T_n-\la_{L_n}^*\|_{\cM(w)}^2 & \le 
		2^{-L_n} \sum_{L\le L_n} w_L^{-2}2^L 2^{L-L_n} \\
		& \le 2^{-L_n} \sum_{L\le L_n} 2^{L}/L = O(L_n^{-1})=o(1).
	\end{align*}
	Using the definition of the bounded Lipschitz metric,
	\begin{equation}\label{centerings}
		\left|\cB_{\mathcal M_0(w)}\left(\Pi(\cdot\given X)\circ\tau_{T_n}^{-1}, \Pi(\cdot\given X)\circ\tau_{\la^*}^{-1}\right)\right| \le \rn\|T_n-\la^*\|_{\cM(w)},
	\end{equation} 
	and the last bound is $o_{P_0}(1)$ by the previous computation.
\end{proof}

\begin{lem} \label{lem:diff}
	Let $T_n$ be as in \eqref{tn} with cut--off $L_n$ as in \eqref{eln}, for $\ga>0$. Let  $\mathbb{T}_n(t) =  \int_0^t T_n(u) du$, $t\in[0,1]$ and set
	\[ \La^*(t)= \La_0(t)+\frac{1}{\rn} W_n\left(\1_{\cdot\le t}/M_0(\cdot) \right),\quad t\in[0,1],\]
	where $W_n$ is as in \eqref{eq:LAN}. Then, if $\ga<\be+1/2$, as $n\to\infty$,
	\[ \rn\|\mathbb{T}_n-\La^*\|_\infty=\op. \]
\end{lem}
\begin{proof}
	By definition of  $T_n$ and $\mathbb{T}_n$, and writing $P_{L_n^c}\la=\la-P_{L_n}\la$ the projection of $\la$ onto the orthocomplement of $\cV_{L_n}$,
	\begin{align*}
		\mathbb{T}_n(t)  & = \int_0^t (P_{L_n}\la_0)(u)du 
		+ \frac{1}{\rn} \sum_{L\le L_n,\, K} W_n(\psi_{LK}/M_0) \int_0^t\psi_{LK}(u)du\\
		& = \La_0(t) - \int_0^t (P_{L_n^c}\la_0)(u)du +  \frac{1}{\rn}
		W_n\left(\frac{1}{M_0(\cdot)}  \sum_{L\le L_n,\, K} \psg \1_{[0,t]},\psi_{LK}\psd \psi_{LK}(\cdot) \right)\\
		& = \La_0(t) - \int_0^t (P_{L_n^c}\la_0)(u)du +  \frac{1}{\rn}
		W_n\left(\frac{(P_{L_n} \1_{[0,t]}) }{M_0}(\cdot)  \right).
	\end{align*}
	From this one deduces, using the definition of $\La^*$ above, for any $t\in[0,1]$,
	\[ \rn(\mathbb{T}_n(t)-\La^*(t)) = -\rn\int_0^1 (P_{L_n^c}\1_{[0,t]})(u)(P_{L_n^c}\la_0)(u)du 
	+  W_n\left(\frac{(P_{L_n^c} \1_{[0,t]}) }{M_0}(\cdot)  \right),\]
	where one uses $ \int_0^t P_{L_n^c}\la_0 = \int \1_{[0,t]} P_{L_n^c}\la_0 = \int 
	P_{L_n^c} \1_{[0,t]} P_{L_n^c}\la_0$ since $P_{L_n} \1_{[0,t]}$ and $P_{L_n^c}\la_0$ are orthogonal in $L^2[0,1]$. We deal with each term in the previous display separately. For the first, one uses the bound $\int fg \le \|f\|_\infty\|g\|_1$. Since $\la_0$ is $\be$--H\"older, we have
	\[ \| P_{L_n^c}\la_0\|_\infty \le \sum_{l>L_n} 2^{l/2} \max_{k}|\la_{0,lk}|\leqa 
	2^{-L_n\be}. \] 
	On the other hand, since $\|\psi_{lk}\|_1\leqa 2^{-l/2}$,
	\begin{align*}
		\| P_{L_n^c} \1_{[0,t]} \|_1 & \le  \sum_{l>L_n,\ k} |\psg  \1_{[0,t]},\psi_{lk}\psd| \int_0^1 |\psi_{lk}(u)|du \\
		& \leqa \sum_{l>L_n} 2^{-l/2} \left[\sup_{t\in[0,1]} \sum_k |\psg  \1_{[0,t]},\psi_{lk}\psd|\right]
		\leqa \sum_{l>L_n} 2^{-l} \leqa 2^{-L_n},
	\end{align*}  
	where we use that the supremum under brackets in the last display is bounded from above by a constant time $2^{-l/2}$, as shown e.g. in the proof of Lemma 3 in \cite{gnaop09}. We now turn to bounding the term involving $W_n$ in the identity for $\rn(\mathbb{T}_n(t)-\La^*(t))$ above. To do so, one notes 
	\[ W_n(a) =\mathbb{G}_n(\Psi(a;X)),\qquad \Psi(a;X):=\Psi(a;(\delta,Y))=
	\delta a(Y)-(\La_0 a)(Y). \]
	We wish to apply Lemma \ref{lembra} to the empirical process above. To bound the bracketing entropy, one first notes that for $a,b$ bounded functions,
	\begin{equation} \label{tecemp}
		E_{\la_0}(\Psi(a;X)-\Psi(b;X))^2 \leqa \int_0^1(a-b)^2,
	\end{equation} 
	where one uses $0\le \delta\le 1$ and, since $(\La h)(Y)=\int_0^1 \1_{u\le Y}h(u)\la_0(u)du$ for a given bounded function $h$, using Cauchy-Schwarz inequality,
	\[ E (\La h)(Y)^2 \le \int_0^1h(u)^2\la_0(u)^2du \leqa \int_0^1 h(u)^2 du.\]
	One wishes to bound $\|\mathbb{G}_n\|_{\cF_n}=\sup_{f\in\cF_n} |\mathbb{G}_n(f)|$, with \[ \cF_n=\{f_t:=\Psi(M_0^{-1} P_{L_n^c} \1_{[0,t]}),\, 
	t\in [0,1]\}.\] For any $f\in\cF_n$, we now bound $\int f^2 dP_{\la_0}$, $\|f\|_\infty$ and $\cJ_{[]}(\delta,\cF_n,L^2(P_{\la_0}))$. By using \eqref{tecemp} with $b=0$ and  that $M_0^{-1}$ is bounded, we have for $t\in[0,1]$,
	\[ \int f_t^2dP_{\la_0}\leqa \int_0^1 \left(M_0^{-1} P_{L_n^c} \1_{[0,t]}\right)^2 
	\leqa \|P_{L_n^c} \1_{[0,t]}\|_2^2.  \]
	Now, using $|\psg \1_{[0,t]},\psi_{lk}\psd|\le \|\psi_{lk}\|_1\leqa 2^{-l/2}$, one can bound
	\begin{align*}
		\|P_{L_n^c} \1_{[0,t]}\|_2^2 &
		= \sum_{l>L_n,\, k} \psg \1_{[0,t]},\psi_{lk}\psd^2\le
		\sum_{l>L_n} \sum_{k} 2^{-l/2} |\psg \1_{[0,t]},\psi_{lk}\psd| \\
		& \leqa \sum_{l>L_n} 2^{-l/2} \sup_{t\in[0,1]} \sum_{k}  |\psg \1_{[0,t]},\psi_{lk}\psd|\leqa \sum_{l>L_n} 2^{-l}\leqa 2^{-L_n}=:\bar\delta,
	\end{align*}
	where we use as above that the supremum in the last display is bounded by $C2^{-l/2}$. Also, for $f=f_t\in\cF_n$, we have  $\| f_t \|_\infty \leqa 
	\|P_{L_n^c} \1_{[0,t]}\|_\infty$. This last quantity is at most $CL_n$. Indeed, as $\1_{[0,t]}=P_{L_n} \1_{[0,t]}+P_{L_n^c} \1_{[0,t]}$ almost surely, this follows from $\|\1_{[0,t]}\|_\infty=1$ and
	\begin{align*}
		|P_{L_n} \1_{[0,t]}|& \le \sum_{l\le L_n} 2^{l/2} \max_{k} |\psg \1_{[0,t]},\psi_{lk}\psd|\leqa \sum_{l\le L_n}1\leqa L_n.
	\end{align*}
	It remains to bound the entropy $\cJ_{[]}(\delta,\cF_n,L^2(P_{\la_0}))$. Let $f_s, f_t\in\cF_n$, $0\le s\le t\le 1$. Then proceeding similarly as for the bound on  $\|P_{L_n^c} \1_{[0,t]}\|_2$ above, noting that $| \psg \1_{[s,t]},\psi_{lk}\psd|\leqa |s-t|^{1/2}$ by Cauchy-Schwarz inequality, one has
	\[ \|f_t-f_s\|_{L^2(P_{\la_0})}^2\leqa \sqrt{t-s}2^{-L_n/2},\]
	from which one deduces $N_{[]}(\veps,\cF_n,L^2(P_{\la_0}))\leqa \veps^{-4}2^{L_n}$ and then for small $\delta>0$,
	\[ \cJ_{[]}(\delta,\cF_n,L^2(P_{\la_0})) \leqa \sqrt{L_n}\delta + \delta\log(1/\delta).\]
	An application of Lemma \ref{lembra} now gives, with $\bar\delta=2^{-L_n}$ as above,
	\[ \|\mathbb{G}_n\|_{\cF_n} \leqa  j(\bar\delta)\left(1+\frac{j(\bar\delta)L_n}{\bar{\delta}^2\sqrt{n}}\right) \leqa L_n 2^{-L_n}+L_n^3/\rn=o(1). \]
	Putting the previous bounds on the different terms together leads to
	\[\|\rn(\mathbb{T}_n-\La^*)\|_\infty \le \rn 2^{-L_n(1+\be)} + \|\mathbb{G}_n\|_{\cF_n}.  \]
	The last bound is a $o(1)$ as soon as the first term goes to $0$, which happens as soon as $(1+\be)/(1+2\ga)>1/2$, or equivalently $\ga<\be+1/2$ as announced. 
\end{proof}

\section{$\ $ Lower bound for hazard rate in supremum norm}

Huber and McGibbon \cite{huber04} proved lower bounds for hazard estimation in terms of $L^p$--losses, $p<\infty$, but not $p=\infty$. 
The following result gives the $L^\infty$--counterpart, with the rate featuring the expected additional logarithmic term. 
\begin{thm} \label{lbhazard}
	Let $\be, L>0$. There exists a finite constant $M=M(\be,L)>0$ such that for large enough $n$,
	\[ R_M=\inf_T \sup_{\la \in\cH(\be,L)} E_\la \|T - \la\|_\infty \ge M \left(\frac{\log{n}}{n}\right)^{\frac{\be}{2\be+1}}=M\veps_{n,\be}^*, \]
	where the infimum is taken over all possible estimators $T=T(X)$ of the hazard rate in the survival model and $E_\la$ denotes the expectation with respect to the law of the observations in the survival model with hazard $\la$ on $[0,1]$.
\end{thm}
\begin{remark}
	Theorem \ref{thm:sn} assumes slightly more than $\la\in \cH(\be,L)$ for some $\be, L>0$ and requires that $r=\log\la$ belongs to $\cH(\be,D)$ for some $D>0$. The lower bound proof below goes through for that subclass of hazards $\la$ (it suffices to note that all considered $\la_j$'s in that proof are bounded away from $0$ and $\infty$), thus showing the matching lower bound in the considered setting as well. 
\end{remark}
\begin{proof}
	We use the principle of lower bounds `based on many hypotheses' as developed by Ibragimov and Has'minskii \cite{ibra77} (see also \cite{tsybook}, Section 2.6). Suppose one can find hazards $\la_0,\la_1,\ldots,\la_M$ with $M\ge 2$ such that, for some $\al\in(0,1/8)$,
	\begin{align}
		\| \la_i - \la_j \|_\infty & \ge 2s>0\quad 0\le i<j\le M \label{mhyp1}\\
		\frac{1}{M}\sum_{j=1}^M K(P_{\la_j},P_{\la_0}) & \le \al \log M, \label{mhyp2}
	\end{align}
	where $K(P,Q)$ denotes the Kullback--Leibler divergence. 
	Then by the lower bound principle based on many hypotheses (\cite{tsybook},  Thm. 2.7), the minimax risk as in the statement is bounded from below by
	\[ R_M \ge C s, \]
	where $C=C(\al)>0$ is a constant depending only on $\al$. We now construct the hazards $\la_k$. One sets $\la_0=1$ a constant hazard and consider perturbations $\la_k$, $1\le k\le M$ as follows. Let $\psi$ be a smooth function verifying the following conditions: $\psi$ has compact support $[-1/2,1/2]$, 
	and $\psi\in\cH(\be,1)$ and $\psi(0)=c>0$ a small enough constant (see \cite{tsybook}, Section 2.7 for an explicit construction). Next set 
	\[ \la_k=\la_0+Lh^\be\psi\left(\frac{x-x_k}{h}\right),\ x_k=\frac{k-1/2}{M},\ h=1/M. \]
	
	By construction $\la_k$ are hazard functions (for small enough $h$) and $\la_k\in \cH(\be,L)$. Further, $\|\la_k-\la_j\|_\infty=cLh^\be$.  
	One uses the expression of the Kullback-Leibler divergence $K(P_{\la_1},P_{\la_2})$ for two given hazard rates $\la_1,\la_2$ bounded away from $0$ (say; and with the model parameters satisfying our general assumptions) given in Lemma \ref{lemkl}, combined with $K(P_{\la_1}^{\otimes n},P_{\la_2}^{\otimes n})=nK(P_{\la_1},P_{\la_2})$. 
	The computation is similar to that in \cite{huber04}, Lemma 1, up to the fact one works with the survival model up to time $\ta$ only (in the version of  \cite{huber04} we could consult, we note the presence of a minor sign typo in the expressions of Hellinger distance and KL divergence). Using Lemma \ref{lemkl}, one sees that it is enough to bound $\int (\la_k-\la_0)^2/\la_0 S_{\la_0}\bar{G} \le C\int(\la_k-\la_0)^2$ (as the survival function for the constant hazard as well as $\bar{G}$ are bounded).  As $\|\la_k-\la_0\|_2^2\le CL^2h^{2\be+1}$, deduce that
	\[  \frac{1}{M}\sum_{j=1}^M K(P_{\la_j},P_{\la_0}) \le CL^2nh^{2\be+1}. \]
	Hence choosing $h=(\delta \log{n}/n)^{1/(2\be+1)}$, the last bound can be made smaller than $\al\log{M}$, provided $d$ is a small enough constant (depending on the chosen value of $\al<1/8$), which concludes the proof.
\end{proof}

\begin{lem} \label{lemkl}
	In the survival model with respective distributions $P_{\la_1}, P_{\la_2}$, 
	\[ K(P_{\la_1}, P_{\la_2}) = \int \left[\log\left(\frac{\la_1}{\la_2}\right)+1 - \frac{\la_2}{\la_1}\right]\la_1 S_1\bar{G}, \]
	assuming hazards are bounded away from $0$.  
	If $\|(\la_1-\la_2)/\la_1\|_\infty\le 1/2$, there exists $C>0$ such that
	\[ K(P_{\la_1}, P_{\la_2}) \le C\int \left( \frac{\la_2-\la_1}{\la_1} \right)^2 \la_1 S_1\bar{G}.
	\]
\end{lem}
\begin{proof}
	From the expression of the density in the survival model it follows 
	\begin{align*}
		K(P_{\la_1}, P_{\la_2}) = \int_0^1 & \log(S_1/S_2) gS_1
		+\int_0^1 \log\left(\frac{\la_1S_1}{\la_2S_2}\right)\bar{G}\la_1S_1\\
		& +\log\left(\frac{S_1(1)}{S_2(1)}\right)\bar{G}(1)S_1(1).
	\end{align*}
	One splits the second integral in the last display by writing $\log\{(\la_1S_1)/(\la_2S_2)\}=\log(\la_1/\la_2)+\log(S_1/S_2)$. By integrating by parts the integral $\int_0^1 \log(S_1/S_2)\bar{G}\la_1S_1$, the third term in the previous display cancels (using $S_1(0)=S_2(0)=1$) and by rearranging the obtained expression for the KL divergence one obtains 
	\[ K(P_{\la_1}, P_{\la_2}) = \int\log\left(\frac{\la_1}{\la_2}\right)\la_1 S_1\bar{G}
	- \int(\la_2-\la_1)S_1\bar{G}, \]
	which gives the first part of the lemma. The second part follows by using the inequality $|\log(1+x)-x|\le Cx^2$ for $|x|\le 1/2$. 
\end{proof}

\section{$\ $ Supporting lemmas}

\subsection{$\ \ $Smoothness of $M_0$}

\begin{lem} \label{lem:lip}
	Recall that $M_0(u)=E_{\la_0}[\1_{u\le Y}]$ for $u\in[0,1]$. Under assumption \ref{c:mod}, the maps $M_0(\cdot)$ and $M_0^{-1}(\cdot)$ are Lipschitz on $[0,1]$.
\end{lem}
\begin{proof}
	By definition $Y=T\wedge C$ and $T, C$ are independent, so $M_0(u)=\bar{G}(u)e^{-\La_0(u)}$ for $u\in[0,1]$. Note that
	$\bar{G}(u)=1-\int_0^ug(v)dv$. As $g$ is bounded by assumption, $\bar{G}$ is Lipschitz on $[0,1]$. The same holds for $e^{-\La_0}$, which is  $\cC^1$ as $\la_0$ is continuous. 
	So $M_0$ is Lipschitz, as a product of Lipschitz maps. The same is true for $M_0^{-1}$, noting that $|M_0(x)^{-1}-M_0(y)^{-1}|\leqa |M_0(x)-M_0(y)|$ for $x,y\in[0,1]$, as $M_0$ is bounded away from $0$ under \ref{c:mod}.
\end{proof}

\subsection{$\ \ $Bounds on empirical processes}

Recall the definition of the bracketing integral of a class of functions $\cF$ with respect to a norm $\|\cdot\|$, as in \cite{vvw96}, with $N_{[]}$ denoting the usual bracketing number,
\begin{equation} \label{bra}
	\cJ_{[\,]}(\delta,\cF,\|\cdot\|)=\int_0^\delta
	\sqrt{1+\log N_{[\,]}(\veps,\cF,\|\cdot\|)} d\veps. 
\end{equation}

\begin{lem}[variation on \cite{aad98}, Example 19.11] \label{lembe}
	Let $\cF(M_V)$ be the set of all functions $f:\RR\to\RR$ with $f(0)=0$ and total variation bounded by $M_V\ge \veps$. Then there exists a constant $K$ such that for every distribution $P$ and $\veps>0$,
	\[ \log N_{[\,]}(\veps,\cF(M_V),L^2(P))\le K \frac{M_V}{\veps}. \]
\end{lem} 
\begin{proof}
	
	For a function $f\in \cF(M_V)$ as in the statement, we have $f/M_V\in \cF(1)$, so that
	\[ \log N_{[\,]}(\veps,\cF(M_V),L^2(P))
	\le \log N_{[\,]}\left(\frac{\veps}{M_V},\cF(1),L^2(P)\right) \]
	and the latter expression is at most $KM_V/\veps$ by Example 19.11 in \cite{aad98}, which holds uniformly for bounded $\veps'=\veps/M_V$.
\end{proof}

In the next statement, $E^*_{\la_0}$ denotes the outer--expectation under $P_{\la_0}$, and one denotes $\|\mathbb{G}_n \|_\cF:=\sup_{f\in\cF} |\mathbb{G}_n f|$.
\begin{lem} \label{lemep}
	Let $\cF$ be a class of functions of bounded variation on $[0,1]$ verifying the following for any $f\in\cF$ and some $\mu>0$,
	\[ f(0)=0,\quad \|f\|_{BV}=: \int_0^1 |f'(u)|du \le \mu.\]
	There exists $c_1>0$ such that for  $\cF,\mu$ as above, 
	\[ E^*_{\la_0} \|\mathbb{G}_n \|_\cF \le c_1\mu.\]
	In particular if $\mu=\mu_n=o(1)$, we have $E^*_{\la_0} \|\mathbb{G}_n \|_\cF=o(1)$. 
\end{lem}
\begin{proof}
	One first notes that for any $f\in\cF$, we have $\int f^2 dP_{\la_0} \le \|f\|_\infty^2$ and in turn $\|f\|_\infty\le |f(0)|+\|f\|_{BV} \le \mu$, so that $\int f^2 dP_{\la_0}\le \mu^2$. 
	
	Next one remarks that it is enough to prove the Lemma when $\mu=1$, as otherwise one can consider the set $\mathcal{G}=\{g=f/\mu,\, f\in\cF\}$ and $\|\mathbb{G}_n\|_{\cF}=\mu\|\mathbb{G}_n\|_{\mathcal{G}}$. 
	Setting $\mu=1$ and using Lemma  \ref{lembra} below applied with $\delta=M=1$, deduce that 
	\[ E^*_{\la_0}\|\mathbb{G}_n \|_\cF \le C\cJ_{[\,]}(1,\cF,L^2(P_{\la_0}))\left(
	1+\frac{\cJ_{[\,]}(1,\cF,L^2(P_{\la_0}))}{\rn} \right). \] 
	For $\veps\le 1$, Lemma \ref{lembe} bounds the bracketing entropy $\log N_{[]}(\veps,\cF,L^2(P_{\la_0}))$ from above by $K/\veps$, so that 
	\begin{align*}
		\cJ_{[\,]}(1,\cF,L^2(P_{\la_0})) & \le \int_0^1 \sqrt{1+K/\veps}d\veps,
	\end{align*}
	which is bounded by a constant. This concludes the proof for $\mu=1$ and then for any $\mu>0$ arguing as above. 
\end{proof}

\begin{lem}[Lemma 3.4.2 in \cite{vvw96}] \label{lembra}
	Let $\cF$ be a class of measurable functions such that for any $f\in\cF$, 
	\[ \int f^2 dP<\delta^2,\quad \|f\|_\infty\le M. \]
	Then for $j(\delta):=\cJ_{[\,]}(\delta,\cF,L^2(P_{\la_0}))$ as defined  in \eqref{bra},
	\[ E^*_{P_{\la_0}}\|\mathbb{G}_n \|_\cF \leqa j(\delta)\left(1+\frac{j(\delta)M}{\delta^2\rn}\right).\] 
\end{lem}

\begin{lem}\label{lem:parttwo} For any $n \geq 2$, cut--off $L_n\ge 1$, and
	$\la^*$ as in \eqref{lastar}, it holds
	\begin{equation*}
		E_{\la_0} \|\la_{L_n}^* - \la_{0,L_n}\|_\infty \leqa
		{E_{\la_0} \left[\ell_\infty(\la_{L_n}^*,\la_{0,L_n})\right] }
		\leqa \sqrt{\frac{L_n 2^{L_n}}{n}},
	\end{equation*}
	where $\ell_\infty(f,g)$ is defined in \eqref{ellinfty}.
\end{lem}
\begin{proof}
	Arguing as in Lemma 7 of \cite{ic14}, we find, for any $t > 0$:
	\begin{align*}
		E_{\la_0}  \| &\la_{L_n}^* - \la_{0,L_n}\|_\infty 
		\leqa {E_{\la_0} \ell_\infty(\la_{L_n}^*,\la_{0,L_n})} \\
		& \leqa  \frac{1}{\sqrt{n}}\sum_{l \leq L_n} \frac{2^{l/2}}{t} \log \sum_{k=0}^{2^l - 1} E_{\la_0}\left[ e^{tW_n(\psi_{n,lk})} +  e^{-tW_n(\psi_{n,lk})} \right].
	\end{align*}
	We now study $E_{\la_0} e^{sW_n(\psi_{n,lk})}$ for any real $s$ and some particular $l, k$. We introduce the notation $H_{n,lk}(X_i) = \delta_i\psi_{n,lk}(Y_i) - \La_0\psi_{n,lk}(Y_i)$, where $X_i$ is short for $(\delta_i, Y_i)$, and remark that $W_n(\psi_{n,lk}) = \frac{1}{\sqrt{n}}\sum_{i=1}^n H_{n,lk}(X_i)$. We also remark that, by construction, $E_{\la_0} H_{n,lk}(X_1) = 0$. We rewrite, by independence of the $X_i$:
	\begin{align*}
		E_{\la_0}  e^{sW_n(\psi_{n,lk})} = e^{n\log \E [e^{\frac{s}{\sqrt{n}} H_{n,lk}(X_1)}]},
	\end{align*}
	and bound, using $\E_{\la_0} H_{n,lk}(X_1) = 0$,
	\begin{align*}
		E \left[ e^{\frac{s}{\sqrt{n}} H_{n,lk}(X_1)} \right]
		&=E \left[ \sum_{k \geq 0} \left( \frac{s H_{n,lk}(X1)}{\sqrt{n}}\right)^k \frac{1}{k!}\right]
		= 1  + E \left[ \sum_{k \geq 2} \left( \frac{s H_{n,lk}(X1)}{\sqrt{n}}\right)^k \frac{1}{k!}\right]\\
		&\leq 1 + \sum_{k \geq 2} \left(\frac{|t|\|H_{n,lk}\|_\infty}{\sqrt{n}}\right)^{k-2} \frac{s^2 E[H_{n,lk}(X_1)^2]}{nk!}\\
		&\leq 1 + \frac{s^2}{2n}E[H_{n,lk}(X_1)^2]  e^{\frac{|s| \|H_{n,lk}\|_\infty}{\sqrt{n}}}.
	\end{align*}
	Returning to the first display, we use the just obtained bound with $t = \sqrt{l}$ and $s = t$ or $s = -t$. Noting that $E[H_{n,lk}(X_1)^2] = \|\psi_{n,lk}\|_L^2$ which is bounded by a constant by Lemma \ref{lem:psi}, and bounding $\|H_{n,lk}\|_\infty \leq (1 + \|\La_0\|_\infty)\|\psi_{n,lk}\|_\infty$ combined with  $\|\psi_{n,lk}\|_\infty \lesssim L_n 2^{L_n/2}$ as follows from Lemma \ref{lem:psi}, one obtains
	\[
	E_{\la_0} \|\la_{L_n}^* - \la_{0,L_n}\|_\infty \leq 
	\frac{1}{\sqrt{n}}\sum_{l \leq L_n} \frac{2^{l/2}}{t} \log \left( 2^l 2 e^{Ct}\right),
	\]
	which is bounded from above, if one chooses $t=t_l=l$, by a constant times $\sum_{l\le L_n}2^{l/2}\sqrt{l}/\rn\leqa (L_n2^{L_n}/n)^{1/2}$ and 
	the proof is finished. 
\end{proof}

\subsection{$\ $ Bounds on LAN remainders}

Let $\cH$ be the class of (possibly $n$--dependent) functions defined by, for some constant $D>0$ independent of $n$ and sequences $(\mu_n)$, $(v_n)$ of positive real numbers,
\begin{align}
	\cH_n & =\{h\in L^\infty[0,1],\quad \|h\|_\infty\le \mu_n,\, \|h\|_2\le D \}, 
	\label{hach}\\
	\cL_n & = \{\la\in L^\infty[0,1],\quad \la\ge 0,\, \|\la-\la_0\|_\infty\le v_n\}. \label{elln}
\end{align}
Further consider the sets of functions, recalling the notation $(\La f)(\cdot)=\int_0^\cdot fd\La=\int_0^\cdot f\la$, for $t>0$,
\begin{align}
	\cF_n & = \{ (\La_0-\La)h,\quad h\in\cH_n, \la\in\cL_n \}, \label{fn}\\
	\cG_n & =\left\{\rn\La\left( e^{- \tfrac{t}{\rn}h} - 1 
	+ \frac{t}{\rn} h \right),\quad h\in\cH_n, \la\in\cL_n \right\}.\label{gn}
\end{align} 

\begin{lem} \label{lem:rn1}
	Let $\cF_n, \cG_n$ be defined as in \eqref{fn}--\eqref{gn}, where $\cH_n, \cL_n$ are as in \eqref{hach}--\eqref{elln}, for sequences $(\mu_n)$, $(v_n)$ of positive real numbers and $t>0$. Suppose $|t|\mu_n/\rn\le d$ for some $d>0$. Then for universal constants $C_1, C_2>0$,
	\begin{align*}
		E^*_{\la_0}\left[\sup_{f_n\in\cF_n} |\mathbb{G}_n f_n| \right] & \le C_1 v_n, \\
		E^*_{\la_0}\left[\sup_{g_n\in\cG_n} |\mathbb{G}_n g_n| \right] & \le C_2 \frac{t^2}{\rn}(1+v_n).
	\end{align*} 
\end{lem}
\begin{proof}
	First, one observes that any $f_n\in\cF_n$ and $g_n\in\cG_n$ are functions of finite total variation: both are differentiable and for $f_n=(\La_0-\La)h$ (for which $f_n(0)=0$) we have
	\[ \|f_n\|_{BV} = \int_0^1 |(\la-\la_0)h| \le \|\la-\la_0\|_\infty\|h\|_2\le Dv_n,\]
	where we use $\|h\|_1\le \|h\|_2$ (Cauchy--Schwarz). 
	For $g_n$, noting that $|t|\|h\|_\infty/\rn$ is bounded by assumption and using the inequality $|e^x-1-x|\le Cx^2$ for $x\le 1$, one obtains that the total variation of $g_n$ is bounded by
	\begin{align*}
		\|g_n\|_{BV}=\int_0^1 |g_n'| & \le C\frac{t^2}{\rn}\int_0^1 \la h^2 
		\le C\frac{t^2}{\rn} \left[ \int_0^1 (\la-\la_0)h^2  + \int_0^1 \la_0 h^2\right]\\
		& \le C\frac{t^2}{\rn} \|h\|_2^2\left[\|\la-\la_0\|_\infty +1\right]
		\le \frac{t^2}{\rn} D(1+v_n),
	\end{align*}
	where one uses the definitions of $\cH_n, \cL_n$. The results follow by applying Lemma \ref{lemep} to both empirical processes at stake.
\end{proof}

\begin{lem}\label{lem:rn2}
	Let $\cH_n, \cL_n$ be as in \eqref{hach}--\eqref{elln}, for sequences $(\mu_n)$, $(v_n)$ of positive real numbers and $t>0$. Suppose $|t|\mu_n/\rn\le d$ for some $d>0$. Then
	\begin{align*}
		\sup_{h\in\cH_n,\, \la\in\cL_n} & \left|
		n \La_0\left\{M_0\left(e^{r-r_0}\left(e^{- \frac{t}{\sqrt{n}}h} - 1 + \frac{t}{\sqrt{n}}h\right)  - \frac{1}{2}\frac{t^2}{n}h^2  \right)  \right\} 
		\right| \\
		&\qquad =O\left(t^2\left\{v_n + |t| \frac{1+v_n}{\rn} \right\}\right). 
	\end{align*}
\end{lem}
\begin{proof}
	As $|t|\|h\|_\infty/\rn$ is bounded by assumption, one 
	expands $e^{-th/\rn}$ with third-order remainder term $R_h$: we have $|R_h| \leqa t^3n^{-3/2} e^{t\|h\|_\infty/\rn}\leqa t^3n^{-3/2}$. 
	Further writing $e^{r-r_0}=\la/\la_0$ gives
	\begin{align*}
		n \La_0&\left\{M_0\left(e^{r-r_0}\left(e^{- \frac{t}{\sqrt{n}}h} - 1 + \frac{t}{\sqrt{n}}h\right)  - \frac{1}{2}\frac{t^2}{n}h^2  \right)  \right\}\\
		& \le n\La_0\left\{M_0\left(\frac{|\la-\la_0|}{\la_0}\frac{t^2}{2n}h^2 \right)
		+e^{r-r_0}|R_h| \right\}\\
		& \leqa t^2\|\la-\la_0\|_\infty \|h\|_2^2 +   t^3\frac{\|\La\|_\infty}{\rn} e^{t\|h\|_\infty/\rn} \\
		& \leqa t^2 \left[ v_n + |t|(v_n+1)/\rn \right], 
	\end{align*}
	where for the last inequality we use that $\La(1)=\int_0^1 \la(u)du$ and the bound $\|\La\|_\infty =\La(1) \le \|\la-\la_0\|_\infty + \|\la_0\|_\infty \leqa (v_n+1)$. 
\end{proof}

The following variant is helpful if one does not wish to use any supremum norm bound on $\la-\la_0$, and instead base arguments only on the $\|\cdot\|_1$--norm. Lemma \ref{lem:rn3} is used for Theorem \ref{thm:bvmlin}, which involves the set $A_n$ of \ref{c:l1}, indeed based on the $\|\cdot\|_1$--norm only. Note, on the other hand, that this has consequences on the obtained bounds, whose dependence in $\mu_n$ is not as good as those above for $\mu_n$ growing fast to infinity, so could not be used for later results such as Theorem \ref{thm:sn} and \ref{thm:bvm}, whose proofs require  $h$'s with rapidly growing $\|\cdot\|_\infty$--norm.   Let us first update slightly the definition of $\cL_n$ as, for some sequence $(\veps_n)$,
\begin{align}
	\cL^1_n & = \{\la\in L^\infty[0,1],\quad \la\ge 0,\, \|\la-\la_0\|_1\le \veps_n\}. \label{ellnb}
\end{align}
Further set
\begin{align}
	\cF^1_n & = \{ (\La_0-\La)h,\quad h\in\cH_n, \la\in\cL_n^1 \}, \label{fnb}\\
	\cG^1_n & =\left\{\rn\La\left( e^{- \tfrac{t}{\rn}h} - 1 
	+ \frac{t}{\rn} h \right),\quad h\in\cH_n, \la\in\cL_n^1 \right\}.\label{gnb}
\end{align} 
\begin{lem} \label{lem:rn3}
	Let $\cF^1_n, \cG^1_n$ be defined as in \eqref{fnb}--\eqref{gnb}, where $\cH_n, \cL^1_n$ are as in \eqref{hach}--\eqref{ellnb}, for a sequence  $(\veps_n)$ of positive real numbers and $t$ such that $|t|\mu_n/\rn\le d$ for some $d>0$. Then for universal constants $C_1, C_2>0$,
	\begin{align*}
		E^*_{\la_0}\left[\sup_{f_n\in \cF^1_n} |\mathbb{G}_n f_n| \right] & \le C_1 \veps_n\mu_n, \\
		E^*_{\la_0}\left[\sup_{g_n\in \cG^1_n} |\mathbb{G}_n g_n| \right] & \le C_2 \frac{t^2}{\rn}(1+\veps_n\mu_n^2).
	\end{align*} 
	Further, under the same notation and conditions,
	\begin{align*}
		\sup_{h\in\cH_n,\, \la\in\cL_n^1} & \left|
		n \La_0\left\{M_0\left(e^{r-r_0}\left(e^{- \frac{t}{\sqrt{n}}h} - 1 + \frac{t}{\sqrt{n}}h\right)  - \frac{1}{2}\frac{t^2}{n}h^2  \right)  \right\} 
		\right| \\
		&\qquad =O\left(t^2\left\{\veps_n\mu_n^2 + |t| \frac{1+\veps_n}{\rn} \right\}\right). 
	\end{align*}
\end{lem}
\begin{proof}
	The proof is nearly identical to that of Lemmas \ref{lem:rn1}--\ref{lem:rn2}, except one now bounds $\int |\la-\la_0||h|\le \|\la-\la_0\|_1\|h\|_\infty$ and, for some uniformly bounded function $a$, 
	\[ \int_0^1 |\la-\la_0|(u)h(u)^2a(u)du \le C  \|\la-\la_0\|_1\|h\|_\infty^2,\]
	instead of bounding from above by $C\|\la-\la_0\|_\infty\|h\|_2^2$ as before. Finally to bound $\|\La\|_\infty\l=\La(1)$, one uses $\La(1)\le \La_0(1)+\int_0^1|\la-\la_0|\le \veps_n + \|\La_0\|_\infty\leqa 1+\veps_n$.
\end{proof}

\section{$\ $ Additional simulation and data application results}\label{sec:suppsim}

\subsection{$\ $ Simulation results for the cumulative hazard\label{sec:simhaz}} 
Table \ref{table:cumhaz} shows the coverage results for the credible bands for the cumulative hazard $\La$. In all scenarios, coverage is as expected from  Corollary \ref{cor:bands}. This illustrates that the Bayesian credible bands are suitable for uncertainty quantification in practice.

\begin{table}[h]
	\caption{Coverage of the credible bands for the cumulative hazard, using the dependent and independent Gamma priors. The parameter $\gamma$ is that of \eqref{eln}, so $\gamma = 1/2$ corresponds to $K = \left\lceil \left(n/\log{n}\right)^{1/2} \right\rceil$ intervals and $\gamma = 1$ to  $K = \left\lceil \left(n/\log{n}\right)^{1/3} \right\rceil$ intervals.}
	\label{table:cumhaz}
	\begin{tabular}{l|rr|rr}
		& \multicolumn{2}{c|}{$\gamma = 1/2$} & \multicolumn{2}{c|}{$\gamma = 1$}  \\
		& dep.	& indep.	& dep.	& indep.\\
		\hline
		\multicolumn{1}{l}{ \textbf{Smooth hazard}} \\
		\hline
		$n = 200$, adm. + uniform & 0.96 & 0.97 &  0.95 & 0.98 \\
		$n = 2000$, adm. + uniform & 0.96 & 0.96 & 0.95  & 0.97 \\
		$n = 200$, adm. & 0.96  &  0.96 & 0.96  & 0.97 \\
		$n = 2000$, adm. & 0.95 & 0.95  & 0.95  & 0.95 \\
		\hline
		\multicolumn{1}{l}{ \textbf{Piecewise linear hazard}} \\
		\hline
		$n = 200$, adm. + uniform & 0.96  & 0.98 & 0.95  & 0.98  \\
		$n = 2000$, adm. + uniform & 0.96  & 0.96  & 0.94  & 0.98  \\
		$n = 200$, adm. & 0.95  & 0.95 & 0.96  & 0.95 \\
		$n = 2000$, adm. & 0.95 & 0.96  & 0.95 & 0.94 \\
	\end{tabular}
\end{table}

\subsection{$\ $ Data application: results for the hazard and cumulative hazard\label{sec:datahaz}}

For the hazard, we compare the posterior mean to the kernel-based methods of Mueller and Wang \cite{Mueller1994} and Cao and Lopez-de-Ullibarri \cite{Cao2007}, in Figure \ref{fig:cancer_haz}.

We compare the posterior mean and credible band for the cumulative hazard to the Nelson-Aalen estimator with its pointwise confidence intervals, in Figure \ref{fig:cancer_haz} (although we draw it here for the sake of comparison, let us stress again that, contrary to the credible band, the corresponding `patched' band obtained from putting together the confidence intervals typically does {\em not} cover with the desired confidence level).

The posterior mean of the hazard resembles the other two estimates, and provides an estimate for a longer timeframe than the two kernel-based methods, compensating for the low number of events in the final two intervals by exploiting the dependence induced through the prior. 

In this example the credible band for the cumulative hazard is quite wide, more or less equal in width to the width of the final pointwise interval of the Nelson-Aalen estimator. This is not surprising for two reasons: (i) the pointwise intervals have very different guarantees than the credible band; (ii)  the credible band is constructed to have a fixed radius and needs to cover just as well near the end of follow-up as in the beginning, leading to a possibly too conservative width at the start of follow-up. Although beyond the scope of the present contribution, it would also be interesting to investigate constructions of credible bands with possibly time-varying radius.

\begin{figure}[h]
	\caption{Posterior mean of the hazard (solid piecewise constant function) and the estimators  of Cao and Lopez-de-Ullibarri (dotted) and of Mueller and Wang (dashed).  }
	\label{fig:cancer_haz}
	\includegraphics[width=0.8\textwidth]{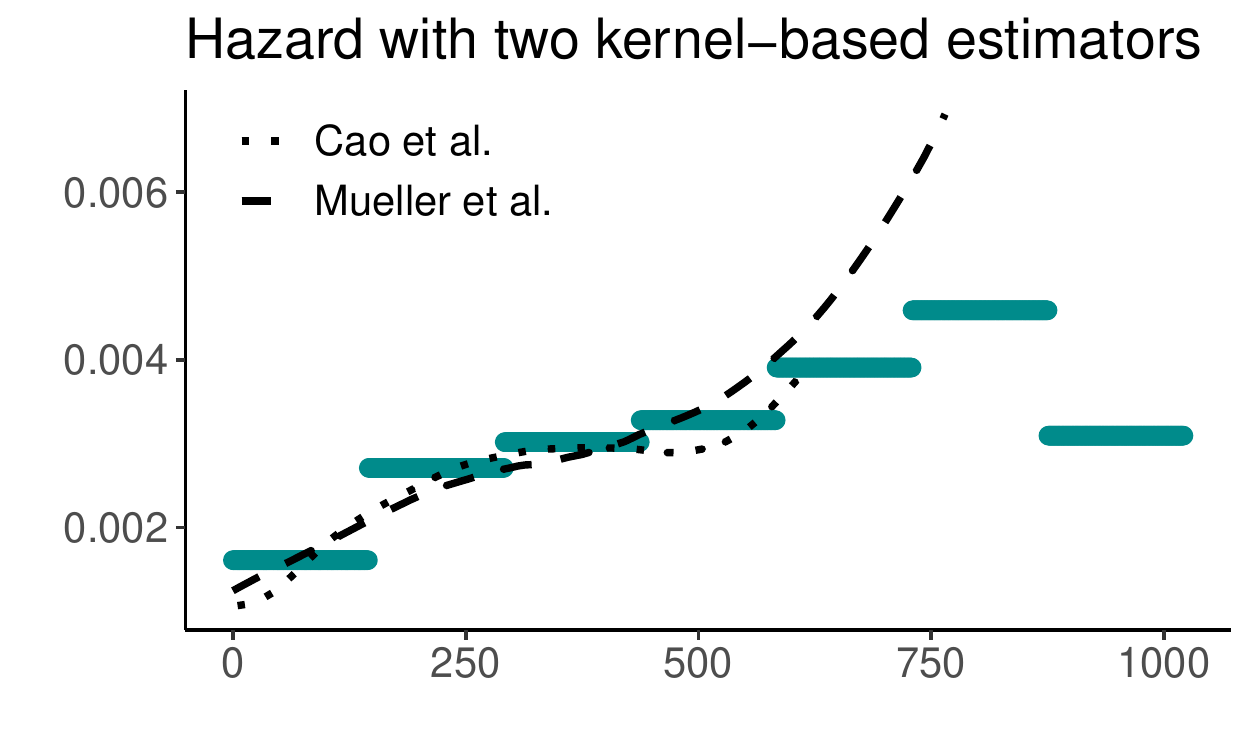}
\end{figure}

\begin{figure}[h]
	\caption{Posterior mean of the cumulative hazard (solid) with credible band (shaded area) and the Nelson-Aalen estimator (dashed) with its pointwise confidence intervals (dashed). Unlike the credible band, the collated pointwise confidence intervals do not form a confidence band.  }
	\label{fig:cancer_cumhaz}
	\includegraphics[width=0.8\textwidth]{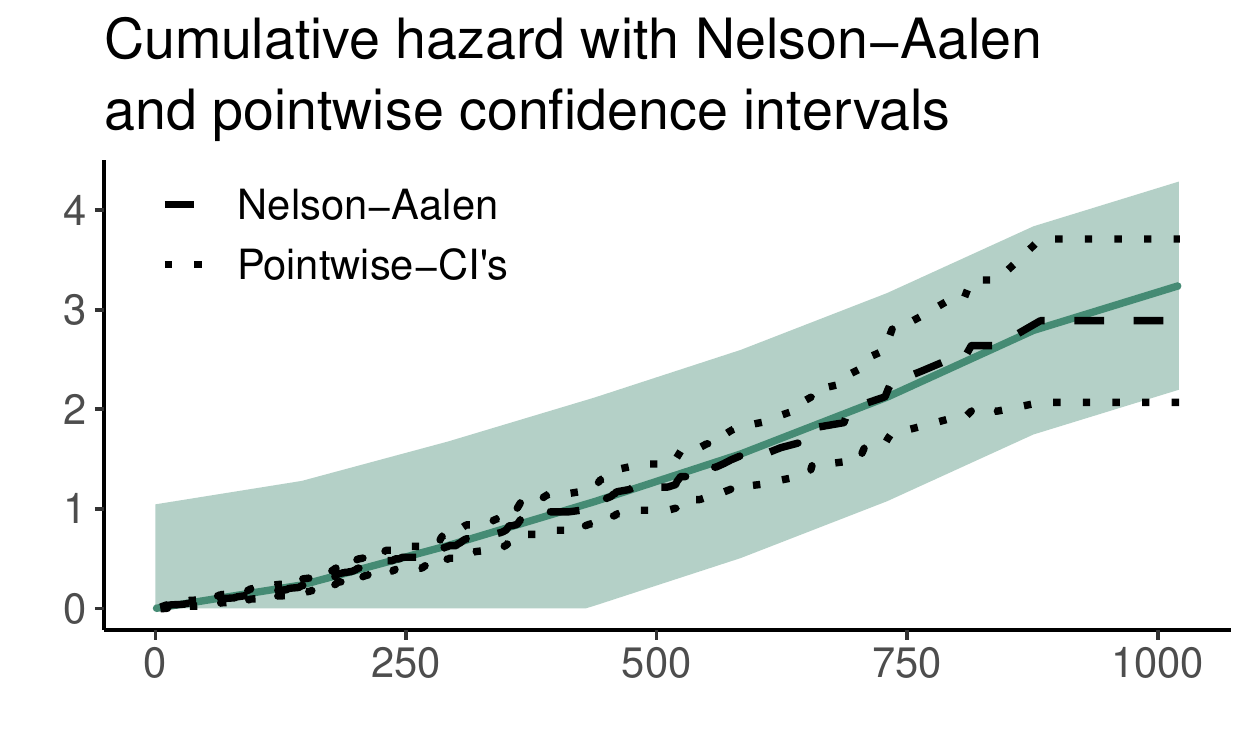}
\end{figure}

\subsection{$\ $ The sampler, independent and dependent Gamma priors\label{sec:sampler}} 
To sample from the posterior,  we derive a Gibbs sampler with Metropolis-Hastings steps within Gibbs. Let $K$ be the number of intervals and denote the draw of the hazard from iteration $j$ by $(\lambda_1^{(j)}, \ldots, \lambda_K^{(j)})$. We augment the data $\{(y_i, \delta_i)\}_{i=1}^n$ as in \cite{Holford1980, Laird1981}: $y_i$ is represented as $(y_{1i}, y_{2i}, \ldots, y_{Ki})$, where $y_{ki}$ is the total time individual $i$ was under follow-up during interval $k$, and $\delta_i$ is represented as $(\delta_{1i}, \delta_{2i}, \ldots, \delta_{Ki})$, where $\delta_{ki}$ is equal to one if individual $i$ experiences the event in interval $k$, and equal to zero otherwise.

For the sampler, we require for each interval $k$ the number of events in the interval,  $d_k = \sum_{i=1}^n \delta_{ki}$  and the total exposure in the interval, $T_k = \sum_{i=1}^n y_{ki}$. We initalize at $\la_k^{(1)} = \frac{d_k}{T_k+1} + \varepsilon$ for some small $\epsilon > 0$. At each iteration $j$ and for each interval $k$, we will draw a proposal $\la_k^{prop}$, compute the acceptance ratio $A$, then draw $u \sim \operatorname{Unif}[0,1]$ and accept the proposal if $u < A$, otherwise we reject the proposal.

For the dependent prior, at iteration $j$, for the first interval we draw a proposal $\la_1^{prop} \sim \operatorname{Gamma}(d_1 + \alpha_0 - \alpha, \beta_0 + T_1)$ and for the other intervals we draw a proposal $$\la_k^{prop} \sim  \operatorname{Gamma}\left(d_k + \varepsilon, \tfrac{\alpha}{\la_{k-1}^{(j)}} + T_k\right)$$ for some small $\varepsilon > 0$. This leads to 
$$ A = \exp\left(\alpha \la_2^{(j-1)}\left(\tfrac{1}{\la_1^{(j-1)}} - \tfrac{1}{\la_1^{prop}}\right)\right)$$ for the first intervals, and for $k = 2, \ldots, K-1$ to  
$$A = \exp\left(-\alpha \la_{k+1}^{(j-1)}\left( \tfrac{1}{\la_k^{(j-1)}} - \tfrac{1}{\la_k^{prop}}\right)\right)\left(\frac{\la_k^{(j-1)}}{\la_k^{prop}}\right)^\varepsilon.$$ 
For the final interval, the marginal conditional posterior distribution is a $ \operatorname{Gamma}(d_K + \alpha, \alpha  \la_{K-1}^{(j)} + T_K)$, from which we can sample directly.

For the independent prior, we can sample directly from the posterior, drawing from a $ \operatorname{Gamma}(d_k + \alpha, T_k + \beta)$ for each interval $k$.

\bibliographystyle{abbrv}
\bibliography{survival} 

\end{document}